%% file: dissertation.tex
\documentclass[letterpaper, 12pt, oneside]{book}

    \usepackage[left=108pt, right=74pt, top=108pt, bottom=81pt, dvips]{geometry}
    \usepackage{epsfig}
    \usepackage{amsmath}
    \usepackage{amssymb}
    \usepackage{mathrsfs}
    \usepackage{fancyhdr} 
    \renewcommand{\bibname}{References}

    \usepackage{amsmath,amsthm, amsfonts,amssymb} 
    \usepackage{setspace} 
    \usepackage[linktocpage,bookmarksopen,bookmarksnumbered,
                pdftitle={UC Davis Ph.D. Doctoral Thesis},
                pdfauthor={Department of Mathematics},%
                pdfsubject={UC Davis Ph.D. Doctoral Thesis},%
                pdfkeywords={UC Davis Ph.D. Doctoral Thesis}]{hyperref}
    
    \DeclareMathOperator*{\supp}{supp}

\DeclareMathOperator*{\range}{range}
\DeclareMathOperator*{\Id}{Id}
\DeclareMathOperator*{\Idg}{Id_\Gamma}

\DeclareMathOperator*{\argmin}{arg min}
\newcommand{\smnorm}[2]{{\bigl\Vert {#2} \bigr\Vert}_{#1}}

\newcommand{\term}{\emph}
\newcommand{\cnst}[1]{\mathrm{#1}}

\def \R {\mathbb{R}}
\def \C {\mathbb{C}}

\newcommand{\abs}[1]{\left\vert {#1} \right\vert}
\newcommand{\mtx}[1]{\bf{#1}}

\def \e {\varepsilon}
\newcommand{\Fee}{\Phi}
\newcommand{\adj}{*}
\newcommand{\restrict}[1]{\vert_{#1}}

\def \L {\Lambda}

\def \< {\langle}
\def \> {\rangle}
\def \^ {\widehat}

\def \conv {{\rm conv}}

\def \supp {{\rm supp}}

    \theoremstyle{plain}
\newtheorem{theorem}{Theorem}[section]
\newtheorem{proposition}[theorem]{Proposition}
\newtheorem{corollary}[theorem]{Corollary}
\newtheorem{lemma}[theorem]{Lemma}

\newtheorem{remark}[theorem]{Remark}
\newlength{\algorithmwidth}
\newcommand{\psinv}{\dagger}
    \theoremstyle{remark}

    \numberwithin{equation}{section}
    \numberwithin{figure}{section}
    \newcommand{\enorm}[1]{\norm{#1}_2}
\newcommand{\enormsq}[1]{\enorm{#1}^2}

\newcommand{\norm}[1]{\left\Vert {#1} \right\Vert}

\newcommand{\pnorm}[2]{\norm{#2}_{#1}}
\newcommand{\infnorm}[1]{\norm{#1}_{\infty}}
\newcommand{\pr}[2]{\langle {#1} , {#2} \rangle}
    \newcommand{\bigO}{\mathrm{O}}
    \newcommand{\coll}[1]{\mathscr{#1}}
    
\newcommand{\Cspace}[1]{\mathbb{C}^{#1}}
    
    \newcommand{\n}{\vspace{12pt}} 
            
            \newcommand{\abssq}[1]{{\abs{#1}}^2}                       
    \newcommand{\vct}{}
    \newcommand{\subjto}{\quad\text{subject to}\quad}
    \newcommand{\defby}{\overset{\mathrm{\scriptscriptstyle{def}}}{=}}
    \newcommand{\eps}{\varepsilon}
    \newcommand{\newchapter}[3] 
	{                           
        \chapter[#2]{#3}
        \chaptermark{#1}
        \thispagestyle{myheadings}
	}

\begin{document}
    \bibliographystyle{plain}

    \pagenumbering{roman}
    \pagestyle{plain}

    %
    %

    \singlespacing

    ~\vspace{-0.75in} 
    \begin{center}

        \begin{huge}
            Topics in Compressed Sensing
        \end{huge}\\\n
        By\\\n
        {\sc Deanna Needell}\\
        B.S. (University of Nevada, Reno) 2003\\
        M.A. (University of California, Davis) 2005\\\n
        DISSERTATION\\\n
        Submitted in partial satisfaction of the requirements for the degree of\\\n
        DOCTOR OF PHILOSOPHY\\\n
        in\\\n
        MATHEMATICS\\\n
        in the\\\n
        OFFICE OF GRADUATE STUDIES\\\n
        of the\\\n
        UNIVERSITY OF CALIFORNIA\\\n
        DAVIS\\\n\n
        
        Approved:\\\n\n
        
%
%

\underline{\quad\quad  Roman Vershynin  \quad\quad}\\
        ~Roman Vershynin (Co-Chair)\\\n\n
        
        \underline{\quad\quad  Thomas Strohmer  \quad\quad}\\
        ~Thomas Strohmer (Co-Chair)\\\n\n
        
        \underline{\quad\quad  Jes$\acute{\text u}$s DeLoera  \quad\quad}\\
        ~Jes$\acute{\text u}$s DeLoera\\

        \vfill
        
        Committee in Charge\\
        ~2009

    \end{center}

    \newpage

    %
    %
        

    ~\\[7.75in] 
    \centerline{
                \copyright\ Deanna Needell,
                            2009. All rights reserved.
               }
    \thispagestyle{empty}
    \addtocounter{page}{-1}

    \newpage

    %
    %
    
    \centerline{I dedicate this dissertation to all my loved ones, those with me today and those who have passed on.}
    
    \newpage

    \doublespacing

    %
    %
    
    \tableofcontents
    
    \newpage

    %
    %
    
    %

    ~\vspace{-1in} 
    \begin{flushright}
        \singlespacing
    \end{flushright}

    \centerline{\large Topics in Compressed Sensing}
    
    \centerline{\textbf{\underline{Abstract}}}
    
        \input{abstract}
    
    \newpage

    %
    %

    \chapter*{\vspace{-1.5in}Acknowledgments}

        \input{acknowledgements}
    
    \newpage

    %
    %

    \pagestyle{fancy}
    \pagenumbering{arabic}
       
    %
    %

    \newchapter{Introduction}{Introduction}{Introduction}
    \label{sec:Intro}

        
        %
    
        \section[Overview]{Overview}
        \label{sec:Intro:Overview}
        
            \input{overview}
        
        \section[Applications]{Applications}
        \label{sec:Intro:Applications}
            \input{applications}

    %
    %

    \newchapter{Approaches}{Major Algorithmic Approaches}{Major Algorithmic Approaches}
    \label{sec:Approaches}
				Compressed Sensing has provided many methods to solve the sparse recovery problem and thus its applications. There are two major algorithmic approaches to this problem. The first relies on an optimization problem which can be solved using linear programming, while the second approach takes advantage of the speed of greedy algorithms. Both approaches have advantages and disadvantages which are discussed throughout this chapter along with descriptions of the algorithms themselves. First we discuss Basis Pursuit, a method that utilizes a linear program to solve the sparse recovery problem. 

        
    
        \section[Basis Pursuit]{Basis Pursuit}
        \label{sec:Approaches:Basis}
        \input{basispursuit}

        \section[Greedy Methods]{Greedy Methods}
        \label{sec:Appraoches:Greedy}
        \input{greedymethods}



    %
    %

    \newchapter{Contributions}{Contributions}{Contributions}
    \label{sec:New}

        
    
        \section[Regularized Orthogonal Matching Pursuit]{Regularized Orthogonal Matching Pursuit}
        \label{sec:New:Regularized}

\input{romp}
    
        \section[Compressive Sampling Matching Pursuit]{Compressive Sampling Matching Pursuit}
        \label{sec:New:Compressive}
        \input{cosamp}

        \section[Reweighted L1-Minimization]{Reweighted L1-Minimization}
        \label{sec:New:Reweighted}
        \input{rwl1}
        
			  \section[Randomized Kaczmarz]{Randomized Kaczmarz}
        \label{sec:New:Kaczmarz}
        \input{rk}

            %
            
%


    %
    %

    \newchapter{Summary}{Summary}{Summary}
    \label{sec:summary}
    \input{summary}

        


    %
    %

    %
    %
    


 
 \appendix
 \newchapter{Matlab Code}{Matlab Code}{Matlab Code}\label{app:code}
 This section contains original Matlab code used to produce the figures contained within this work.
 \section{Basis Pursuit}\label{app:code:BPcode}
 \input{BPcode}
 
  \section{Orthogonal Matching Pursuit}\label{app:code:OMPcode}
 \input{OMPcode}
 
   \section{Regularized Orthogonal Matching Pursuit}\label{app:code:ROMPcode}
 \input{ROMPcode}
 
    \section{Compressive Sampling Matching Pursuit}\label{app:code:Cosampcode}
 \input{Cosampcode}
 
     \section{Reweighted L1 Minimization}\label{app:code:RWL1}
 \input{RWL1code}

      \section{Randomized Kaczmarz}\label{app:code:RKcode}
 \input{RKcode}
 
%
%
 
        \thispagestyle{myheadings}
        \addcontentsline{toc}{chapter}{\bibname}
 \bibliography{dissertation}

\end{document}

%% file: abstract.tex
Compressed sensing has a wide range of applications that include error correction, imaging, radar and many more. Given a sparse signal in a high dimensional space, one wishes to reconstruct that signal accurately and efficiently from a number of linear measurements much less than its actual dimension. Although in theory it is clear that this is possible, the difficulty lies in the construction of algorithms that perform the recovery efficiently, as well as determining which kind of linear measurements allow for the reconstruction. 
There have been two distinct major approaches to sparse recovery that each present different benefits and shortcomings. The first, $\ell_1$-minimization methods such as Basis Pursuit, use a linear optimization problem to recover the signal. This method provides strong guarantees and stability, but relies on Linear Programming, whose methods do not yet have strong polynomially bounded runtimes. The second approach uses greedy methods that compute the support of the signal iteratively. These methods are usually much faster than Basis Pursuit, but until recently had not been able to provide the same guarantees. This gap between the two approaches was bridged when we developed and analyzed the greedy algorithm Regularized Orthogonal Matching Pursuit (ROMP). ROMP provides similar guarantees to Basis Pursuit as well as the speed of a greedy algorithm. Our more recent algorithm Compressive Sampling Matching Pursuit (CoSaMP) improves upon these guarantees, and is optimal in every important aspect. Recent work has also been done on a reweighted version of the $\ell_1$-minimization method that improves upon the original version in the recovery error and measurement requirements. These algorithms are discussed in detail, as well as previous work that serves as a foundation for sparse signal recovery.  

%% file: acknowledgements.tex
There were many people who helped me in my graduate career, and I would like to take this opportunity to thank them.

I am extremely grateful to my adviser, Prof. Roman Vershynin. I have been very fortunate to have an adviser who gave me the freedom to explore my area on my own, and also the incredible guidance to steer me down a different path when I reached a dead end. Roman taught me how to think about mathematics in a way that challenges me and helps me succeed. He is an extraordinary teacher, both in the classroom and as an adviser. 

My co-adviser, Prof. Thomas Strohmer, has been very helpful throughout my graduate career, and especially in my last year. He always made time for me to talk about ideas and review papers, despite an overwhelming schedule. His insightful comments and constructive criticisms have been invaluable to me. 

Prof. Jes$\acute{\text u}$s DeLoera has been very inspiring to me throughout my time at Davis. His instruction style has had a great impact on the way I teach, and I strive to make my classroom environment as enjoyable as he always makes his. 

Prof. Janko Gravner is an amazing teacher, and challenges his students to meet high standards. The challenges he presented me have given me a respect and fondness for the subject that still helps in my research today. 

I would also like to thank all the staff at UC Davis, and especially Celia Davis, Perry Gee, and Jessica Potts for all their help and patience. They had a huge impact on my success at Davis, and I am so grateful for them.

My friends both in Davis and afar have helped to keep me sane throughout the graduate process. Our time spent together has been a blessing, and I thank them for all they have done for me. I especially thank Blake for all his love and support. His unwaivering confidence in my abilities has made me a stronger person.

My very special thanks to my family. My father, George, has always provided me with quiet encouragement and strength. My mother, Debbie, has shown me how to persevere through life challenges and has always been there for me. My sister, Crystal, has helped me keep the joy in my life, and given me a wonderful nephew, Jack. My grandmother and late grandfather were a constant support and helped teach me the rewards of hard work.

Finally, thank you to the NSF for the financial support that helped fund parts of the research in this dissertation.

%% file: overview.tex
\subsection{Main Idea}
The phrase \textit{compressed sensing} refers to the problem of realizing a sparse input $\vct{x}$ using few linear measurements that possess some incoherence properties.  The field originated recently from an unfavorable opinion about the current signal compression methodology. The conventional scheme in signal processing, acquiring the entire signal and then compressing it, was questioned by Donoho \cite{DE03:Optimally}. Indeed, this technique uses tremendous resources to acquire often very large signals, just to throw away information during compression. The natural question then is whether we can combine these two processes, and directly sense the signal or its essential parts using few linear measurements. Recent work in compressed sensing has answered this question in positive, and the field continues to rapidly produce encouraging results.

The key objective in compressed sensing (also referred to as sparse signal recovery or compressive sampling) is to reconstruct a signal accurately and efficiently from a set of few non-adaptive linear measurements. Signals in this context are vectors, many of which in the applications will represent images. Of course, linear algebra easily shows that in general it is not possible to reconstruct an arbitrary signal from an incomplete set of linear measurements. Thus one must restrict the domain in which the signals belong. To this end, we consider \textit{sparse} signals, those with few non-zero coordinates.
It is now known that many signals such as real-world images or audio signals are sparse either in this sense, or with respect to a different basis. 

Since sparse signals lie in a lower dimensional space, one would think intuitively that they may be represented by few linear measurements. This is indeed correct, but the difficulty is determining in which lower dimensional subspace such a signal lies. That is, we may know that the signal has few non-zero coordinates, but we do not know which coordinates those are. It is thus clear that we may not reconstruct such signals using a simple linear operator, and that the recovery requires more sophisticated techniques. The compressed sensing field has provided many recovery algorithms, most with provable as well as empirical results.

There are several important traits that an optimal recovery algorithm must possess. The algorithm needs to be fast, so that it can efficiently recover signals in practice. Of course, minimal storage requirements as well would be ideal. The algorithm should provide uniform guarantees, meaning that given a specific method of acquiring linear measurements, the algorithm recovers all sparse signals (possibly with high probability). Ideally, the algorithm would require as few linear measurements as possible. Linear algebra shows us that if a signal has $s$ non-zero coordinates, then recovery is theoretically possible with just $2s$ measurements. However, recovery using only this property would require searching through the exponentially large set of all possible lower dimensional subspaces, and so in practice is not numerically feasible. Thus in the more realistic setting, we may need slightly more measurements. Finally, we wish our ideal recovery algorithm to be stable. This means that if the signal or its measurements are perturbed slightly, then the recovery should still be approximately accurate. This is essential, since in practice we often encounter not only noisy signals or measurements, but also signals that are not exactly sparse, but close to being sparse. For example, \textit{compressible} signals are those whose coefficients decay according to some power law. Many signals in practice are compressible, such as smooth signals or signals whose variations are bounded. 

\subsection{Problem Formulation}

Since we will be looking at the reconstruction of sparse vectors, we need a way to quantify the sparsity of a vector. We say that a $d$-dimensional signal $\vct{x}$ is $s$-sparse if it has $s$ or fewer non-zero coordinates:
$$
\vct{x} \in \R^d, \quad \|\vct{x}\|_0 := |\supp(\vct{x})| \leq s \ll d,
$$
where we note that $\|\cdot\|_0$ is a quasi-norm. For $ 1 \leq p < \infty $, we denote by $\|\cdot\|_p$ the usual $p$-norm,
$$
\|\vct{x}\|_p := \big(\sum_{i=1}^{d}|\vct{x}_i|^p\big)^{1/p},
$$
and $\|x\|_\infty = \max|x_i|$. 
In practice, signals are often encountered that are not exactly sparse, but whose coefficients decay rapidly. As mentioned, compressible signals are those satisfying a power law decay:
\begin{equation}\label{comp}
|\vct{x}^*_i| \leq Ri^{(-1/q)},
\end{equation}
where $\vct{x}^*$ is a non-increasing rearrangement of $\vct{x}$, $R$ is some positive constant, and $ 0 < q < 1 $. Note that in particular, sparse signals are compressible.

Sparse recovery algorithms reconstruct sparse signals from a small set of non-adaptive linear measurements. Each measurement can be viewed as an inner product with the signal $\vct{x}\in\R^d$ and some vector $\phi_i\in\R^d$ (or in $\C^d$)\footnote{Although similar results hold for measurements taken over the complex numbers, for simplicity of presentation we only consider real numbers throughout.}. If we collect $m$ measurements in this way, we may then consider the $m \times d$ measurement matrix $\Phi$ whose columns are the vectors $\phi_i$. We can then view the sparse recovery problem as the recovery of the $s$-sparse signal $\vct{x}\in\R^d$ from its measurement vector $\vct{u} = \Phi \vct{x} \in \R^m$. One of the theoretically simplest ways to recover such a vector from its measurements $\vct{u}=\Phi \vct{x}$ is to solve the $\ell_0$-minimization problem 
\begin{equation}\label{eq:ell0}
\min_{\vct{z}\in\R^d} \|\vct{z}\|_0 \quad \text{subject to} \quad \Phi \vct{z} = \vct{u}.
\end{equation}  
If $\vct{x}$ is $s$-sparse and $\Phi$ is one-to-one on all $2s$-sparse vectors, then the minimizer to~\eqref{eq:ell0} must be the signal $\vct{x}$. Indeed, if the minimizer is $\vct{z}$, then since $\vct{x}$ is a feasible solution, $\vct{z}$ must be $s$-sparse as well. Since $\Phi \vct{z} = \vct{u}$, $\vct{z} - \vct{x}$ must be in the kernel of $\Phi$. But $\vct{z} - \vct{x}$ is $2s$-sparse, and since $\Phi$ is one-to-one on all such vectors, we must have that $\vct{z}=\vct{x}$. Thus this $\ell_0$-minimization problem works perfectly in theory. However, it is not numerically feasible and is NP-Hard in general~\cite[Sec.~9.2.2]{M05:Data}.

Fortunately, work in compressed sensing has provided us numerically feasible alternatives to this NP-Hard problem. One major approach, Basis Pursuit, relaxes the $\ell_0$-minimization problem to an $\ell_1$-minimization problem. Basis Pursuit requires a condition on the measurement matrix $\Phi$ stronger than the simple injectivity on sparse vectors, but many kinds of matrices have been shown to satisfy this condition with number of measurements $m = s\log^{O(1)}d$. The $\ell_1$-minimization approach provides uniform guarantees and stability, but relies on methods in Linear Programming. Since there is yet no known strongly polynomial bound, or more importantly, no \textit{linear bound} on the runtime of such methods, these approaches are often not optimally fast.

The other main approach uses greedy algorithms such as Orthogonal Matching Pursuit~\cite{TG07:Signal-Recovery}, Stagewise Orthogonal Matching Pursuit~\cite{DTDS06:Sparse-Solution}, or Iterative Thresholding~\cite{FR07:Iterative, BD08:Iterative}. Many of these methods calculate the support of the signal iteratively. Most of these approaches work for specific measurement matrices with number of measurements $m = O(s\log d)$. Once the support $S$ of the signal has been calculated, the signal $x$ can be reconstructed from its measurements $\vct{u} = \Phi \vct{x} $ as $\vct{x} = (\Phi_S)^\psinv\vct{u}$, where $\Phi_S$ denotes the measurement matrix $\Phi$ restricted to the columns indexed by $S$ and $\psinv$ denotes the pseudoinverse. Greedy approaches are fast, both in theory and practice, but have lacked both stability and uniform guarantees. 
 
There has thus existed a gap between the approaches. The $\ell_1$-minimization methods have provided strong guarantees but have lacked in optimally fast runtimes, while greedy algorithms although fast, have lacked in optimal guarantees. We bridged this gap in the two approaches with our new algorithm Regularized Orthogonal Matching Pursuit (ROMP). ROMP provides similar uniform guarantees and stability results as those of Basis Pursuit, but is an iterative algorithm so also provides the speed of the greedy approach. Our next algorithm, Compressive Sampling Matching Pursuit (CoSaMP) improves upon the results of ROMP, and is the first algorithm in sparse recovery to be provably optimal in every important aspect.

%% file: applications.tex
The sparse recovery problem has applications in many areas, ranging from medicine and coding theory to astronomy and geophysics. Sparse signals arise in practice in very natural ways, so compressed sensing lends itself well to many settings. Three direct applications are error correction, imaging, and radar.

	\subsection[Error Correction]{Error Correction}
          	\label{sec:Intro:Applications:Error}
          	When signals are sent from one party to another, the signal is usually encoded and gathers errors. Because the errors usually occur in few places, sparse recovery can be used to reconstruct the signal from the corrupted encoded data. This error correction problem is a classic problem in coding theory. Coding theory usually assumes the data values live in some finite field, but there are many practical applications for encoding over the continuous reals. In digital communications, for example, one wishes to protect results of onboard computations that are real-valued. These computations are performed by circuits that experience faults caused by effects of the outside world. This and many other examples are difficult real-world problems of error correction.
          	
          	 The error correction problem is formulated as follows. Consider a $m$-dimensional input vector $\vct{f}\in\R^m$ that we wish to transmit reliably to a remote receiver. In coding theory, this vector is referred to as the ``plaintext.'' We transmit the measurements $\vct{z} = A\vct{f}$ (or ``ciphertext'') where $A$ is the $d \times m$ measurement matrix, or the \textit{linear code}. It is clear that if the linear code $A$ has full rank, we can recover the input vector $\vct{f}$ from the ciphertext $\vct{z}$. But as is often the case in practice, we consider the setting where the ciphertext $z$ has been corrupted. We then wish to reconstruct the input signal $f$ from the corrupted measurements $\vct{z}' = A\vct{f} + \vct{e}$ where $\vct{e}\in\R^N$ is the sparse error vector. To realize this in the usual compressed sensing setting, consider a matrix $B$ whose kernel is the range of $A$. Apply $B$ to both sides of the equation $\vct{z}' = A\vct{f} + \vct{e}$ to get $B\vct{z}' = B\vct{e}$. Set $\vct{y} = B\vct{z}'$ and the problem becomes reconstructing the sparse vector $\vct{e}$ from its linear measurements $\vct{y}$. Once we have recovered the error vector $\vct{e}$, we have access to the actual measurements $A\vct{f}$ and since $A$ is full rank can recover the input signal $\vct{f}$. 
          	
        		\subsection[Imaging]{Imaging}
          	\label{sec:Intro:Applications:Imaging}
          	
          	Many images both in nature and otherwise are sparse with respect to some basis. Because of this, many applications in imaging are able to take advantage of the tools provided by Compressed Sensing.  
          	The typical digital camera today records every pixel in an image before compressing that data and storing the compressed image. Due to the use of silicon, everyday digital cameras today can operate in the megapixel range. A natural question asks why we need to acquire this abundance of data, just to throw most of it away immediately. This notion sparked the emerging theory of Compressive Imaging. In this new framework, the idea is to directly acquire random linear measurements of an image without the burdensome step of capturing every pixel initially. 
          	
          	Several issues from this of course arise. The first problem is how to reconstruct the image from its random linear measurements. The solution to this problem is provided by Compressed Sensing. The next issue lies in actually sampling the random linear measurements without first acquiring the entire image. Researchers~\cite{RICE08:Single} are working on the construction of a device to do just that. Coined the ``single-pixel'' compressive sampling camera, this camera consists of a digital micromirror device (DMD), two lenses, a \textit{single} photon detector and an analog-to-digital (A/D) converter. The first lens focuses the light onto the DMD. Each mirror on the DMD corresponds to a pixel in the image, and can be tilted toward or away from the second lens. This operation is analogous to creating inner products with random vectors. This light is then collected by the lens and focused onto the photon detector where the measurement is computed. This optical computer computes the random linear measurements of the image in this way and passes those to a digital computer that reconstructs the image. 
          	
          	Since this camera utilizes only one photon detector, its design is a stark contrast to the usual large photon detector array in most cameras. The single-pixel compressive sampling camera also operates at a much broader range of the light spectrum than traditional cameras that use silicon. For example, because silicon cannot capture a wide range of the spectrum, a digital camera to capture infrared images is much more complex and costly. 
          	
          	Compressed Sensing is also used in medical imaging, in particular with magnetic resonance (MR) images which sample Fourier coefficients of an image. MR images are implicitly sparse and can thus capitalize on the theories of Compressed Sensing. Some MR images such as angiograms are sparse in their actual pixel representation, whereas more complicated MR images are sparse with respect to some other basis, such as the wavelet Fourier basis. MR imaging in general is very time costly, as the speed of data collection is limited by physical and physiological constraints. Thus it is extremely beneficial to reduce the number of measurements collected without sacrificing quality of the MR image. Compressed Sensing again provides exactly this, and many Compressed Sensing algorithms have been specifically designed with MR images in mind~\cite{GGR95:Neuro, LDP07:MRI}.    
          		\subsection[Radar]{Radar}
          	\label{sec:Intro:Applications:Radar}
          	There are many other applications to compressed sensing (see~\cite{CSwebpage}), and one additional application is Compressive Radar Imaging. A standard radar system transmits some sort of pulse (for example a linear chirp), and then uses a matched filter to correlate the signal received with that pulse. The receiver uses a pulse compression system along with a high-rate analog to digital (A/D) converter. This conventional approach is not only complicated and expensive, but the resolution of targets in this classical framework is limited by the radar uncertainty principle. Compressive Radar Imaging tackles these problems by discretizing the time-frequency plane into a grid and considering each possible target scene as a matrix. If the number of targets is small enough, then the grid will be sparsely populated, and we can employ Compressed Sensing techniques to recover the target scene. See~\cite{BS07:radar, HS07:radar} for more details.

%% file: basispursuit.tex
Recall that sparse recovery can be formulated as the generally NP-Hard problem~\eqref{eq:ell0} to recover a signal $\vct{x}$. Donoho and his collaborators showed (see e.g. \cite{DS89:Uncertainty-Principles}) that for certain measurement matrices $\Phi$, this hard problem is equivalent to its relaxation,
\begin{equation}\label{eq:ell1}
\min_{\vct{z}\in\R^d} \|\vct{z}\|_1 \quad \text{subject to} \quad \Phi \vct{z} = \vct{u}.
\end{equation}
Cand\`es and Tao proved that for measurement matrices satisfying a certain quantitative property, the programs~\eqref{eq:ell0} and~\eqref{eq:ell1} are equivalent~\cite{CT05:Decoding}. 

			\subsection[Description]{Description}
    	\label{sec:Approaches:Basis:Description}
    	Since the problem~\eqref{eq:ell0} is not numerically feasible, it is clear that if one is to solve the problem efficiently, a different approach is needed. At first glance, one may instead wish to consider the mean square approach, based on the minimization problem, 
\begin{equation}\label{eq:ell2}
\min_{\vct{z}\in\R^d} \|\vct{z}\|_2 \quad \text{subject to} \quad \Phi \vct{z} = \vct{u}.
\end{equation}
    	Since the minimizer $\vct{x}^*$ must satisfy $\Phi \vct{x}^* = \vct{u} = \Phi \vct{x}$, the minimizer must be in the subspace $K \defby \vct{x} + \ker\Phi$. In fact, the minimizer $\vct{x}^*$ to~\eqref{eq:ell2} is the contact point where the smallest Euclidean ball centered at the origin meets the subspace $K$. As is illustrated in Figure~\ref{fig:mins}, this contact point need not coincide with the actual signal $\vct{x}$. This is because the geometry of the Euclidean ball does not lend itself well to detecting sparsity. 
    	
    	We may then wish to consider the $\ell_1$-minimization problem~\eqref{eq:ell1}. In this case, the minimizer $\vct{x}^*$ to~\eqref{eq:ell1} is the contact point where the smallest octahedron centered at the origin meets the subspace $K$. Since $\vct{x}$ is sparse, it lies in a low-dimensional coordinate subspace. Thus the octahedron has a wedge at $\vct{x}$ (see Figure ~\ref{fig:mins}), which forces the minimizer $\vct{x}^*$ to coincide with $\vct{x}$ for many subspaces $K$. 
    	
    	Since the $\ell_1$-ball works well because of its geometry, one might think to then use an $\ell_p$ ball for some $0 < p < 1$. The geometry of such a ball would of course lend itself even better to sparsity. Indeed, some work in compressed sensing has used this approach (see e.g. \cite{FL08:Sparsest,DG08:RIC}), however, recovery using such a method has not yet provided optimal results. The program~\eqref{eq:ell1} has the advantage over those with $p < 1$ because linear programming can be used to solve it. See Section~\ref{sec:Approaches:Basis:Linear} for a discussion on linear programming. Basis Pursuit utilizes the geometry of the octahedron to recover the sparse signal $\vct{x}$ using measurement matrices $\Phi$ that satisfy a deterministic property.

\begin{center}
 	\begin{figure}[ht] 
 	\centering
  \includegraphics[width=0.4\textwidth,height=1.6in]{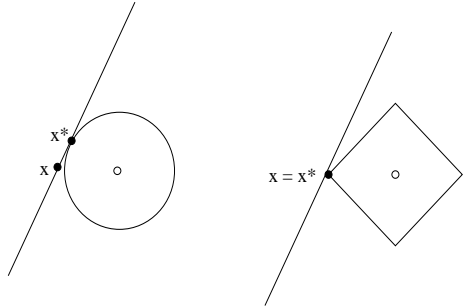}
  \caption{The minimizers to the mean square (left) and $\ell_1$ (right) approaches.}\label{fig:mins}
\end{figure} 
\end{center}
    	
    	\subsection[Restricted Isometry Condition]{Restricted Isometry Condition}
    	\label{sec:Approaches:Basis:Restricted}
    	As discussed above, to guarantee exact recovery of every $s$-sparse signal, the measurement matrix $\Phi$ needs to be one-to-one on all $2s$-sparse vectors. Cand\`es and Tao~\cite{CT05:Decoding} showed that under a slightly stronger condition, Basis Pursuit can recover every $s$-sparse signal by solving~\eqref{eq:ell1}. To this end, we say that the restricted isometry condition (RIC) holds with parameters $(r, \delta)$ if
    	\begin{equation}\label{eq:RIC}
    	(1 - \delta)\|\vct{x}\|_2 \leq \|\Phi \vct{x}\|_2 \leq (1 + \delta)\|\vct{x}\|_2
    	\end{equation}
    	holds for all $r$-sparse vectors $x$. Often, the quadratic form 
    	\begin{equation}\label{eq:RIC2}
    	(1 - \delta)\|\vct{x}\|_2^2 \leq \|\Phi \vct{x}\|_2^2 \leq (1 + \delta)\|\vct{x}\|_2^2
    	\end{equation}
    	is used for simplicity.  Often the notation $\delta_r$ is used to denote the smallest $\delta$ for which the above holds for all $r$-sparse signals.  Now if we require $\delta$ to be small, this condition essentially means that every subset of $r$ or fewer columns of $\Phi$ is approximately an orthonormal system. Note that if the restricted isometry condition holds with parameters $(2s, 1)$, then $\Phi$ must be one-to-one on all $s$-sparse signals. Indeed, if $\Phi \vct{x} = \Phi \vct{z}$ for two $s$-sparse vectors $\vct{x}$ and $\vct{z}$, then $\Phi (\vct{x}-\vct{z})=0$, so by the left inequality, $\|\vct{x}-\vct{z}\|_2 = 0$. 
    	To use this restricted isometry condition in practice, we of course need to determine what kinds of matrices have small restricted isometry constants, and how many measurements are needed. Although it is quite difficult to check whether a given matrix satisfies this condition, it has been shown that many matrices satisfy the restricted isometry condition with high probability and few measurements. In particular, it has been show that with exponentially high probability, random Gaussian, Bernoulli, and partial Fourier matrices satisfy the restricted isometry condition with number of measurements nearly linear in the sparsity level. 
    	\begin{description} \setlength{\itemsep}{0.5pc}
\item[Subgaussian matrices:]
			A random variable $X$ is subgaussian if $\mathbb{P}(|X|>t) \leq Ce^{-ct^2}$ for all $t > 0$ and some positive constants $C$, $c$. Thus subgaussian random variables have tail distributions that are dominated by that of the standard Gaussian random variable. Choosing $C=c=1$, we trivially have that standard Gaussian matrices (those whose entries are Gaussian) are subgaussian. Choosing $C=\frac{1}{e}$ and $c=1$, we see that Bernoulli matrices (those whose entries are uniform $\pm 1$) are also subgaussian. More generally, any bounded random variable is subgaussian. The following theorem proven in~\cite{MPJ06:Uniform} shows that any subgaussian measurement matrix satisfies the restricted isometry condition with number of measurements $m$ nearly linear in the sparsity $s$.
			\begin{theorem}[Subgaussian measurement matrices] Let $\Phi$ be a $m \times d$ subgaussian measurement matrix, and let $s \geq 1$, $0 < \delta < 1$, and $0 < \varepsilon < 0.5$. Then with probability $1 - \alpha$ the matrix $\frac{1}{\sqrt{m}}\Phi$ satisfies the restricted isometry condition with parameters $(s, \varepsilon)$ provided that the number of measurements $m$ satisfies
			$$
			m \geq \frac{Cs}{\varepsilon^2}\log\big(\frac{d}{\varepsilon^2 s}\big),
			$$  
			where $C$ depends only on $\alpha$ and other constants in the definition of subgaussian (for details on the dependence, see~\cite{MPJ06:Uniform}). 
			\end{theorem}
			
\item[Partial bounded orthogonal matrices:] Let $\Psi$ be an orthogonal $d \times d$ matrix whose entries are bounded by $C/\sqrt{d}$ for some constant $C$. A $m \times d$ partial bounded orthogonal matrix is a matrix $\Phi$ formed by choosing $m$ rows of such a matrix $\Psi$ uniformly at random. Since the $d \times d$ discrete Fourier transform matrix is orthogonal with entries bounded by $1/\sqrt{d}$, the $m \times d$ random partial Fourier matrix is a partial bounded orthogonal matrix. The following theorem proved in~\cite{RV08:sparse} shows that such matrices satisfy the restricted isometry condition with number of measurements $m$ nearly linear in the sparsity $s$.

\begin{theorem}[Partial bounded orthogonal measurement matrices] Let $\Phi$ be a $m \times d$ partial bounded orthogonal measurement matrix, and let $s \geq 1$, $0 < \delta < 1$, and $0 < \varepsilon < 0.5$. Then with probability $1 - \alpha$ the matrix $\frac{d}{\sqrt{m}}\Phi$ satisfies the restricted isometry condition with parameters $(s, \varepsilon)$ provided that the number of measurements $m$ satisfies
$$
m \geq C\big(\frac{s\log d}{\varepsilon^2}\big)\log\big(\frac{s\log d}{\varepsilon^2}\big)\log^2 d,
$$
where $C$ depends only on the confidence level $\alpha$ and other constants in the definition of partial bounded orthogonal matrix (for details on the dependence, see~\cite{RV08:sparse}).

\end{theorem}  

\end{description}

    	\subsection[Main Theorems]{Main Theorems}
    	\label{sec:Approaches:Basis:Main}
    	Cand\`es and Tao showed in~\cite{CT05:Decoding} that for measurement matrices that satisfy the restricted isometry condition, Basis Pursuit recovers all sparse signals exactly.  This is summarized in the following theorem.
    	\begin{theorem}[Sparse recovery under RIC \cite{CT05:Decoding}]  \label{recovery RIC}
  Assume that the measurement matrix $\Phi$ satisfies the
  restricted isometry condition with parameters $(3s, 0.2)$.
  Then every $s$-sparse vector $\vct{x}$ can be exactly recovered from 
  its measurements $\vct{u} = \Phi \vct{x}$ as a unique solution to the linear
  optimization problem \eqref{eq:ell1}.  
\end{theorem}
    	
    	 Note that these guarantees are \textit{uniform}.  Once the measurement matrix $\Phi$ satisfies the restricted isometry condition, Basis Pursuit correctly recovers \textit{all} sparse vectors.  
    	
    	As discussed earlier, exactly sparse vectors are not encountered in practice, but rather nearly sparse signals.  The signals and measurements are also noisy in practice, so practitioners seek algorithms that perform well under these conditions.  Cand\`es, Romberg and Tao showed in~\cite{CRT06:Stable} that a version of Basis Pursuit indeed approximately recovers signals contaminated with noise.  It is clear that in the noisy case, \eqref{eq:ell1} is not a suitable method since the exact equality in the measurements would be most likely unattainable.  Thus the method is modified slightly to allow for small perturbations, searching over all signals consistent with the measurement data.  Instead of \eqref{eq:ell1}, we consider the formulation
    	\begin{equation} \label{eqn:bp}
\min \| \vct{y} \|_1
\quad\subjto\quad
\|\Phi \vct{y} - \vct{u}\|_2\leq \eps.
\end{equation}
    	Cand\`es, Romberg and Tao showed that the program~\eqref{eqn:bp} reconstructs the signal with error at most proportional to the noise level.  First we consider exactly sparse signals whose measurements are corrupted with noise.  In this case, we have the following results from~\cite{CRT06:Stable}.
    	
    	\begin{theorem}[Stability of BP~\cite{CRT06:Stable}]\label{stableBP1}
    	Let $\Phi$ be a measurement matrix satisfying the restricted isometry condition with parameters $(3s, 0.2)$. Then for any $s$-sparse signal $\vct{x}$ and corrupted measurements $\vct{u} = \Phi\vct{x} + \vct{e}$ with $\|\vct{e}\|_2 \leq \eps$, the solution $\hat{\vct{x}}$ to~\eqref{eqn:bp} satisfies
    	$$
    	\|\hat{\vct{x}}-\vct{x}\|_2 \leq C_s\cdot\eps,
    	$$
    	where $C_s$ depends only on the RIC constant $\delta$.
    	\end{theorem}
    	
    	Note that in the noiseless case, this theorem is consistent with Theorem~\ref{recovery RIC}.  Theorem~\ref{stableBP1} is quite surprising given the fact that the matrix $\Phi$ is a wide rectangular matrix.  Since it has far more columns than rows, most of the singular values of $\Phi$ are zero.  Thus this theorem states that even though the problem is very ill-posed, Basis Pursuit still controls the error.  It is important to point out that Theorem~\ref{stableBP1} is fundamentally optimal.  This means that the error level $\eps$ is in a strong sense \textit{unrecoverable}.  Indeed, suppose that the support $S$ of $\vct{x}$ was known a priori.  The best way to reconstruct $\vct{x}$ from the measurements $\vct{u} = \Phi\vct{x} + \vct{e}$ in this case would be to apply the pseudoinverse $\Phi_S^\psinv \defby (\Phi_S^*\Phi_S)^{-1} \Phi_S^*$ on the support, and set the remaining coordinates to zero. That is, one would reconstruct $\vct{x}$ as
    	$$
    	\hat{\vct{x}} = \left\{ \begin{array}{lr}
    	\Phi^\psinv_S u & \text{on $S$}\\
    	 0 & \text{elsewhere}
    	 \end{array}\right.
    	$$
    	Since the singular values of $\Phi_S$ are controlled, the error on the support is approximately $\|\vct{e}\|_2 \leq \eps$, and the error off the support is of course zero.  This is also the error guaranteed by Theorem~\ref{stableBP1}.  Thus no recovery algorithm can hope to recover with less error than the original error introduced to the measurements.  
    	
    	Thus Basis Pursuit is stable to perturbations in the measurements of exactly sparse vectors. This extends naturally to the approximate recovery of nearly sparse signals, which is summarized in the companion theorem from~\cite{CRT06:Stable}.
    	
    	\begin{theorem}[Stability of BP II~\cite{CRT06:Stable}]\label{stableBP2}
    	Let $\Phi$ be a measurement matrix satisfying the restricted isometry condition with parameters $(3s, 0.2)$. Then for any \textit{arbitrary} signal and corrupted measurements $\vct{u} = \Phi\vct{x} + \vct{e}$ with $\|\vct{e}\|_2 \leq \eps$, the solution $\hat{\vct{x}}$ to~\eqref{eqn:bp} satisfies
    	$$
    	\|\hat{\vct{x}}-\vct{x}\|_2 \leq C_s\cdot\eps + C_s'\cdot\frac{\|\vct{x}-\vct{x}_s\|_1}{\sqrt{s}},
    	$$
      where $\vct{x}_s$ denotes the vector of the largest coefficients in magnitude of $\vct{x}$.
    	\end{theorem}
    	
    	\begin{remark}
    	In~\cite{Can08:Restricted-Isometry}, Cand\`es sharpened Theorems~\ref{recovery RIC},~\ref{stableBP1}, and~\ref{stableBP2} to work under the restricted isometry condition with parameters $(2s, \sqrt{2}-1)$.
    	\end{remark}
    	
    	Theorem~\ref{stableBP2} says that for an arbitrary signal $\vct{x}$, Basis Pursuit approximately recovers its largest $s$ coefficients.  In the particularly useful case of compressible signals, we have that for signals $\vct{x}$ obeying~\eqref{comp}, the reconstruction satisfies 
    	\begin{equation}\label{BPcomp}
    	\|\hat{\vct{x}} - \vct{x}\|_2 \leq C_s\cdot\eps + C's^{-q+1/2},
    	\end{equation}
    	where $C'$ depends on the RIC constant and the constant $R$ in the compressibility definition~eqref{comp}.  We notice that for such signals we also have
    	$$
    	\|\vct{x}_s - \vct{x}\|_2 \leq C_R s^{-q+1/2},
    	$$
    	where $C_R$ depends on $R$.  Thus the error in the approximation guaranteed by Theorem~\ref{stableBP2} is comparable to the error obtained by simply selecting the $s$ largest coefficients in magnitude of a compressible signal.  So at least in the case of compressible signals, the error guarantees are again optimal.  See Section~\ref{summary} for a discussion of advantages and disadvantages to this approach. 
    	
    	\subsection[Numerical Results]{Numerical Results}
    	\label{sec:Approaches:Basis:Numerical}
    	Many empirical studies have been conducted using Basis Pursuit.  Several are included here, other results can be found in~\cite{CRT06:Stable,DT08:Counting,CT05:Decoding}.  The studies discussed here were performed in Matlab, with the help of $\ell_1$-Magic code by Romberg~\cite{L1Magic}.  The code is given in Appendix~\ref{app:code:BPcode}.  In all cases here, the measurement matrix is a Gaussian matrix and the ambient dimension $d$ is 256.  In the first study, for each trial we generated binary signals with support uniformly selected at random as well as an independent Gaussian matrix for many values of the sparsity $s$ and number of measurements $m$.  Then we ran Basis Pursuit on the measurements of that signal and counted the number of times the signal was recovered correctly out of $500$ trials.  The results are displayed in Figure~\ref{bpfig1}.  The $99\%$ recovery trend is depicted in Figure~\ref{bpfig2}.  This curve shows the relationship between the number of measurements $m$ and the sparsity level $s$ to guarantee that correct recovery occurs $99\%$ of the time. Note that by \textit{recovery}, we mean that the estimation error falls below the threshold of $10^{-5}$. Figure~\ref{bpfig3} depicts the recovery error of Basis Pursuit when the measurements were perturbed.  For this simulation, the signals were again binary (flat) signals, but Gaussian noise was added to the measurements.  The norm of the noise was chosen to be the constant $1/2$.  The last figure, Figure~\ref{bpfig4} displays the recovery error when Basis Pursuit is run on compressible signals.  For this study, the Basis Pursuit was run on signals whose coefficients obeyed the power law~\eqref{comp}.  This simulation was run with sparsity $s=12$, dimension $d=256$, and for various values of the compressibility constant $q$.  Note that the smaller $q$ is, the more quickly the coefficients decay.  
    	
    	\begin{figure}[htbp] 
  \includegraphics[width=0.8\textwidth,height=3.2in]{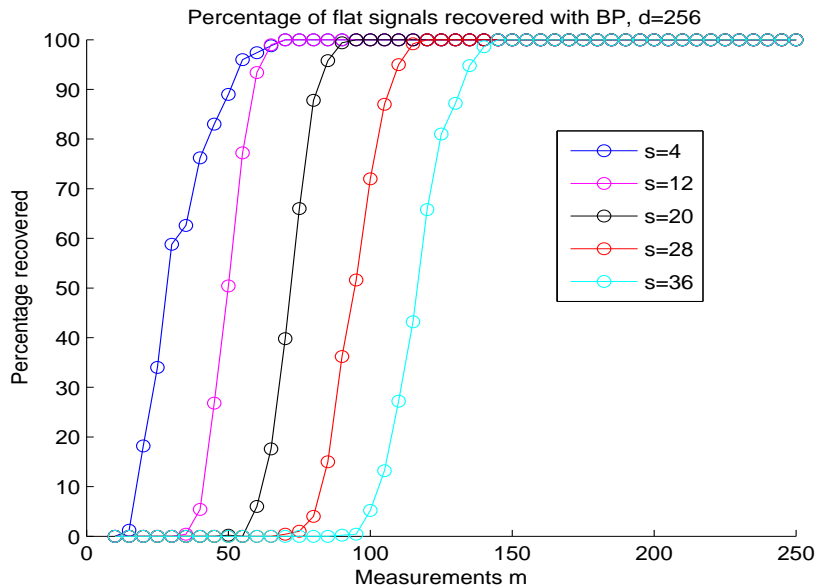}
  \caption{The percentage of sparse flat signals exactly recovered by Basis Pursuit as a function of the number of measurements $m$ in dimension $d=256$ for various levels of sparsity $s$.}\label{bpfig1}
\end{figure}

\begin{figure}[htbp] 
  \includegraphics[width=0.8\textwidth,height=3.2in]{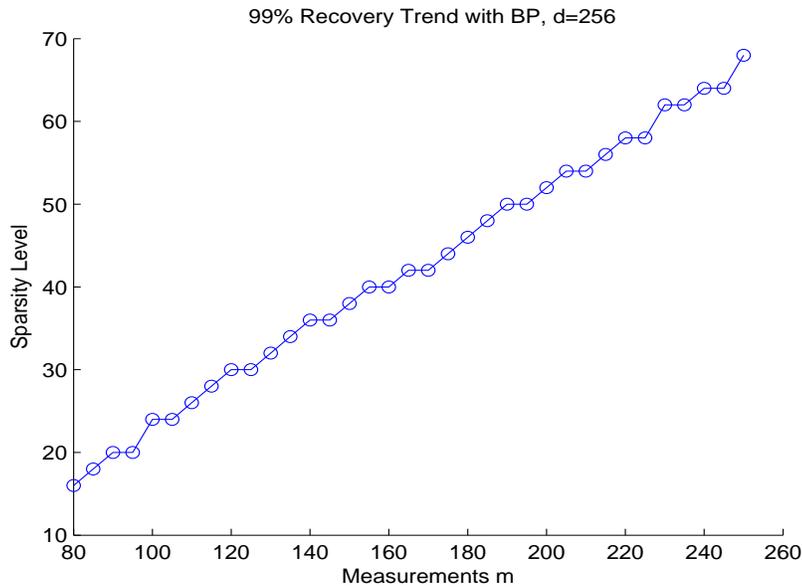}
  \caption{The $99\%$ recovery trend of Basis Pursuit as a function of the number of measurements $m$ in dimension $d=256$ for various levels of sparsity $s$.}\label{bpfig2}
\end{figure}

\begin{figure}[htbp] 
  \includegraphics[width=0.8\textwidth,height=3.2in]{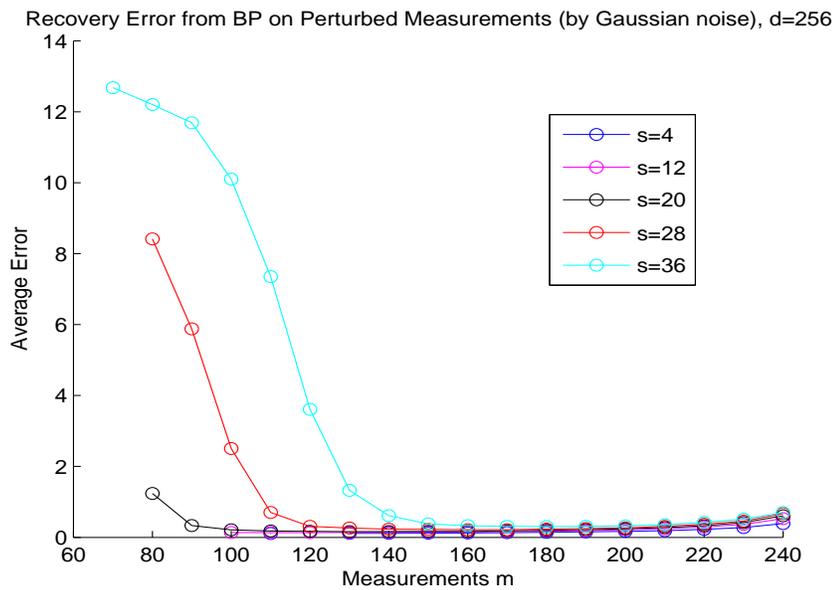}
  \caption{The recovery error of Basis Pursuit under perturbed measurements as a function of the number of measurements $m$ in dimension $d=256$ for various levels of sparsity $s$.}\label{bpfig3}
\end{figure}

\begin{figure}[htbp] 
  \includegraphics[width=0.8\textwidth,height=3.2in]{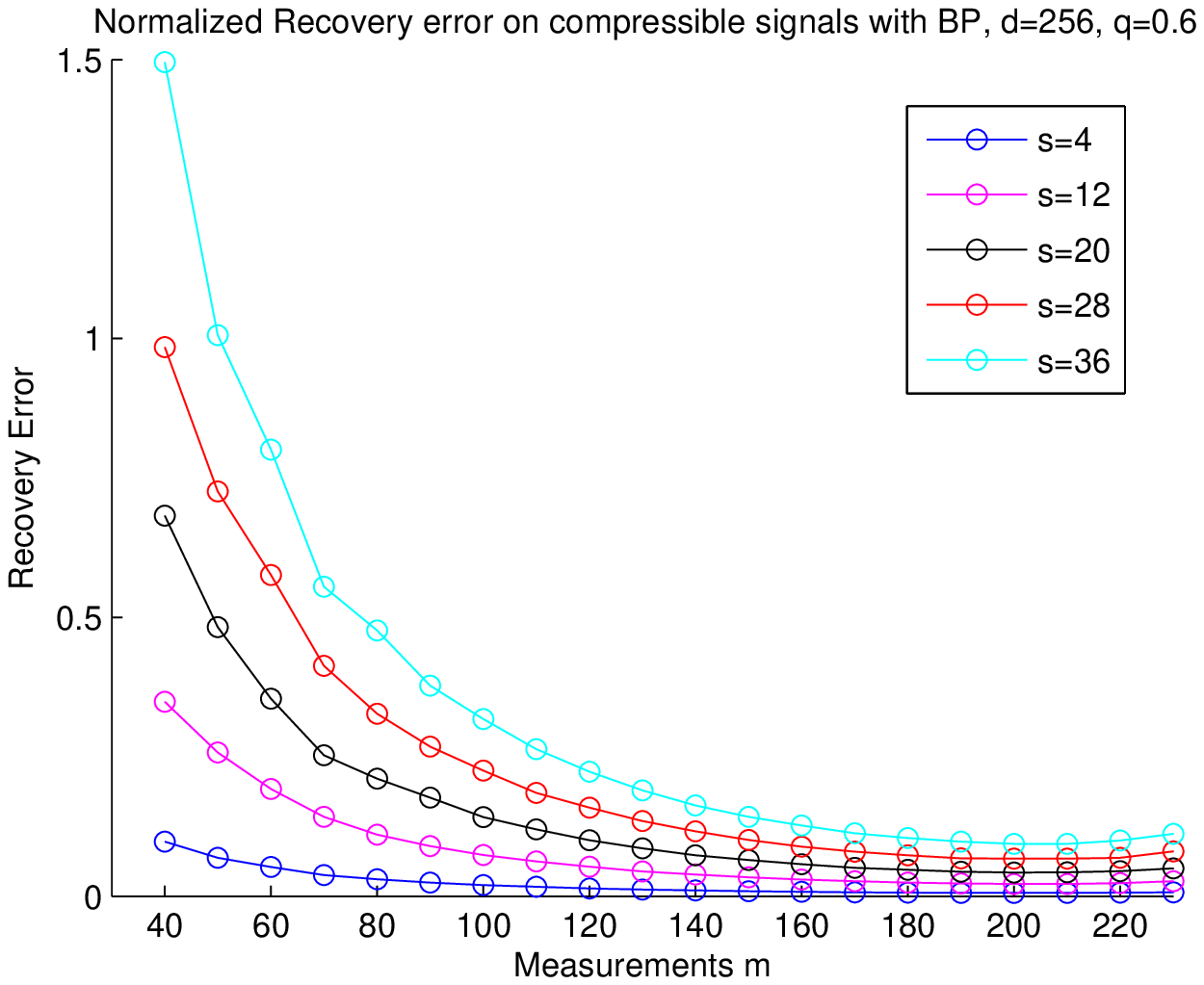}
  \caption{The normalized recovery error $\frac{\|x-\hat{x}\|_2}{\sqrt{s}\|x-x_s\|_1} $ of Basis Pursuit for compressible signals as a function of the number of measurements $m$ in dimension $d=256$ with compressibility $q=0.6$ for various levels of sparsity $s$.}\label{bpfig4}
\end{figure}
    	
    	 	\subsection[Linear Programming]{Linear Programming}
    	\label{sec:Approaches:Basis:Linear}
    	Linear programming is a technique for optimization of a linear objective function under equality and inequality constraints, all of which are linear~\cite{DT97:Linear}.  The problem~\eqref{eq:ell1} can be recast as a linear program whose objective function (to be minimized) is
$$
\sum_{i=1}^{d} t_i,
$$
with constraints
$$
-t_i \leq z_i \leq t_i, \quad \Phi z = u.
$$
    	
Viewed geometrically, the set of linear constraints, making up the feasible region, forms a convex polyhedron.  By the Karush–-Kuhn–-Tucker conditions~\cite{KT51:Nonlinear}, all local optima are also global optima.  If an optimum exists, it will be attained at a vertex of the polyhedron.  There are several methods to search for this optimal vertex.

One of the most popular algorithms in linear programming is the simplex algorithm, developed by George Dantzig~\cite{DT97:Linear}.  The simplex method begins with some admissible starting solution.  If such a point is not known, a different linear program (with an obvious admissible solution) can be solved via the simplex method to find such a point.  The simplex method then traverses the edges of the polytope via a sequence of pivot steps.  The algorithm moves along the polytope, at each step choosing the optimal direction, until the optimum is found.  Assuming that precautions against cycling are taken, the algorithm is guaranteed to find the optimum.  Although it's worst-case behavior is exponential in the problem size, it is much more efficient in practice.  Smoothed analysis has explained this efficiency in practice~\cite{R06:Beyond}, but it is still unknown whether there is a strongly polynomial bound on the runtime.  

The simplex algorithm traverses the polytope's edges, but an alternative method called the interior point method traverses the interior of the polytope~\cite{NN94:Interior-Point}.  One such method was proposed by Karmarkar and is an interior point projective method.  Recently, barrier function or path-following methods are being used for practical purposes.  The best bound currently attained on the runtime of an interior point method is $O(m^2d^{1.5})$.  Other methods have been proposed, including the ellipsoid method by Khachiyan which has a polynomial worst case runtime, but as of yet none have provided strongly polynomial bounds.
    	
    	\subsection[Summary]{Summary}\label{summary}
    	\label{sec:Approaches:Basis:Summary}
    	
 Basis Pursuit presents many advantages over other algorithms in compressed sensing.  Once a measurement matrix satisfies the restricted isometry condition, Basis Pursuit reconstructs \textit{all} sparse signals.  The guarantees it provides are thus uniform, meaning the algorithm will not fail for any sparse signal.  Theorem~\ref{stableBP2} shows that Basis Pursuit is also stable, which is a necessity in practice.  Its ability to handle noise and non-exactness of sparse signals makes the algorithm applicable to real world problems.  The requirements on the restricted isometry constant shown in Theorem~\ref{stableBP2} along with the known results about random matrices discussed in Section~\ref{sec:Approaches:Basis:Restricted} mean that Basis Pursuit only requires $O(s\log d)$ measurements to reconstruct $d$-dimensional $s$-sparse signals.  It is thought by many that this is the optimal number of measurements.
 
 Although Basis Pursuit provides these strong guarantees, its disadvantage is of course speed.  It relies on Linear Programming which although is often quite efficient in practice, has a polynomial runtime.  For this reason, much work in compressed sensing has been done using faster methods.  This approach uses greedy algorithms, which are discussed next.

%% file: greedymethods.tex
An alternative approach to compressed sensing is the use of greedy algorithms.  Greedy algorithms compute the support of the sparse signal $\vct{x}$ iteratively.  Once the support of the signal is compute correctly, the pseudo-inverse of the measurement matrix restricted to the corresponding columns can be used to reconstruct the actual signal $\vct{x}$.  The clear advantage to this approach is speed, but the approach also presents new challenges.

   \subsection[Orthogonal Matching Pursuit]{Orthogonal Matching Pursuit}
   \label{sec:Approaches:Greedy:Orthogonal}
   
   One such greedy algorithm is Orthogonal Matching Pursuit (OMP), put forth by Mallat and his collaborators (see e.g.~\cite{MZ93:Matching-Pursuits}) and analyzed by Gilbert and Tropp~\cite{TG07:Signal-Recovery}.  OMP uses subGaussian measurement matrices to reconstruct sparse signals.  If $\Phi$ is such a measurement matrix, then $\Phi^*\Phi$ is in a loose sense close to the identity.  Therefore one would expect the largest coordinate of the observation vector $\vct{y} = \Phi^*\Phi \vct{x}$ to correspond to a non-zero entry of $\vct{x}$.  Thus one coordinate for the support of the signal $\vct{x}$ is estimated.  Subtracting off that contribution from the observation vector $\vct{y}$ and repeating eventually yields the entire support of the signal $\vct{x}$.  OMP is quite fast, both in theory and in practice, but its guarantees are not as strong as those of Basis Pursuit.
    	
    		\subsubsection[Description]{Description}
    		\label{sec:Approaches:Greedy:Orthogonal:Description}
    		The OMP algorithm can thus be described as follows:
    		
    		\bigskip
   \textsc{Orthogonal Matching Pursuit (OMP)}

\nopagebreak

\fbox{\parbox{\algorithmwidth}{
  \textsc{Input:} Measurement matrix $\Phi$, measurement vector $\vct{u}=\Phi \vct{x}$, sparsity level $s$
  
  \textsc{Output:} Index set $I \subset \{1,\ldots,d\}$

  \textsc{Procedure:}
  \begin{description}
    \item[Initialize] Let the index set $I = \emptyset$ and the residual $\vct{r} = \vct{u}$.\\
      Repeat the following $s$ times:
    \item[Identify] Select the largest coordinate $\lambda$ of $\vct{y}=\Phi^* \vct{r}$ in absolute value. Break ties lexicographically.
    \item[Update] Add the coordinate $\lambda$ to the index set: $I \leftarrow I \cup \{\lambda\}$, 
      and update the residual:
      $$
      \hat{\vct{x}} = \argmin_{\vct{z}} \|\vct{u} - \Phi|_I \vct{z}\|_2; \qquad \vct{r} = \vct{u} - \Phi \hat{\vct{x}}.
      $$
  \end{description}
 }}
 
 \bigskip 
 
  Once the support $I$ of the signal $\vct{x}$ is found, the estimate can be reconstructed as $\hat{\vct{x}} = \Phi_I^\psinv \vct{u}$, where recall we define the pseudoinverse by $\Phi_I^\psinv \defby (\Phi_I^*\Phi_I)^{-1} \Phi_I^*$.
  
  The algorithm's simplicity enables a fast runtime.  The algorithm iterates $s$ times, and each iteration does a selection through $d$ elements, multiplies by $\Phi^*$, and solves a least squares problem.  The selection can easily be done in $O(d)$ time, and the multiplication of $\Phi^*$ in the general case takes $O(md)$.  When $\Phi$ is an unstructured matrix, the cost of solving the least squares problem is $O(s^2d)$.  However, maintaining a QR-Factorization of $\Phi|_I$ and using the modified Gram-Schmidt algorithm reduces this time to $O(|I|d)$ at each iteration.  Using this method, the overall cost of OMP becomes $O(smd)$.  In the case where the measurement matrix $\Phi$ is structured with a fast-multiply, this can clearly be improved.     
    		
    		\subsubsection[Main Theorems and Results]{Main Theorems and Results}
    		\label{sec:Approaches:Greedy:Orthogonal:Main}
    		
    		Gilbert and Tropp showed that OMP correctly recovers a \textit{fixed} sparse signal with high probability. Indeed, in~\cite{TG07:Signal-Recovery} they prove the following.
    		\begin{theorem}[OMP Signal Recovery \cite{TG07:Signal-Recovery}]\label{OMPRecovery}
    		Fix $\delta\in(0, 0.36)$ and let $\Phi$ be an $m\times d$ Gaussian measurement matrix with $m \geq Cm\log(d/\delta)$.  Let $x$ be an $s$-sparse signal in $\R^d$.  Then with probability exceeding $1-2\delta$, OMP correctly reconstructs the signal $x$ from its measurements $\Phi x$.
    		\end{theorem}
    		
    		Similar results hold when $\Phi$ is a subgaussian matrix.  We note here that although the measurement requirements are similar to those of Basis Pursuit, the guarantees are not uniform.  The probability is for a fixed signal rather than for all signals. The type of measurement matrix here is also more restrictive, and it is unknown whether OMP works for the important case of random Fourier matrices.

    			\subsubsection[Numerical Experiments]{Numerical Experiments}
    		\label{sec:Approaches:Greedy:Orthogonal:Numerical}
    		
    		Many empirical studies have been conducted to study the success of OMP.  One study is described here that demonstrates the relationship between the sparsity level $s$ and the number of measurements $m$. Other results can be found in~\cite{TG07:Signal-Recovery}.  The study discussed here was performed in Matlab, and is given in Appendix~\ref{app:code:OMPcode}.. In the study, for each trial I generated binary signals with support uniformly selected at random as well as an independent Gaussian measurement matrix, for many values of the sparsity $s$ and number of measurements $m$.  Then I ran OMP on the measurements of that signal and counted the number of times the signal was recovered correctly out of $500$ trials.  The results are displayed in Figure~\ref{ompfig1}.  The $99\%$ recovery trend is depicted in Figure~\ref{ompfig2}.  This curve shows the relationship between the number of measurements $m$ and the sparsity level $s$ to guarantee that correct recovery occurs $99\%$ of the time.  In comparison with Figures~\ref{bpfig1} and~\ref{bpfig2} we see that Basis Pursuit appears to provide stronger results empirically as well.    
    	
    	\begin{figure}[htbp] 
  \includegraphics[width=0.8\textwidth,height=3.2in]{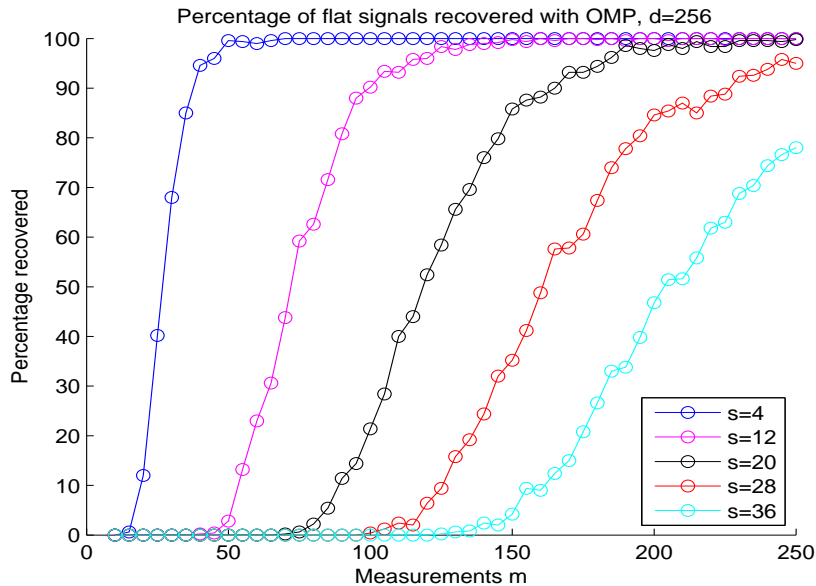}
  \caption{The percentage of sparse flat signals exactly recovered by OMP as a function of the number of measurements $m$ in dimension $d=256$ for various levels of sparsity $s$.}\label{ompfig1}
\end{figure}

\begin{figure}[htbp] 
  \includegraphics[width=0.8\textwidth,height=3.2in]{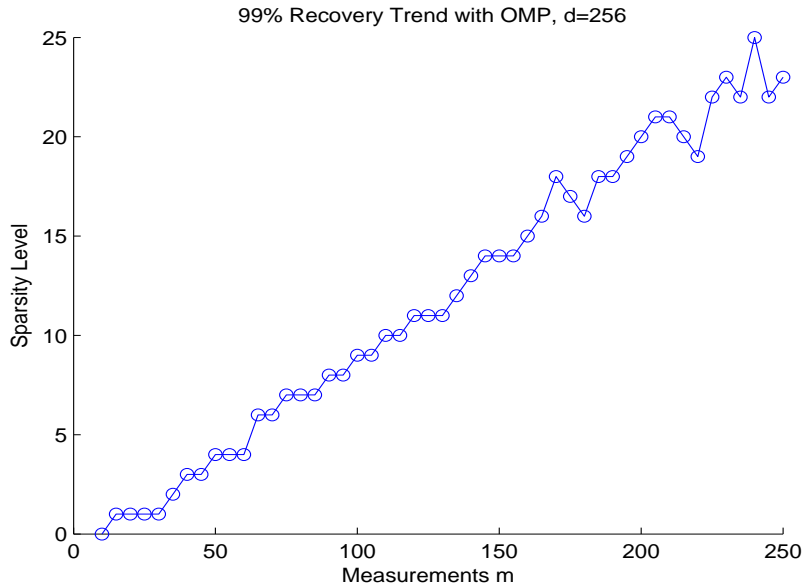}
  \caption{The $99\%$ recovery trend of OMP as a function of the number of measurements $m$ in dimension $d=256$ for various levels of sparsity $s$.}\label{ompfig2}
\end{figure}
	
    		\subsubsection[Summary]{Summary}
    		\label{sec:Approaches:Greedy:Orthogonal:Summary}
    		
    		It is important to note the distinctions between this theorem for OMP and Theorem~\ref{recovery RIC} for Basis Pursuit.  The first important difference is that Theorem~\ref{OMPRecovery} shows that OMP works only for the case when $\Phi$ is a Gaussian (or subgaussian) matrices, whereas Theorem~\ref{recovery RIC} holds for a more general class of matrices (those which satisfy the RIC).  Also, Theorem~\ref{recovery RIC} demonstrates that Basis Pursuit works correctly for \textit{all} signals, once the measurement matrix satisfies the restricted isometry condition.  Theorem~\ref{OMPRecovery} shows only that OMP works with high probability for each fixed signal.  The advantage to OMP however, is that its runtime has a much faster bound than that of Basis Pursuit and Linear Programming.
    	
  \subsection[Stagewise Orthogonal Matching Pursuit]{Stagewise Orthogonal Matching Pursuit}
  \label{sec:Approaches:Greedy:Stagewise}
  
  An alternative greedy approach, Stagewise Orthogonal Matching Pursuit (StOMP) developed and analyzed by Donoho and his collaborators~\cite{DTDS06:Sparse-Solution}, uses ideas inspired by wireless communications.  As in OMP, StOMP utilizes the observation vector $\vct{y} = \Phi^*\vct{u}$ where $\vct{u} = \Phi x$ is the measurement vector.  However, instead of simply selecting the largest component of the vector $\vct{y}$, it selects all of the coordinates whose values are above a specified threshold. It then solves a least-squares problem to update the residual. The algorithm iterates through only a fixed number of stages and then terminates, whereas OMP requires $s$ iterations where $s$ is the sparsity level. 
  	
    	\subsubsection[Description]{Description}
    		\label{sec:Approaches:Greedy:Stagewise:Description}
    		
    		The pseudo-code for StOMP can thus be described by the following.

    		\bigskip
   \textsc{Stagewise Orthogonal Matching Pursuit (StOMP)}
    		
    		\nopagebreak

\fbox{\parbox{\algorithmwidth}{
  \textsc{Input:} Measurement matrix $\Phi$, measurement vector $\vct{u}=\Phi \vct{x}$, 
  
  \textsc{Output:} Estimate $\hat{\vct{x}}$ to the signal $\vct{x}$

  \textsc{Procedure:}
  \begin{description}
    \item[Initialize] Let the index set $I = \emptyset$, the estimate $\hat{\vct{x}}=0$, and the residual $\vct{r} = \vct{u}$.\\
      Repeat the following until stopping condition holds:
    \item[Identify] Using the observation vector $\vct{y}=\Phi^* \vct{r}$, set
    $$
    J = \{j : |\vct{y}_j| > t_k\sigma_k\},
    $$
    where $\sigma_k$ is a formal noise level and $t_k$ is a threshold parameter for iteration $k$.
    \item[Update] Add the set $J$ to the index set: $I \leftarrow I \cup J$, 
      and update the residual and estimate:
     $$
     \hat{\vct{x}}|_I = (\Phi_I^*\Phi_I)^{-1}\Phi_I^* \vct{u}, \quad \vct{r} = \vct{u}-\Phi\hat{\vct{x}}.
      $$
  \end{description}
 }}
 
  \bigskip 
  
 				The thresholding strategy is designed so that many terms enter at each stage, and so that algorithm halts after a fixed number of iterations. The formal noise level $\sigma_k$ is proportional the Euclidean norm of the residual at that iteration. See~\cite{DTDS06:Sparse-Solution} for more information on the thresholding strategy.

    		\subsubsection[Main Results]{Main Results}
    		\label{sec:Approaches:Greedy:Stagewise:Main}
    		Donoho and his collaborators studied StOMP empirically and have heuristically derived results.  Figure 6 of~\cite{DTDS06:Sparse-Solution} shows results of StOMP when the thresholding strategy is to control the false alarm rate and the measurement matrix $\Phi$ is sampled from the uniform spherical ensemble.  The false alarm rate is the number of incorrectly selected coordinates (ie. those that are not in the actual support, but are chosen in the estimate) divided by the number of coordinates not in the support of the signal $\vct{x}$. The figure shows that for very sparse
signals, the algorithm recovers a good approximation to the signal. For less
sparse signals the algorithm does not.  The red curve in this figure is the graph of a heuristically derived
function which they call the Predicted Phase transition.  The simulation results and the predicted transition coincide reasonably well.  This thresholding method requires knowledge about the actual sparsity level $s$ of the signal $\vct{x}$. Figure 7 of~\cite{DTDS06:Sparse-Solution} shows similar results for a thresholding strategy that instead tries to control the false discovery rate.  The false discovery rate is the fraction of incorrectly selected coordinates within the estimated support. This method appears to provide slightly weaker results.  It appears however, that StOMP outperforms OMP and Basis Pursuit in some cases. 

Although the structure of StOMP is similar to that of OMP, because StOMP selects many coordinates at each state, the runtime is quite improved. Indeed, using iterative methods to solve the least-squares problem yields a runtime bound of $CNsd + O(d)$, where $N$ is the fixed number of iterations run by StOMP, and $C$ is a constant that depends only on the accuracy level of the least-squares problem.

    		    			\subsubsection[Numerical Experiments]{Numerical Experiments}
    		\label{sec:Approaches:Greedy:Stagewise:Numerical}
    		
    		A thorough empirical study of StOMP is provided in~\cite{DTDS06:Sparse-Solution}. An additional study on StOMP was conducted here using a thresholding strategy with constant threshold parameter.  The noise level $\sigma$ was proportional to the norm of the residual, as~\cite{DTDS06:Sparse-Solution} suggests.  StOMP was run with various sparsity levels and measurement numbers, with Gaussian measurement matrices for 500 trials.  Figure~\ref{stompfig1} depicts the results, and Figure~\ref{stompfig2} depicts the $99\%$ recovery trend.  Next StOMP was run in this same way but noise was added to the measurements. Figure~\ref{stompstab} displays the results of this study. Since the reconstructed signal is always sparse, it is not surprising that StOMP is able to overcome the noise level.  Note that these empirical results are not optimal because of the basic thresholding strategy.  See~\cite{DTDS06:Sparse-Solution} for empirical results using an improved thresholding strategy.
    		
    		    	\begin{figure}[htbp] 
  \includegraphics[width=0.8\textwidth,height=3.2in]{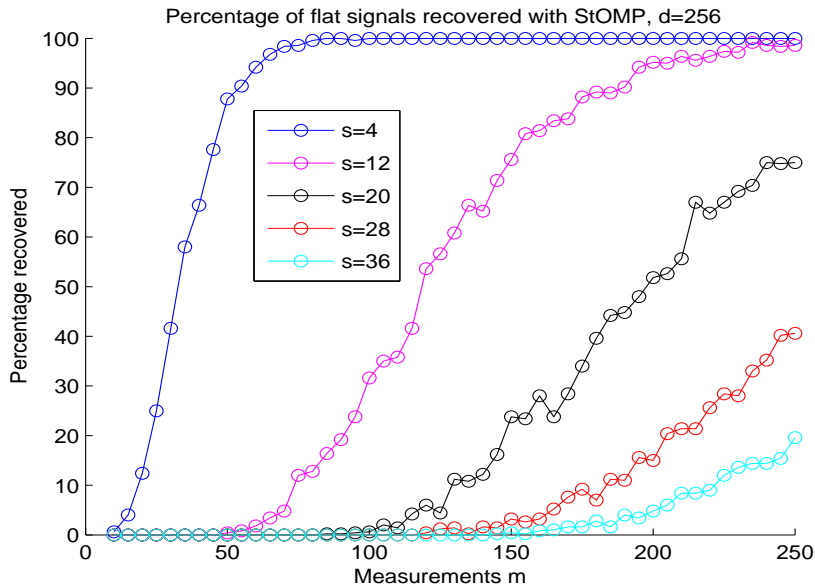}
  \caption{The percentage of sparse flat signals exactly recovered by StOMP as a function of the number of measurements $m$ in dimension $d=256$ for various levels of sparsity $s$.}\label{stompfig1}
\end{figure}

\begin{figure}[htbp] 
  \includegraphics[width=0.8\textwidth,height=3.2in]{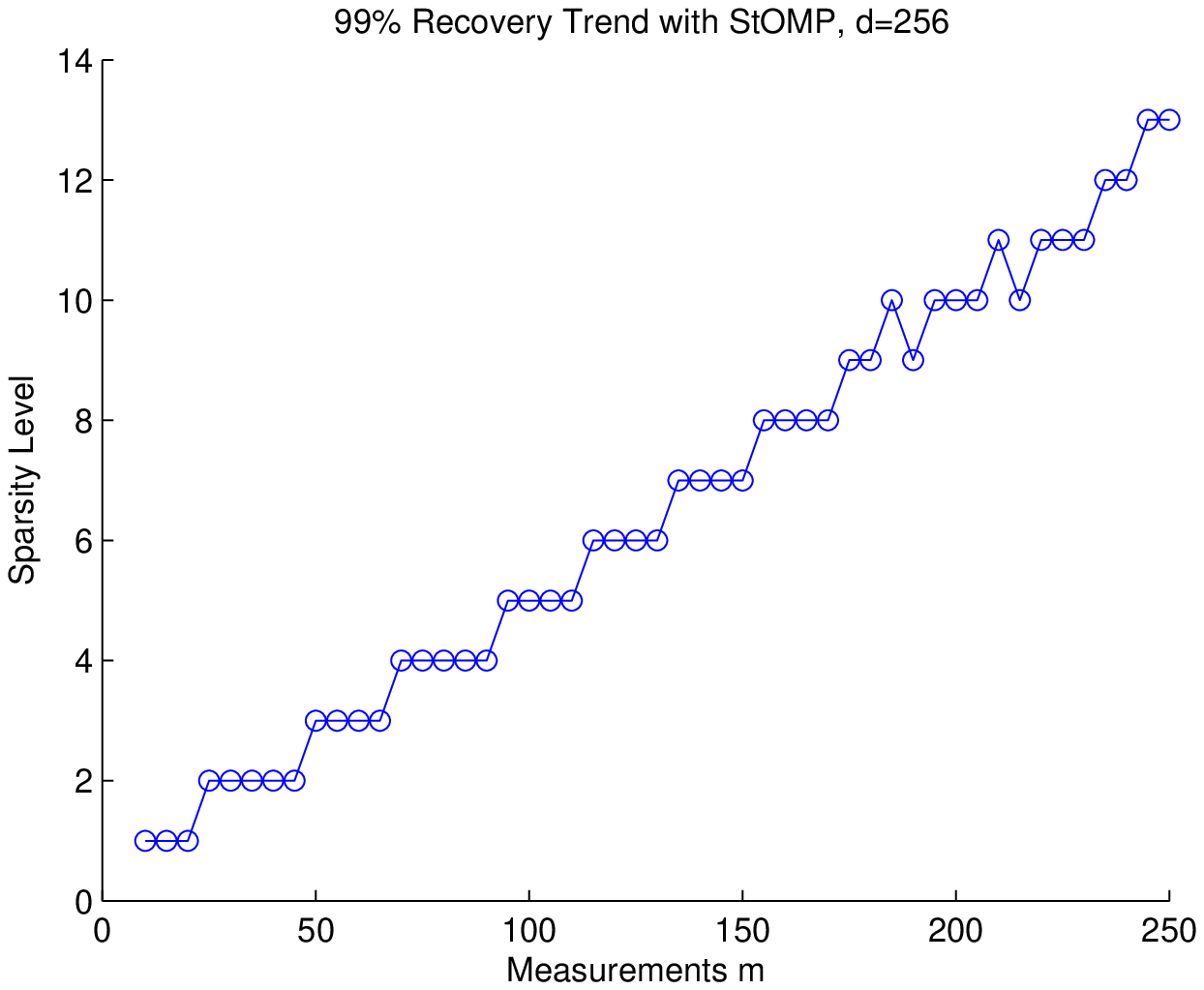}
  \caption{The $99\%$ recovery trend of StOMP as a function of the number of measurements $m$ in dimension $d=256$ for various levels of sparsity $s$.}\label{stompfig2}
\end{figure}
    		
    		\begin{figure}[htbp] 
  \includegraphics[width=0.8\textwidth,height=3.2in]{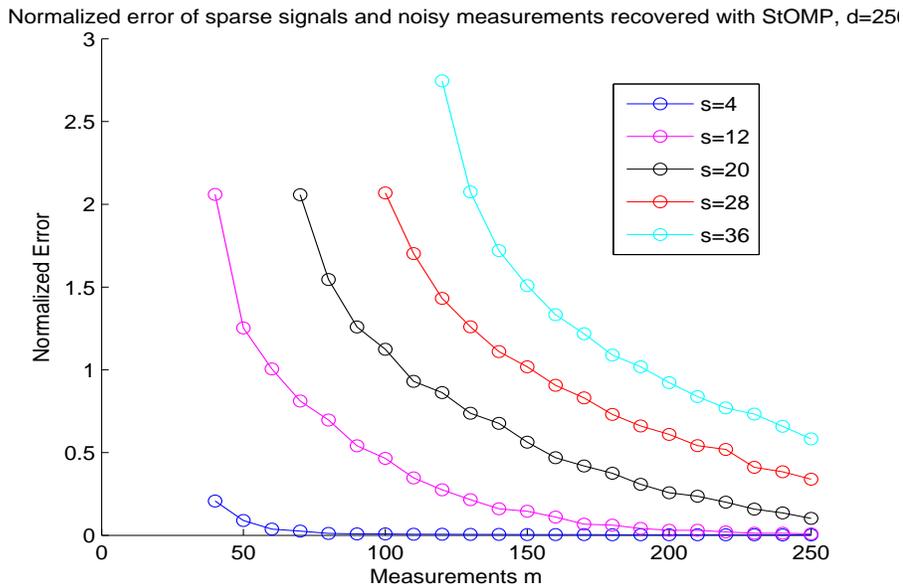}
  \caption{The normalized recovery error $\|\vct{x}-\hat{\vct{x}}\|_2/\|e\|_2$ of StOMP on sparse signals with noisy measurements $\Phi x + e$.}\label{stompstab}
\end{figure}
    		
    		    		\subsubsection[Summary]{Summary}
    		\label{sec:Approaches:Greedy:Stagewise:Summary}
    		The empirical results of StOMP in~\cite{DTDS06:Sparse-Solution} are quite promising, and suggest its improvement over OMP.  However, in practice, the thresholding strategy may be difficult and complicated to implement well.  More importantly, there are no rigorous results for StOMP available.  In the next subsection other greedy methods are discussed with rigorous results, but that require highly structured measurement matrices.

    	  \subsection[Combinatorial Methods]{Combinatorial Methods}
  \label{sec:Approaches:Greedy:Highly}
  
  The major benefit of the greedy approach is its speed, both empirically and theoretically.  There is a group of combinatorial algorithms that provide even faster speed, but that impose very strict requirements on the measurement matrix.  These methods use highly structured measurement matrices that support very fast reconstruction through group testing. The work in this area includes HHS pursuit~\cite{GSTV07:HHS}, chaining pursuit~\cite{GSTV07:Algorithmic}, Sudocodes~\cite{SBB06:Sudocodes}, Fourier sampling~\cite{GGIMS02:Near-Optimal-Sparse,GMS05:Improved} and some others by Cormode--Muthukrishnan~\cite{CM05:Combinatorial} and Iwen~\cite{I07:sub-linear}.  
   	
    	\subsubsection[Descriptions and Results]{Descriptions and Results}
    		\label{sec:Approaches:Greedy:Highly:Description}
    		
    		Many of the sublinear algorithms such as HHS pursuit, chaining pursuit and Sudocodes employ the idea of group testing.  Group testing is a method which originated in the Selective Service during World War II to test soldiers for Syphilis~\cite{D43:detection}, and now it appears in many experimental designs and other algorithms.  During this time, the Wassermann test~\cite{WNB06:syphilis} was used to detect the Syphilis antigen in a blood sample.  Since this test was expensive, the method was to sample a group of men together and test the entire pool of blood samples.  If the pool did not contain the antigen, then one test replaced many.  If it was found, then the process could either be repeated with that group, or each individual in the group could then be tested.  
    		
    		These sublinear algorithms in compressed sensing use this same idea to test for elements of the support of the signal $x$.  Chaining pursuit, for example, uses a measurement matrix consisting of a row tensor product of a bit test matrix and an isolation matrix, both of which are 0-1 matrices.  Chaining pursuit first uses bit tests to locate the positions of the large components of the signal $\vct{x}$ and estimate those values. Then the algorithm retains a portion of the coordinates that are largest magnitude and repeats.  In the end, those coordinates which appeared throughout a large portion of the iterations are kept, and the signal is estimated using these.  Pseudo-code is available in~\cite{GSTV07:Algorithmic}, where the following result is proved. 
    		
    			\begin{theorem}[Chaining pursuit~\cite{GSTV07:Algorithmic}]
    		With probability at least $1 - O(d^{-3})$, the $O(s\log^2d) \times d$ random measurement operator $\Phi$ has the
following property. For $\vct{x}\in\R^d$ and its measurements $u=\Phi\vct{x}$, the Chaining Pursuit algorithm produces a signal $\hat{\vct{x}}$ with at most $s$ nonzero entries. The output $\hat{\vct{x}}$ satisfies
$$
\|\vct{x}-\hat{\vct{x}}\|_1 \leq C(1+\log s)\|\vct{x}-\vct{x}_s\|_1.
$$
The time cost of the algorithm is $O(s\log^2s \log^2d)$. 
    		\end{theorem}
    		
    		HHS Pursuit, a similar algorithm but with improved guarantees, uses a measurement matrix that consists again of two parts.  The first part is an identification matrix, and the second is an estimation matrix.  As the names suggest, the identification matrix is used to identify the location of the large components of the signal, whereas the estimation matrix is used to estimate the values at those locations.  Each of these matrices consist of smaller parts, some deterministic and some random.  Using this measurement matrix to locate large components and estimate their values, HHS Pursuit then adds the new estimate to the previous, and prunes it relative to the sparsity level.  This estimation is itself then sampled, and the residual of the signal is updated. See~\cite{GSTV07:HHS} for the pseudo-code of the algorithm.  Although the measurement matrix is highly structured, again a disadvantage in practice, the results for the algorithm are quite strong. Indeed, in~\cite{GSTV07:HHS} the following result is proved.
    		
    		\begin{theorem}[HHS Pursuit~\cite{GSTV07:HHS}]
    		Fix an integer $s$ and a number $\varepsilon\in (0, 1)$.
With probability at least 0.99, the random measurement matrix $\Phi$ (as described above)
	 has the following property. Let $\vct{x}\in\R^d$ and let $\vct{u} = 	\Phi\vct{x}$ be the measurement vector.
The HHS Pursuit algorithm produces a
signal approximation $\hat{\vct{x}}$ with $O(s/\varepsilon^2)$ nonzero entries. The
approximation satisfies
$$
\|\vct{x}-\hat{\vct{x}}\|_2 \leq \frac{\varepsilon}{\sqrt{s}}\|\vct{x}-\vct{x}_s\|_1,
$$
where again $\vct{x}_s$ denotes the vector consisting of the $s$ largest entries in magnitude of $\vct{x}$.
The number of measurements $m$ is proportional to $(s/\varepsilon^2)$ polylog$(d/\varepsilon)$, and HHS
Pursuit runs in time $(s^2/\varepsilon^4)$polylog$(d/\varepsilon)$. The algorithm
uses working space $(s/\varepsilon^2)$polylog$(d/\varepsilon)$, including storage of
the matrix 	$\Phi$.
    		\end{theorem} 
    		
    	\begin{remark}This theorem presents guarantees that are stronger than those of chaining pursuit. Chaining pursuit, however, still provides a faster runtime.	\end{remark}
    	
    	There are other algorithms such as the Sudocodes algorithm that as of now only work in the noiseless, strictly sparse case.  However, these are still interesting because of the simplicity of the algorithm.  The Sudocodes algorithm is a simple two-phase algorithm.  In the first phase, an easily implemented avalanche bit testing scheme is applied iteratively to recover most of the coordinates of the signal $\vct{x}$.  At this point, it remains to reconstruct an extremely low dimensional signal (one whose coordinates are only those that remain).  In the second phase, this part of the signal is reconstructed, which completes the reconstruction.  Since the recovery is two-phase, the measurement matrix is as well.  For the first phase, it must contain a sparse submatrix, one consisting of many zeros and few ones in each row.  For the second phase, it also contains a matrix whose small submatrices are invertible.  The following result for strictly sparse signals is proved in~\cite{SBB06:Sudocodes}.
    		
    		\begin{theorem}[Sudocodes~\cite{SBB06:Sudocodes}]
    		Let $\vct{x}$ be an $s$-sparse signal in $\R^d$, and let the $m\times d$ measurement matrix $\Phi$ be as described above. Then with $m=O(s\log d)$, the Sudocodes algorithm exactly reconstructs the signal $\vct{x}$ with computational complexity just $O(s\log s\log d)$. 
    		
    		\end{theorem}
    	
    	The Sudocodes algorithm cannot reconstruct noisy signals because of the lack of robustness in the second phase.  However, work on modifying this phase to handle noise is currently being done.  If this task is accomplished Sudocodes would be an attractive algorithm because of its sublinear runtime and simple implementation.

    		\subsubsection[Summary]{Summary}
    		\label{sec:Approaches:Greedy:Highly:Summary}
    		
    		Combinatorial algorithms such as HHS pursuit provide sublinear time recovery with optimal error bounds and optimal number of measurements.  Some of these are straightforward and easy to implement, and others require complicated structures.  The major disadvantage however is the structural requirement on the measurement matrices. Not only do these methods only work with one particular kind of measurement matrix, but that matrix is highly structured which limits its use in practice.  There are no known sublinear methods in compressed sensing that allow for unstructured or generic measurement matrices.

%% file: romp.tex
As is now evident, the two approaches to compressed sensing each presented disjoint advantages and challenges.  While the optimization method provides robustness and uniform guarantees, it lacks the speed of the greedy approach.  The greedy methods on the other hand had not been able to provide the strong guarantees of Basis Pursuit.  This changed when we developed a new greedy algorithm, Regularized Orthogonal Matching Pursuit~\cite{NV07:Uniform-Uncertainty}, that provided the strong guarantees of the optimization method.  This work bridged the gap between the two approaches, and provided the first algorithm possessing the advantages of both approaches.

\subsection[Description]{Description}
    		\label{sec:New:Regularized:Description}
    		
    		Regularized Orthogonal Matching Pursuit (ROMP) is a greedy algorithm, but will correctly recover any sparse signal using any measurement matrix that satisfies the Restricted Isometry Condition~\eqref{eq:RIC}.  Again as in the case of OMP, we will use the observation vector $\Phi^* \Phi \vct{x}$ as a good local approximation to the $s$-sparse signal $\vct{x}$.  Since the Restricted Isometry Condition guarantees that every $s$ columns of $\Phi$ are close to an orthonormal system, we will choose at each iteration not just one coordinate as in OMP, but up to $s$ coordinates using the observation vector.  It will then be okay to choose some incorrect coordinates, so long as the number of those is limited.  To ensure that we do not select too many incorrect coordinates at each iteration, we include a regularization step which will guarantee that each coordinate selected contains an even share of the information about the signal.  The ROMP algorithm can thus be described as follows:
    		
    		\bigskip
   \textsc{Regularized Orthogonal Matching Pursuit (ROMP)~\cite{NV07:Uniform-Uncertainty}}

\nopagebreak

\fbox{\parbox{\algorithmwidth}{
  \textsc{Input:} Measurement matrix $\Phi$, measurement vector $\vct{u}=\Phi \vct{x}$, sparsity level $s$
  
  \textsc{Output:} Index set $I \subset \{1,\ldots,d\}$, reconstructed vector $\hat{\vct{x}} = \vct{w}$

  \textsc{Procedure:}
  \begin{description}
    \item[Initialize] Let the index set $I = \emptyset$ and the residual $\vct{r} = \vct{u}$.\\
      Repeat the following steps until $\vct{r} = 0$:
    \item[Identify] Choose a set $J$ of the $s$ biggest coordinates in magnitude 
      of the observation vector $\vct{y} = \Phi^*\vct{r}$, or all of its nonzero coordinates, 
      whichever set is smaller.
    \item[Regularize] Among all subsets $J_0 \subset J$ with comparable coordinates:
      $$
      |\vct{y}(i)| \leq 2|\vct{y}(j)| \quad \text{for all } i,j \in J_0,
      $$
      choose $J_0$ with the maximal energy $\|\vct{y}|_{J_0}\|_2$.
    \item[Update] Add the set $J_0$ to the index set: $I \leftarrow I \cup J_0$, 
      and update the residual:
      $$
      \vct{w} = \argmin_{\vct{z} \in \R^I} \|\vct{u} - \Phi \vct{z}\|_2; \qquad \vct{r} = \vct{u} - \Phi \vct{w}.
      $$
  \end{description}
 }}
 
 \bigskip 
 
 \begin{remarks}
 
 {\bf 1. }  We remark here that knowledge about the sparsity level $s$ is required in ROMP, as in OMP.  There are several ways this information may be obtained.  Since the number of measurements $m$ is usually chosen to be $O(s\log d)$, one may then estimate the sparsity level $s$ to be roughly $m/\log d$.  An alternative approach would be to run ROMP using various sparsity levels and choose the one which yields the least error $\|\Phi \hat{\vct{x}} - \Phi \vct{x}\|$ for outputs $\hat{\vct{x}}$. Choosing testing levels out of a geometric progression, for example, would not contribute significantly to the overall runtime.
 
 {\bf 2. }  Clearly in the case where the signal is not exactly sparse and the signal and measurements are corrupted with noise, the algorithm as described above will never halt.  Thus in the noisy case, we simply change the halting criteria by allowing the algorithm iterate at most $s$ times, or until $|I| \geq s$.  We show below that with this modification ROMP approximately reconstructs arbitrary signals.  
 \end{remarks}
    		
    		\subsection[Main Theorems]{Main Theorems}
    		\label{sec:New:Regularized:Main}
    		
    	In this section we present the main theorems for ROMP.  We prove these theorems in Section~\ref{sec:New:Regularized:Proof}.	When the measurement matrix $\Phi$ satisfied the Restricted Isometry Condition, ROMP exactly recovers all sparse signals.  This is summarized in the following theorem from~\cite{NV07:Uniform-Uncertainty}.
    		
    		\begin{theorem}[Exact sparse recovery via ROMP~\cite{NV07:Uniform-Uncertainty}]\label{T:main}
  Assume a measurement matrix $\Phi$ satisfies the Restricted Isometry Condition 
  with parameters $(2s, \e)$ for $\e = 0.03 / \sqrt{\log s}$. 
  Let $\vct{x}$ be an $s$-sparse vector in $\R^d$ with measurements $\vct{u} = \Phi \vct{x}$. 
  Then ROMP in at most $s$ iterations outputs a set $I$ such that 
  $$
  \supp(x) \subset I \quad \text{and} \quad |I| \leq 2s.
  $$
\end{theorem}

\begin{remarks}
  {\bf 1. } Theorem~\ref{T:main} shows that ROMP provides {\em exact recovery} of sparse signals.
  Using the index set $I$, one can compute the signal $\vct{x}$ from its measurements $\vct{u} = \Phi \vct{x}$ 
  as $\vct{x} = (\Phi_I)^{-1} \vct{u}$,
  where $\Phi_I$ denotes the measurement matrix $\Phi$ restricted to the columns
  indexed by $I$.

  {\bf 2.} Theorem~\ref{T:main} provides {\em uniform guarantees} of sparse recovery, meaning that once
  the measurement matrix $\Phi$ satisfies the Restricted Isometry Condition,
  ROMP recovers {\em every} sparse signal from its measurements. 
  Uniform guarantees such as this are now known to be impossible for OMP~\cite{R08:Impossibility}, and finding a version of OMP providing uniform guarantees was previously
  an open problem~\cite{TG07:Signal-Recovery}. 
  Theorem~\ref{T:main} shows that ROMP solves this problem.
    
  {\bf 3. } Recall from Section~\ref{sec:Approaches:Basis:Restricted} that random Gaussian, Bernoulli and partial Fourier matrices
  with number of measurements $m$ almost linear in the sparsity $s$, satisfy the Restricted Isometry Condition.
  It is still unknown whether OMP works at all with partial Fourier
  measurements, but ROMP gives sparse recovery 
  for these measurements, and with uniform guarantees. 
  
  {\bf 4. } In Section~\ref{sec:New:Regularized:Implementation} we explain how the identification and regularization steps of ROMP can easily be performed efficiently. 
  In Section ~\ref{sec:New:Regularized:Implementation} we show that the running time of ROMP is comparable to that of OMP in theory, and is better in practice.
  
\end{remarks}

Theorem~\ref{T:main} shows ROMP works correctly for signals which are exactly sparse.  However, as mentioned before, ROMP also performs well for signals and measurements which are corrupted with noise.  This is an essential property for an algorithm to be realistically used in practice.  The following theorem from~\cite{NV07:ROMP-Stable} shows that ROMP approximately reconstructs sparse signals with noisy measurements.  Corollary~\ref{T:stabsig} shows that ROMP also approximately reconstructs \textit{arbitrary} signals with noisy measurements.

\begin{theorem}[Stability of ROMP under measurement perturbations~\cite{NV07:ROMP-Stable}]\label{T:stability}
  Let $\Phi$ be a measurement matrix satisfying the Restricted Isometry Condition 
  with parameters $(4s, \e)$ for $\e = 0.01 / \sqrt{\log s}$. 
  Let $\vct{x}\in\R^d$ be an $s$-sparse vector. 
  Suppose that the measurement vector $\Phi \vct{x}$ becomes corrupted, so that we consider
  $\vct{u} = \Phi \vct{x} + \vct{e}$ where $\vct{e}$ is some error vector. 
  Then ROMP produces a good approximation $\hat{\vct{x}}$ to $\vct{x}$:
  $$
  \|\vct{x} - \hat{\vct{x}}\|_2 \leq 104 \sqrt{\log s}\|e\|_2.
  $$  
\end{theorem}

\begin{corollary}[Stability of ROMP under signal perturbations~\cite{NV07:ROMP-Stable}]\label{T:stabsig}
  Let $\Phi$ be a measurement matrix satisfying the Restricted Isometry Condition 
  with parameters $(8s, \e)$ for $\e = 0.01 / \sqrt{\log s}$. 
  Consider an arbitrary vector $\vct{x}$ in $\R^d$.
 Suppose that the measurement vector $\Phi \vct{x}$ becomes corrupted, 
  so we consider $\vct{u} = \Phi \vct{x} + \vct{e}$ where $\vct{e}$ is some error vector. 
  Then ROMP produces a good approximation $\hat{\vct{x}}$ to $\vct{x}_{2s}$:
  \begin{equation}\label{boundsig}
    \|\hat{\vct{x}} - \vct{x}_{2s}\|_2 
    \leq 159 \sqrt{\log 2s} \Big( \|e\|_2 + \frac{\|\vct{x}-\vct{x}_s\|_1}{\sqrt{s}} \Big).
  \end{equation}  
\end{corollary}

\begin{remarks}

{\bf 1. } In the noiseless case, Theorem~\ref{T:stability} coincides with Theorem~\ref{T:main} in showing exact recovery.

{\bf 2. } Corollary~\ref{T:stabsig} still holds (with only the constants changed) when the term $\vct{x}_{2s}$ is replaced by $\vct{x}_{(1+\delta)s}$ for any $\delta > 0$.  This is evident by the proof of the corollary given below.

  {\bf 2. } Corollary~\ref{T:stabsig} also implies the following bound on the entire signal $\vct{x}$: 
    \begin{equation}\label{vbound}
      \|\hat{\vct{x}} - \vct{x}\|_2 
      \leq 160 \sqrt{\log 2s} \Big( \|e\|_2 + \frac{\|\vct{x}-\vct{x}_s\|_1}{\sqrt{s}} \Big).
    \end{equation}
    Indeed, we have
    \begin{align*}
    \|\hat{\vct{x}} - \vct{x}\|_2 &\leq \|\hat{\vct{x}} - \vct{x}_{2s}\|_2 + \|\vct{x} - \vct{x}_{2s}\|_2\\
    &\leq 159 \sqrt{\log 2s} \Big( \|e\|_2 + \frac{\|\vct{x}-\vct{x}_s\|_1}{\sqrt{s}} \Big) + \|(\vct{x} - \vct{x}_s) - (\vct{x} - \vct{x}_s)_s\|\\
    &\leq 159 \sqrt{\log 2s} \Big( \|e\|_2 + \frac{\|\vct{x}-\vct{x}_s\|_1}{\sqrt{s}} \Big) + \frac{\|\vct{x}-\vct{x}_s\|_1}{\sqrt{s}}\\
    &\leq 160 \sqrt{\log 2s} \Big( \|e\|_2 + \frac{\|\vct{x}-\vct{x}_s\|_1}{\sqrt{s}} \Big),
    \end{align*}
    where the first inequality is the triangle inequality, the second uses Corollary~\ref{T:stabsig} and the identity $\vct{x} - \vct{x}_{2s} = (\vct{x} - \vct{x}_s) - (\vct{x} - \vct{x}_s)_s$, and third uses Lemma~\ref{L:ve} below.
      
  {\bf 3. } The error bound for Basis Pursuit given in Theorem~\ref{stableBP2}, is similar except for the logarithmic factor.  We again believe this to be an artifact of our proof, and our empirical results in Section~\ref{sec:New:Regularized:Numerical} show that ROMP indeed provides much better results than the corollary suggests.

  {\bf 4. } In the case of noise with Basis Pursuit, the problem~\eqref{eqn:bp} needs to be solved, which requires knowledge about the noise vector $\vct{e}$. ROMP requires no such knowledge.
  
  {\bf 5. } If instead one wished to compute a $2s$-sparse approximation to the signal, one may just retain the $2s$ largest coordinates of the reconstructed vector $\hat{\vct{x}}$.  In this case Corollary~\ref{T:stabsig} implies the following:
  \begin{corollary}\label{C:napprox}Assume a measurement matrix $\Phi$ satisfies the Restricted Isometry Condition 
  with parameters $(8s, \e)$ for $\e = 0.01 / \sqrt{\log s}$. 
  Then for an arbitrary vector $\vct{x}$ in $\R^d$, 

    $$
    \|\vct{x}_{2s} - \hat{\vct{x}}_{2s}\|_2 \leq 477 \sqrt{\log 2s}\Big( \|\vct{e}\|_2 + \frac{\|\vct{x}-\vct{x}_{s}\|_1}{\sqrt{s}}\Big).
    $$
\end{corollary}
		This corollary is proved in Section~\ref{sec:New:Regularized:Proof}.
		
  {\bf 6. } As noted earlier, a special class of signals are compressible signals~\eqref{comp}, and for these it is straightforward to see that \eqref{vbound} gives us the following error bound for ROMP:
  $$
  \|\vct{x} - \hat{\vct{x}}\|_2 \le C'_q \frac{\sqrt{\log s}}{s^{q-1/2}} + C''\sqrt{\log s}\|e\|_2.
  $$
   As observed in~\cite{CRT06:Stable}, this bound is optimal (within the logarithmic
    factor), meaning no algorithm can perform fundamentally better.
    
    \end{remarks}

    			\subsection[Proofs of Theorems]{Proofs of Theorems}
    		\label{sec:New:Regularized:Proof}
    		In this section we include the proofs of Theorems~\ref{T:main} and~\ref{T:stability} and Corollaries~\ref{T:stabsig} and~\ref{C:napprox}.  The proofs presented here originally appeared in~\cite{NV07:Uniform-Uncertainty} and~\cite{NV07:ROMP-Stable}.
    		
    		    		\subsubsection[Proof of Theorem~\ref{T:main}]{Proof of Theorem~\ref{T:main}}\label{sec:New:Regularized:Proof:orig}
    		
    		\input{rompProofTmain}
    		
    		\subsubsection[Proof of Theorem~\ref{T:stability}]{Proof of Theorem~\ref{T:stability}}\label{sec:New:Regularized:Proof:main}
    		
    		\input{rompProofstab1}
    		
    		    		\subsubsection[Proof of Corollary~\ref{T:stabsig}]{Proof of Corollary~\ref{T:stabsig}}\label{sec:New:Regularized:Proof:cor1}
    		    		\input{rompProofCor1}
    		    		
    		    		    		    		    		\subsubsection[Proof of Corollary~\ref{C:napprox}]{Proof of Corollary~\ref{C:napprox}}\label{sec:New:Regularized:Proof:cor2}
    		    		    		    		    		\input{rompProofCor2}

    		\subsection[Implementation and Runtime]{Implementation and Runtime}
    		\label{sec:New:Regularized:Implementation}
    		
    		The Identification step of ROMP, i.e. selection of the subset $J$, 
can be done by {\em sorting} the coordinates of $y$ in the nonincreasing order 
and selecting $s$ biggest.
Many sorting algorithms such as Mergesort or Heapsort provide running times 
of $O(d\log d)$. 

The Regularization step of ROMP, i.e. selecting $J_0 \subset J$, 
can be done fast by observing that $J_0$ is an {\em interval} 
in the decreasing rearrangement of coefficients. 
Moreover, the analysis of the algorithm shows that instead of 
searching over all intervals $J_0$, it suffices
to look for $J_0$ among {\em $O(\log s)$ consecutive intervals}
with endpoints where the magnitude of coefficients decreases by a factor of $2$.
(these are the sets $A_k$ in the proof of Lemma~\ref{L:reg}).
Therefore, the Regularization step can be done in time $O(s)$.

In addition to these costs, the $k$-th iteration step of ROMP involves
{\em multiplication} of the $d \times m$ matrix $\Phi^*$ by a vector, 
and solving the {\em least squares problem} with the $d \times |I|$ matrix 
$\Phi_I$, where $|I| \le 2s$. 
For unstructured matrices, these tasks can be done in time 
$dm$ and $O(s^2 m)$ respectively~\cite{Bjo96:Numerical-Methods}. 
Since the submatrix of $\Phi$ when restricted 
to the index set $I$ is near an isometry, using an iterative method 
such as the Conjugate Gradient Method allows us to solve the least 
squares method in a constant number of iterations (up to a specific 
accuracy)~\cite[Sec.~7.4]{Bjo96:Numerical-Methods}. Using such a method then reduces the time of solving the
least squares problem to just $O(sm)$. Thus in the cases where ROMP
terminates after a fixed number of iterations, the total time to solve
all required least squares problems would be just $O(sm)$. 
For structured matrices, such as
partial Fourier, these times can be improved even more using fast multiply
techniques. 

In other cases, however, ROMP may need more than a constant number
of iterations before terminating, say the full $O(s)$ iterations. In this case, 
it may be more efficient to maintain the QR factorization of $\Phi_I$ and use the Modified 
Gram-Schmidt algorithm. With this method, solving all the least squares problems
takes total time just $O(s^2 m).$ However, storing the QR factorization is quite costly, so
in situations where storage is limited it may be best to use the iterative methods mentioned above.

ROMP terminates in at most $2s$ iterations. Therefore, for unstructured 
matrices using the methods mentioned above and in the interesting regime $m \ge \log d$, 
{\em the total running time of ROMP is O(dNn)}. This is the same bound as for OMP
\cite{TG07:Signal-Recovery}. 
    		
    		\subsection[Numerical Results]{Numerical Results}
    		\label{sec:New:Regularized:Numerical}
    		
    		\subsubsection[Noiseless Numerical Studies]{Noiseless Numerical Studies}
    		
    		This section describes our experiments that illustrate the signal recovery power of ROMP, as shown in~\cite{NV07:Uniform-Uncertainty}.  See Section~\ref{app:code:ROMPcode} for the Matlab code used in these studies.
We experimentally examine how many measurements $m$ are necessary to recover various kinds of $s$-sparse
signals in $\R^d$ using ROMP. 
We also demonstrate that the number of iterations ROMP needs to recover a sparse signal is
in practice at most linear the sparsity. 

First we describe the setup of our experiments. For many values of the ambient dimension $d$, 
the number of measurements $m$, and the sparsity $s$, we reconstruct random signals using ROMP.
For each set of values, we generate an $m \times d$ Gaussian measurement matrix $\Phi$ and then perform $500$ independent trials. The results
we obtained using Bernoulli measurement matrices were very similar.
In a given trial, we generate an $s$-sparse signal $x$ in one of two ways. In either case, we first select the support of the signal by choosing $s$ components uniformly 
at random (independent from the measurement matrix $\Phi$). In the cases where we wish to generate flat signals, we then  set these components to one. Our work as well as the analysis of Gilbert and Tropp~\cite{TG07:Signal-Recovery}
show that this is a challenging case for ROMP (and OMP). In the cases where we wish to generate sparse compressible signals, we set the $i^{th}$ component of the support to plus or minus $i^{-1/p}$ for a specified value of $0 < p < 1$. We then execute ROMP with the measurement
vector $u = \Phi x$.  

Figure~\ref{fig:percent} depicts the percentage (from the $500$ trials) of sparse flat signals that were 
reconstructed exactly. This plot was generated with $d = 256$ for various levels of sparsity $s$. The horizontal axis represents the number of measurements $m$, and the vertical
axis represents the exact recovery percentage. We also performed this same test for sparse compressible signals and found the results very similar to those in Figure~\ref{fig:percent}.
Our results show that performance of ROMP is very similar to that of OMP which can be found in \cite{TG07:Signal-Recovery}. 


Figure~\ref{fig:99} depicts a plot of the values for $m$ and $s$ at which $99\%$ of sparse flat signals are recovered exactly. This plot was generated with $d=256$. The horizontal axis represents the number of measurements $m$, and the vertical axis the sparsity level $s$. 


Theorem~\ref{T:main} guarantees that ROMP runs with at most $O(s)$ iterations. Figure~\ref{fig:itsROMP} depicts the number of iterations executed by ROMP for $d=10,000$ and $m=200$. ROMP was 
executed under the same setting as described above for sparse flat signals as well as sparse compressible signals for various values of $p$, and the number of iterations
in each scenario was averaged over the $500$ trials. These averages were plotted against
the sparsity of the signal. As the plot illustrates, only $2$
iterations were needed for flat signals even for sparsity $s$ as high as $40$. The plot also demonstrates that the number of iterations needed for sparse compressible
is higher than the number needed for sparse flat signals, as one would expect. The plot suggests that for smaller
values of $p$ (meaning signals that decay more rapidly) ROMP needs more iterations. However it shows that even in the case of $p=0.5$, only $6$ iterations are needed even for sparsity $s$ as high as $20$.

\begin{figure}[ht] 
  \includegraphics[width=0.8\textwidth,height=3.2in]{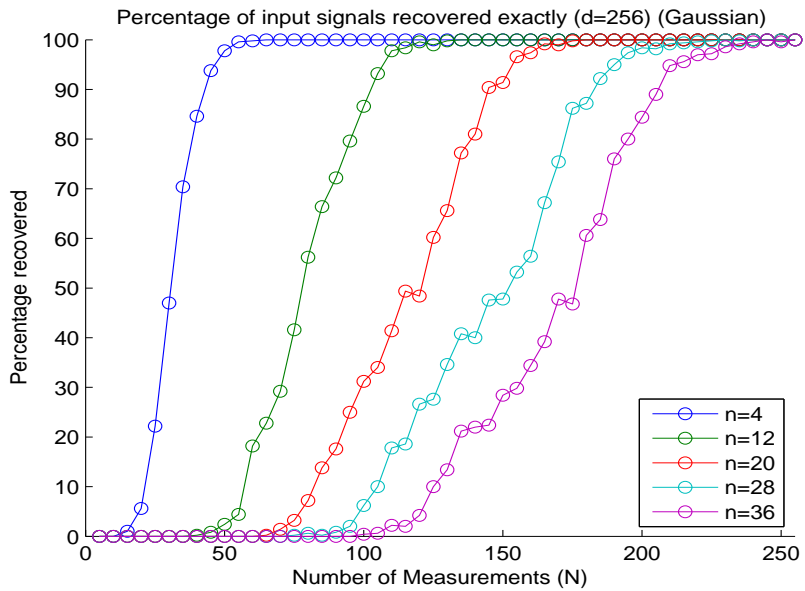}
  \caption{The percentage of sparse flat signals exactly recovered by ROMP as a function of the number of measurements in dimension $d=256$ for various levels of sparsity.}\label{fig:percent}
\end{figure}

\begin{figure}[ht] 
  \includegraphics[width=0.8\textwidth,height=3.2in]{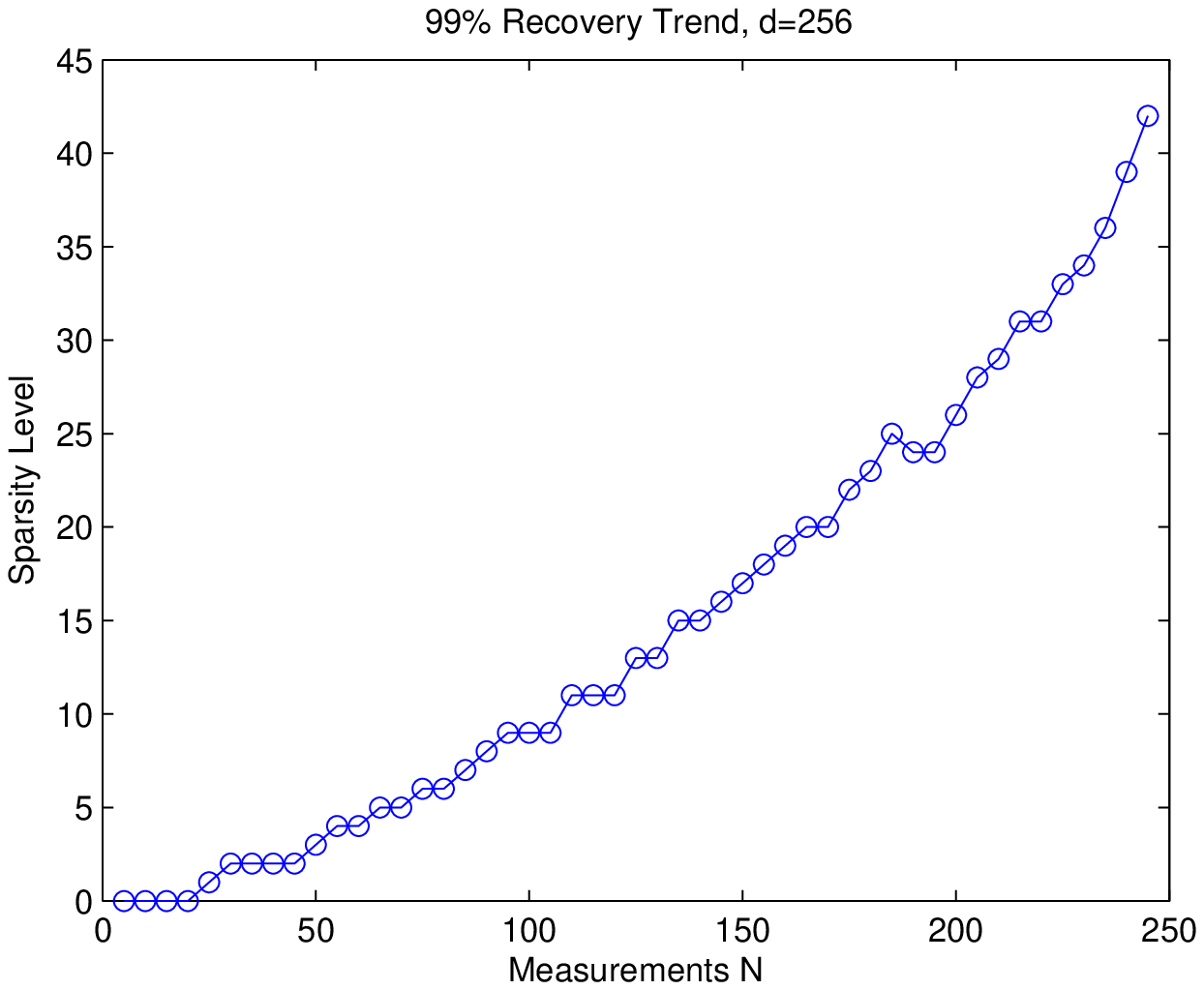}
  \caption{The $99\%$ recovery limit as a function of the sparsity and the number of measurements for sparse flat signals.}\label{fig:99}
\end{figure}

\begin{figure}[ht] 
  \includegraphics[width=0.8\textwidth,height=3.2in]{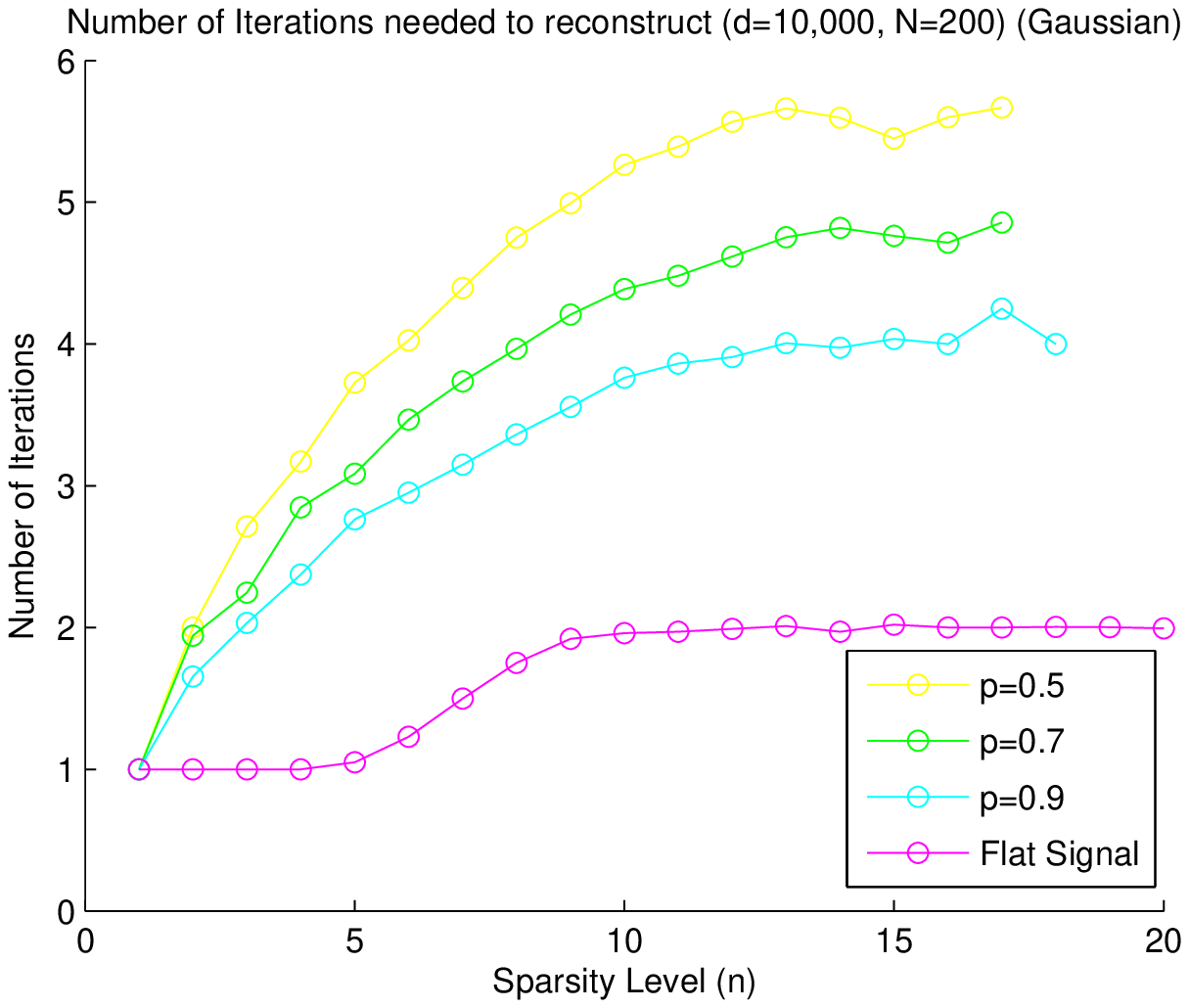}
  \caption{The number of iterations executed by ROMP as a function of the sparsity in dimension $d=10,000$ with $200$ measurements.}\label{fig:itsROMP}
\end{figure}
    		
    		\subsubsection[Noisy Numerical Studies]{Noisy Numerical Studies}
    		
    		This section describes our numerical experiments that illustrate the stability of ROMP as shown in~\cite{NV07:ROMP-Stable}.
We study the recovery error using ROMP for both perturbed
measurements and signals. 
The empirical recovery error is actually much better than that given in the theorems. 

First we describe the setup to our experimental studies. We run ROMP on various values of the ambient dimension $d$, 
the number of measurements $m$, and the sparsity level $s$, and attempt to reconstruct random signals.
For each set of parameters, we perform $500$ trials. Initially, we generate an $m \times d$ Gaussian measurement matrix $\Phi$. For each trial, independent of the matrix, we generate an $s$-sparse signal $x$ by choosing $s$ components uniformly 
at random and setting them to one.
In the case of perturbed signals, we add to the signal a $d$-dimensional error vector with Gaussian entries. In the case of perturbed measurements, we add an $m$-dimensional error vector with Gaussian entries to the measurement vector $\Phi x$.  We then execute ROMP with the measurement
vector $u = \Phi x$ or $u + e$ in the perturbed measurement case. After ROMP terminates, we output the reconstructed vector $\hat{x}$ obtained from the least squares calculation and calculate its distance from the original signal. 

Figure~\ref{fig:meas2} depicts the recovery error $\|x - \hat{x}\|_2$ when ROMP was run with perturbed measurements. This plot was generated with $d = 256$ for various levels of sparsity $s$. The horizontal axis represents the number of measurements $m$, and the vertical
axis represents the average normalized recovery error. Figure~\ref{fig:meas2} confirms the results of Theorem~\ref{T:stability}, while also suggesting the bound may be improved by removing the $\sqrt{\log s}$ factor.

Figure~\ref{fig:sig4} depicts the normalized recovery error when the signal was perturbed by a Gaussian vector. The figure confirms the results of Corollary~\ref{T:stabsig} while also suggesting again that the logarithmic factor in the corollary is unnecessary. 


\begin{figure}[ht] 
  \includegraphics[width=0.8\textwidth,height=3.2in]{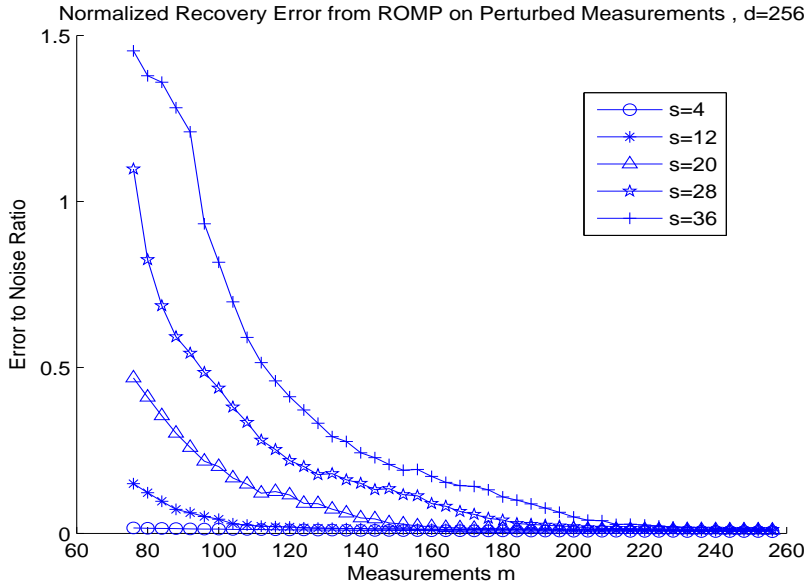}
  \caption{The error to noise ratio $\frac{\|\hat{x} - x\|_2}{\|e\|_2}$ as a function of the number of measurements $m$ in dimension $d=256$ for various levels of sparsity $s$.}\label{fig:meas2}
\end{figure}

\begin{figure}[ht] 
  \includegraphics[width=0.8\textwidth,height=3.2in]{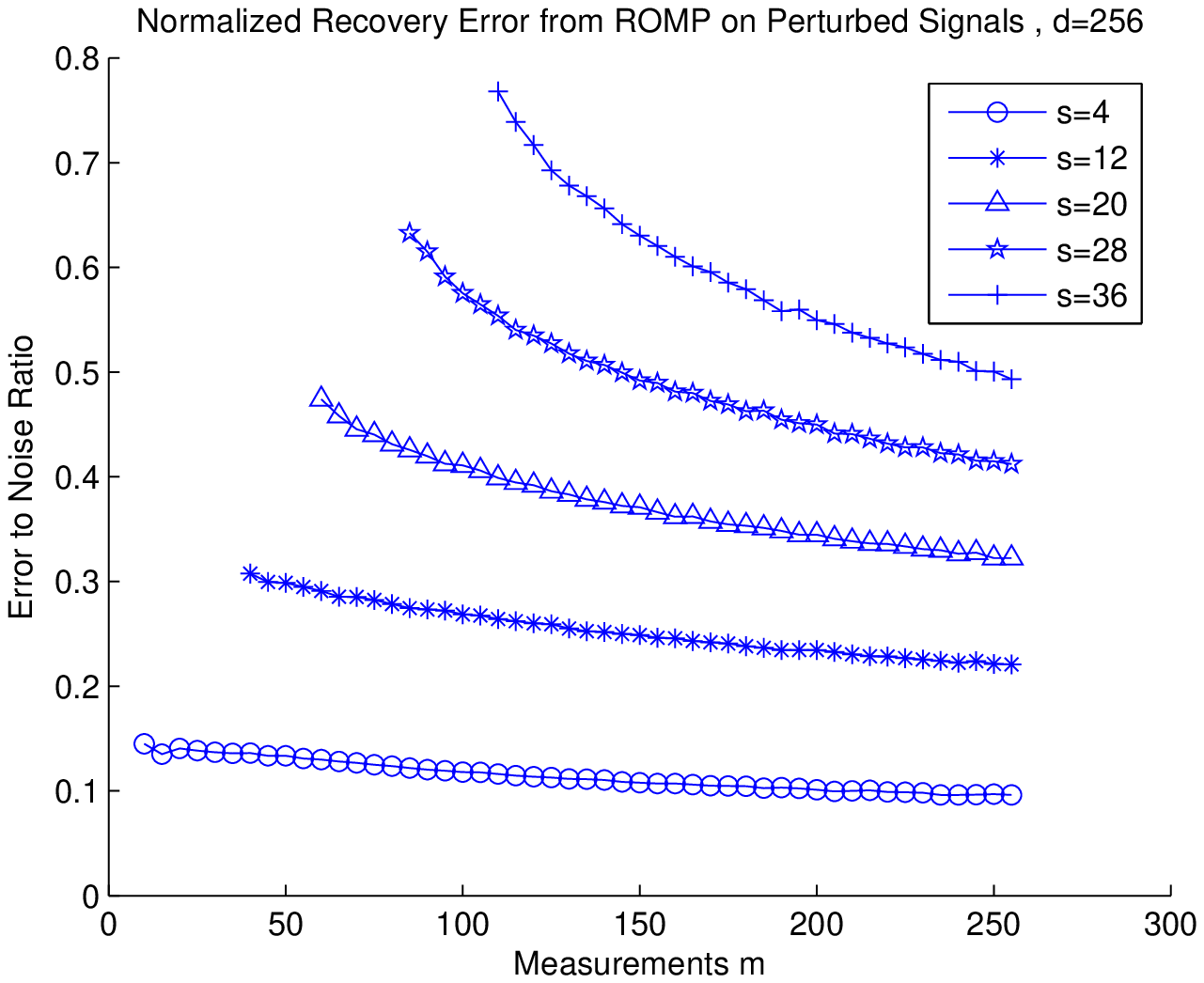}
  \caption{The error to noise ratio $\frac{\|\hat{x} - x_{2s}\|_2}{\|x-x_s\|_1/\sqrt{s}}$ using a perturbed signal, as a function of the number of measurements $m$ in dimension $d=256$ for various levels of sparsity $s$.}\label{fig:sig4}
\end{figure}
    		
    		\subsection[Summary]{Summary}
    		\label{sec:New:Regularized:Summary}
    		
    		There are several critical properties that an ideal algorithm in compressed sensing should possess. One such property is stability, guaranteeing that under small perturbations of the inputs, the algorithm still performs approximately correct. Secondly, the algorithm needs to provide uniform guarantees, meaning that with high probability the algorithm works correctly for \textit{all} inputs. Finally, to be ideal in practice, the algorithm would need to have a fast runtime. The $\ell_1$-minimization approach to compressed sensing is stable and provides uniform guarantees, but since it relies on the use of Linear Programming, lacks a strong bound on its runtime. The greedy approach is quite fast both in theory and in practice, but had lacked both stability and uniform guarantees. We analyzed the restricted isometry property in a unique way and found consequences that could be used in a greedy fashion. Our breakthrough algorithm ROMP is the first to provide all these benefits (stability, uniform guarantees, and speed), and essentially bridges the gap between the two major approaches in compressed sensing.  

%% file: rompProofTmain.tex
We shall prove a stronger version of Theorem~\ref{T:main}, which states that 
{\em at every iteration} of ROMP, at least $50\%$ of the newly selected coordinates
are from the support of the signal $x$. 

\begin{theorem}[Iteration Invariant of ROMP] \label{T:it}
  Assume $\Phi$ satisfies the Restricted Isometry Condition 
  with parameters $(2s, \e)$ for $\e = 0.03 / \sqrt{\log s}$. 
  Let $x \ne 0$ be a $s$-sparse vector with measurements $u = \Phi x$. 
  Then at any iteration of ROMP, after the regularization step, we have
  $J_0 \ne \emptyset$, $J_0 \cap I = \emptyset$ and 
  \begin{equation}                  \label{J support}
    |J_0 \cap \supp(x)| \geq \frac{1}{2}|J_0|.
  \end{equation}
  In other words, at least $50\%$ of the coordinates in the newly selected set $J_0$ 
  belong to the support of $x$.
\end{theorem}

In particular, at every iteration ROMP finds at least one new coordinate 
in the support of the signal $x$. Coordinates outside the support can also be found, 
but \eqref{J support} guarantees that the number of such ``false'' coordinates 
is always smaller than those in the support. 
This clearly implies Theorem~\ref{T:main}.

\medskip

Before proving Theorem~\ref{T:it} we explain how the Restricted Isometry Condition 
will be used in our argument. RIC is necessarily a local principle, which concerns 
not the measurement matrix $\Phi$ as a whole, but its submatrices of $s$ columns. 
All such submatrices $\Phi_I$, $I \subset \{1,\ldots,d\}$, $|I| \le s$ are almost isometries.
Therefore, for every $s$-sparse signal $x$, the observation vector 
$y = \Phi^*\Phi x$ approximates $x$ locally, when restricted to a set of 
cardinality $s$. The following proposition formalizes these local properties
of $\Phi$ on which our argument is based.

\begin{proposition}[Consequences of Restricted Isometry Condition]\label{P:cons} 
  Assume a measurement matrix $\Phi$ satisfies the Restricted Isometry Condition 
  with parameters $(2s, \e)$. Then the following holds.
  \begin{enumerate}
    \item {\em (Local approximation)} 
      For every $s$-sparse vector $x \in \R^d$ 
      and every set $I \subset \{1, \ldots, d\}$, $|I| \le s$, 
      the observation vector $y = \Phi^* \Phi x$ satisfies 
      $$
      \|y|_I - x|_I\|_2 \le 2.03 \e \|x\|_2.
      $$
    \item {\em (Spectral norm)} 
      For any vector $z \in \R^m$
      and every set $I \subset \{1, \ldots, d\}$, $|I| \le 2s$, we have 
      $$
      \|(\Phi^*z)|_I\|_2 \leq (1+\e)\|z\|_2.
      $$
    \item {\em (Almost orthogonality of columns)} 
      Consider two disjoint sets $I,J \subset \{1, \ldots, d\}$, $|I \cup J| \le 2s$.
      Let $P_I, P_J$ denote the orthogonal projections in $\R^m$
      onto $\range(\Phi_I)$ and $\range(\Phi_J)$, respectively. Then 
      $$
      \|P_I P_J\|_{2\rightarrow 2} \leq 2.2 \e.
      $$
  \end{enumerate}
\end{proposition}

\begin{proof}

{\sc Part 1.}
Let $\Gamma = I\cup\supp(x)$, so that $|\Gamma| \leq 2s$. 
Let $\Idg$ denote the identity operator on $\R^\Gamma$. 
By the Restricted Isometry Condition,
$$
\|\Phi_\Gamma^*\Phi_\Gamma - \Idg\|_{2\rightarrow 2} 
= \sup_{w \in \R^\Gamma, \, \|w\|_2 = 1} \big| \|\Phi_\Gamma w\|_2^2 - \|w\|_2^2 \big| 
\le (1+\e)^2 - 1
\le 2.03\e.
$$
Since $\supp(x) \subset \Gamma$, we have
$$
\|y|_\Gamma - x|_\Gamma\|_2 
= \|\Phi_\Gamma^*\Phi_\Gamma x - \Id_\Gamma x\|_2 
\le 2.03\e \|x\|_2.
$$   
The conclusion of Part~1 follows since $I\subset \Gamma$.

\medskip

{\sc Part 2.} Denote by $Q_I$ the orthogonal projection in $\R^d$ 
onto $\R^I$. Since $|I| \le 2s$, the Restricted Isometry Condition yields
$$    
\|Q_I\Phi^*\|_{2\rightarrow 2} = \|\Phi Q_I\|_{2 \rightarrow 2} \leq 1 + \e.
$$
This yields the inequality in Part 2. 

\medskip

{\sc Part 3.} 
The desired inequality is equivalent to:
$$
\frac{|\langle x, y\rangle|}{\|x\|_2\|y\|_2} \leq 2.2\e \qquad \text{for all } x \in \range(\Phi_I), \; y \in \range(\Phi_J).
$$
Let $K = I \cup J$ so that $|K| \leq 2s$. For any $x \in \range(\Phi_I), y \in \range(\Phi_J)$, there are $a, b$ so that
$$
x = \Phi_K a, \; y = \Phi_K b, \qquad a \in \R^I, \; b \in \R^J.
$$
By the Restricted Isometry Condition,
$$
\|x\|_2 \geq (1-\e)\|a\|_2, \; \|y\|_2 \geq (1-\e)\|b\|_2.
$$
By the proof of Part 2 above and since $\langle a b \rangle = 0$, we have
$$
|\langle x, y\rangle| = |\langle (\Phi_K^* \Phi_K - \Idg)a, b \rangle| \leq 2.03\e\|a\|_2\|b\|_2.
$$
This yields  
$$
\frac{|\langle x, y\rangle|}{\|x\|_2\|y\|_2} \leq \frac{2.03\e}{(1-\e)^2} \leq 2.2\e,
$$
which completes the proof.
\end{proof}

\bigskip

We are now ready to prove Theorem \ref{T:it}.

\medskip

The proof is by induction on the iteration of ROMP.
The induction claim is that for all previous iterations, the set of newly chosen 
indices $J_0$ is nonempty, disjoint from the set of previously chosen indices $I$, 
and \eqref{J support} holds.
 
Let $I$ be the set of previously chosen indices at the start of a given iteration.
The induction claim easily implies that 
\begin{equation}					\label{suppv I}
  |\supp(x) \cup I| \le 2s.    
\end{equation}
Let $J_0$, $J$, be the sets found by ROMP in the current iteration. 
By the definition of the set $J_0$, it is nonempty.

Let $r \ne 0$ be the residual at the start of this iteration. 
We shall approximate $r$ by a vector in $\range(\Phi_{\supp(x) \setminus I})$.
That is, we want to approximately realize the residual $r$  
as measurements of some signal which lives on the still unfound
coordinates of the the support of $x$. 
To that end, we consider the subspace
$$
H := \range (\Phi_{\supp(x) \cup I})
$$
and its complementary subspaces
$$
F := \range (\Phi_I), \quad 
E_0 := \range (\Phi_{\supp(x) \setminus I}).
$$
The Restricted Isometry Condition in the form of Part~3 of Proposition~\ref{P:cons} 
ensures that $F$ and $E_0$ are almost orthogonal. Thus $E_0$ is close to 
the orthogonal complement of $F$ in $H$,
$$
E := F^{\perp}\cap H.
$$
\begin{center}
\begin{figure}[ht] \label{fig:cap}
  \includegraphics[scale=0.6]{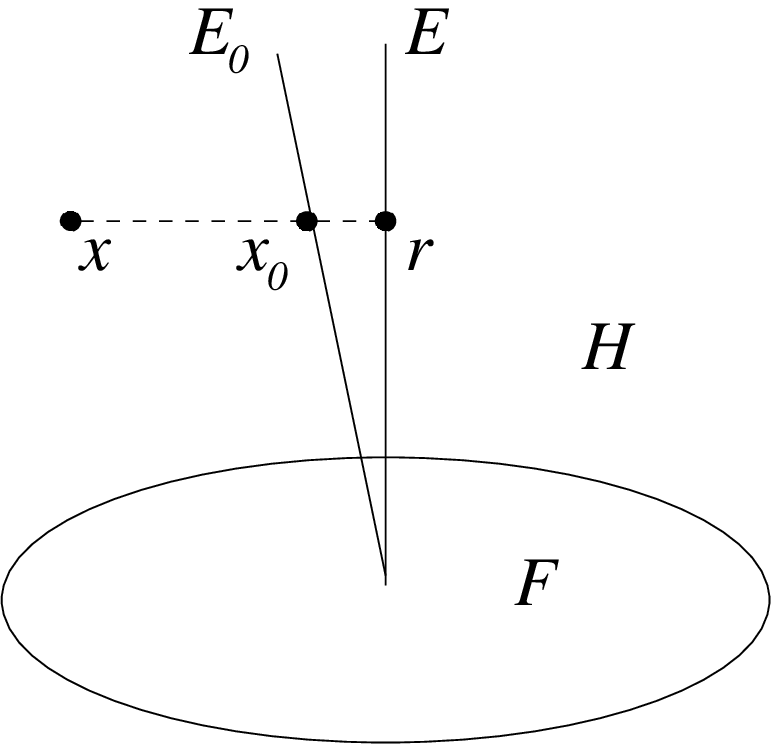}
\end{figure}
\end{center}

We will also consider the signal we seek to identify at the current iteration, 
its measurements, and its observation vector:
\begin{equation}            \label{v0 x0}
  x_0 := x|_{\supp(x) \setminus I}, \quad 
  u_0 := \Phi x_0 \in E_0, \quad y_0 := \Phi^*u_0.
\end{equation}

Lemma~\ref{L:uj} will show that $\|(y-y_0)|_T\|_2$ for any small enough subset $T$ is small, and Lemma~\ref{C:uj0} will show that $\|y|_{J_0}\|_2$ is not too small. First, we show that the residual $r$ has a simple description:

\begin{lemma}[Residual]     \label{residual}
  Here and thereafter, let $P_L$ denote the orthogonal projection in $\R^m$ 
  onto a linear subspace $L$. Then
  $$
  r = P_E u.
  $$
\end{lemma}

\begin{proof}
By definition of the residual in the algorithm,
$r = P_{F^\perp} u$. Since $u \in H$, we conclude from the orthogonal 
decomposition $H = F + E$ that $u = P_F u + P_E u$. Thus
$r = u - P_F u = P_E u$.
\end{proof}

To guarantee a correct identification of $x_0$, we first state
two approximation lemmas that reflect in two different ways the fact 
that subspaces $E_0$ and $E$ are close to each other.
This will allow us to carry over information from $E_0$ to $E$.

\begin{lemma}[Approximation of the residual]\label{C:proj}
  We have
  $$
  \|u_0 - r\|_2 \leq 2.2 \e \|u_0\|_2.
  $$
\end{lemma}

\begin{proof}
By definition of $F$, we have 
$u - u_0 = \Phi(x - x_0) \in F$. 
Therefore, by Lemma~\ref{residual},
$r = P_{E}u = P_{E}u_0$, and so
$$
u_0 - r = u_0 - P_Eu_0 = P_Fu_0 = P_FP_{E_0}u_0.
$$
Now we use Part 3 of Proposition~\ref{P:cons} for the sets $I$ and $\supp(x) \setminus I$
whose union has cardinality at most $2s$ by \eqref{suppv I}. It follows that
$\|P_FP_{E_0}u_0\|_2 \le 2.2 \e \|u_0\|_2$ as desired.
\end{proof}

\begin{lemma}[Approximation of the observation]\label{L:uj}
  Consider the observation vectors 
  $y_0 = \Phi^*u_0$ and $y = \Phi^*r$. Then for any set $T \subset \{1, \ldots, d\}$ with $|T| \le 2s$, we have
  $$
  \|(y_0 - y)|_T\|_2 \leq 2.4 \e \|x_0\|_2.
  $$
\end{lemma}

\begin{proof}
Since $u_0 = \Phi x_0$, we have by Lemma~\ref{C:proj} 
and the Restricted Isometry Condition that
$$
\|u_0 - r\|_2 
\le 2.2 \e \|\Phi x_0\|_2 
\le 2.2 \e (1+\e) \|x_0\|_2 
\le 2.3 \e \|x_0\|_2.
$$
To complete the proof, it remains to apply Part 2 of Proposition~\ref{P:cons},
which yields 
$\|(y_0 - y)|_T\|_2 \le (1 + \e)\|u_0 - r\|_2$.
\end{proof}

We next show that the energy (norm) of $y$ when restricted to $J$, and furthermore to 
$J_0$, is not too small. By the approximation lemmas, this will yield that ROMP 
selects at least a fixed percentage of energy of the still unidentified part of the signal. 
By the regularization step of ROMP, since all selected coefficients have comparable
magnitudes, we will conclude that not only a portion of energy
but also of the {\em support} is selected correctly. This 
will be the desired conclusion.

\begin{lemma}[Localizing the energy]\label{C:uj}
  We have $\|y|_J\|_2 \ge 0.8 \|x_0\|_2$.
\end{lemma}

\begin{proof}
Let $S$ = $\supp(x) \setminus I$.
Since $|S| \leq s$, the maximality property of $J$ in the algorithm
implies that 
$$
\|y_0|_J\|_2 \geq \|y_0|_S\|_2.
$$
Furthermore, since $x_0|_S = x_0$, by Part 1 of Proposition~\ref{P:cons} we have
$$
\|y_0|_S\|_2 \geq (1 - 2.03\e)\|x_0\|_2.
$$
Putting these two inequalities together and using
Lemma~\ref{L:uj}, we conclude that
$$
\|y|_J\|_2 \ge (1 - 2.03\e)\|x_0\|_2 - 2.4\e\|x_0\|_2 \ge 0.8 \|x_0\|_2.
$$
This proves the lemma.
\end{proof}

We next bound the norm of $y$ restricted to the smaller set $J_0$. 
We do this by first noticing a general property of regularization:

\begin{lemma}[Regularization]\label{L:reg}
  Let $v$ be any vector in $\R^m$, $m > 1$.
  Then there exists a subset $A \subset \{1, \ldots, m\}$ 
  with comparable coordinates:
  \begin{equation}\label{E:comp}
    |v(i)| \le 2|v(j)| \quad \text{for all $i, j \in A$,}
  \end{equation}
  and with big energy:
  \begin{equation}\label{E:big} 
    \|v|_A\|_2 \ge \frac{1}{2.5\sqrt{\log m}}\|v\|_2.
  \end{equation}
\end{lemma}

\begin{proof}
We will construct at most $O(\log m)$ subsets $A_k$ with comparable coordinates 
as in (\ref{E:comp}), and such that at least one of these sets will have 
large energy as in (\ref{E:big}). 

Let $v = (v_1, \ldots, v_m)$, and consider a partition of $\{1, \ldots, m\}$ 
using sets with comparable coordinates:
$$
A_k := \{i : 2^{-k}\|v\|_2 < |v_i| \leq 2^{-k+1}\|v\|_2 \}, \qquad k=1, 2, \ldots 
$$
Let $k_0 = \left\lceil \log m \right\rceil + 1$, so that 
$|v_i| \leq \frac{1}{m}\|v\|_2$ for all $i \in A_k$, $k > k_0$.
Then the set $U = \bigcup_{k \le k_0} A_k$ contains most of the energy of $v$:
$$
\|v|_{U^c}\|_2 \le \big(m(\frac{1}{m}\|v\|_2)^2\big)^{1/2} = \frac{1}{\sqrt{m}}\|y\|_2
\le \frac{1}{\sqrt{2}}\|y\|_2.
$$
Thus
$$
\big( \sum_{k\leq k_0}\|v|_{A_k}\|_2^2 \big)^{1/2} 
= \|v|_U\|_2 = \big(\|v\|_2^2 - \|v|_{U^c}\|_2^2\big)^{1/2}
\ge \frac{1}{\sqrt{2}} \|v\|_2.
$$
Therefore there exists $k \leq k_0$ such that
$$
\|v|_{A_k}\|_2 \ge \frac{1}{\sqrt{2k_0}} \|v\|_2 
\ge \frac{1}{2.5\sqrt{\log m}}\|v\|_2,
$$
which completes the proof.
\end{proof}

In our context, Lemma~\ref{L:reg} applied to the vector $y|_J$ along with 
Lemma~\ref{C:uj} directly implies:

\begin{lemma}[Regularizing the energy]\label{C:uj0} 
  We have
  $$
  \|y|_{J_0}\|_2 \ge \frac{0.32}{\sqrt{\log s}}\|x_0\|_2.
  $$
\end{lemma}

\medskip

We now finish the proof of Theorem~\ref{T:it}. 

To show the first claim, that $J_0$ is nonempty, we note that 
$x_0 \ne 0$. Indeed, otherwise by \eqref{v0 x0} we have $I \subset \supp(x)$, 
so by the definition of the residual in the algorithm, we would have $r = 0$
at the start of the current iteration, which is a contradiction. 
Then $J_0 \ne \emptyset$ by Lemma~\ref{C:uj0}.

The second claim, that $J_0 \cap I = \emptyset$, is also simple. 
Indeed, recall that by the definition of the algorithm, 
$r = P_{F^\perp} \in F^\perp = (\range(\Phi_I))^\perp$. 
It follows that the observation vector $y = \Phi^* r$  
satisfies $y|_I = 0$. Since by its definition the set $J$ contains only 
nonzero coordinates of $y$ we have $J \cap I = \emptyset$.
Since $J_0 \subset J$, the second claim $J_0 \cap I = \emptyset$ follows. 

The nontrivial part of the theorem is its last claim, inequality \eqref{J support}.
Suppose it fails. Namely, suppose that 
$|J_0 \cap \supp(x)| < \frac{1}{2}|J_0|$, 
and thus
$$
|J_0 \backslash \supp(x)| > \frac{1}{2}|J_0|.
$$
Set $\Lambda = J_0\backslash\supp(x)$. 
By the comparability property of the coordinates in $J_0$ 
and since $|\Lambda| > \frac{1}{2}|J_0|$, there is a fraction of energy 
in $\Lambda$:
\begin{equation}\label{E:ubig} 
  \|y|_{\Lambda}\|_2 > \frac{1}{\sqrt{5}}\|y|_{J_0}\|_2 
  \ge \frac{1}{7\sqrt{\log s}}\|x_0\|_2, 
\end{equation}
where the last inequality holds by Lemma~\ref{C:uj0}.

On the other hand, we can approximate $y$ by $y_0$ as
\begin{equation}                \label{u u0}
  \|y|_{\Lambda}\|_2 
  \le \|y|_{\Lambda} - y_0|_{\Lambda}\|_2 + \|y_0|_{\Lambda}\|_2.
\end{equation}
Since $\L \subset J$ and using Lemma~\ref{L:uj}, we have
$$
\|y|_{\Lambda} - y_0|_{\Lambda}\|_2 \le 2.4\e\|x_0\|_2
$$
Furthermore, by definition \eqref{v0 x0} of $x_0$, we have $x_0|_\Lambda = 0$. 
So, by Part 1 of Proposition~\ref{P:cons}, 
$$
\|y_0|_{\Lambda}\|_2 \le 2.03 \e \|x_0\|_2.
$$
Using the last two inequalities in \eqref{u u0}, we conclude that 
$$
\|y|_{\Lambda}\|_2 \le 4.43 \e \|x_0\|_2.
$$
This is a contradiction to~(\ref{E:ubig}) 
so long as $\e \leq
 0.03 /  \sqrt{\log s}$. 
This proves Theorem~\ref{T:it}.
\qed

%% file: rompProofstab1.tex
The proof of Theorem~\ref{T:stability} parallels that of Theorem~\ref{T:main}.  We begin by showing that at every iteration of ROMP, either at least $50\%$ of the selected coordinates from that iteration
are from the support of the actual signal $v$, or the error bound already holds. This directly implies Theorem~\ref{T:stability}.

\begin{theorem}[Stable Iteration Invariant of ROMP] \label{T:itSt}
  Let $\Phi$ be a measurement matrix satisfying the Restricted Isometry Condition 
  with parameters $(4s, \e)$ for $\e = 0.01 / \sqrt{\log s}$. 
  Let $x$ be a non-zero $s$-sparse vector with measurements $u = \Phi x + e$. 
  Then at any iteration of ROMP, after the regularization step where $I$ is the current chosen
  index set, we have $J_0 \cap I = \emptyset$ and (at least) one of the following:
  \renewcommand{\labelenumi}{{\normalfont (\roman{enumi})}}
  \begin{enumerate}
  \item \label{J supportSt}
  $|J_0 \cap \supp(v)| \geq \frac{1}{2}|J_0|$;
  \item \label{error} $\|x|_{\supp(x)\backslash I}\|_2 \leq 100 \sqrt{\log s}\|e\|_2$.
  \end{enumerate} 
\end{theorem}

We show that the Iteration Invariant implies Theorem~\ref{T:stability} by examining the three possible cases:

{\bf Case 1: (ii) occurs at some iteration.} We first note that since $|I|$ is nondecreasing, if (ii) occurs at some iteration, then it holds for all subsequent iterations. To show that this would then imply Theorem~\ref{T:stability}, 
we observe that by the Restricted Isometry Condition and since $|\supp(\hat{x})| \leq |I| \leq 3s$,
$$
(1-\e)\|\hat{x}-x\|_2 - \|e\|_2 \leq \|\Phi \hat{x} - \Phi x - e\|_2.
$$

Then again by the Restricted Isometry Condition and definition of $\hat{x}$, 
$$
\|\Phi \hat{x} - \Phi x - e\|_2 \leq \|\Phi (x|_I) - \Phi x - e\|_2 \leq (1 + \e)\|x|_{\supp(x)\backslash I}\|_2+\|e\|_2.
$$
Thus we have that
$$
\|\hat{x}-x\|_2 \leq \frac{1+\e}{1-\e}\|x|_{\supp(x)\backslash I}\|_2 + \frac{2}{1-\e}\|e\|_2.
$$
Thus (ii) of the Iteration Invariant would imply Theorem~\ref{T:stability}. 

{\bf Case 2: (i) occurs at every iteration and $J_0$ is always non-empty.} In this case, by (i) and the fact that $J_0$ is always non-empty, the algorithm identifies at least one element of the support in every iteration. Thus if the algorithm runs $s$ iterations or until $|I| \geq 2s$, it must be that $\supp(x) \subset I$, meaning that $x|_{\supp(x)\backslash I} = 0$.  Then by the argument above for Case 1, this implies Theorem~\ref{T:stability}.

{\bf Case 3: (i) occurs at each iteration and $J_0 = \emptyset$ for some iteration.} By the definition of $J_0$, if $J_0 = \emptyset$ then $y = \Phi^* r = 0$ for that iteration. By definition of $r$, this must mean that 
$$
\Phi^*\Phi(x-w) + \Phi^* e = 0.
$$
This combined with Part 1 of Proposition~\ref{P:cons} below (and its proof, see~\cite{NV07:Uniform-Uncertainty}) applied with the set $I' = \supp(x)\cup I$ yields
$$
\|x-w+(\Phi^* e)|_{I'}\|_2 \leq 2.03\e\|x-w\|_2.
$$
Then combinining this with Part 2 of the same Proposition, we have
$$
\|x-w\|_2 \leq 1.1\|e\|_2.
$$
Since $x|_{\supp(x)\backslash I} = (x-w)|_{\supp(x)\backslash I}$, this means that the error bound (ii) must hold, so by Case 1 
this implies Theorem~\ref{T:stability}.


\medskip

We now turn to the proof of the Iteration Invariant, Theorem~\ref{T:itSt}. 
We prove Theorem~\ref{T:itSt} by inducting on each iteration of ROMP.
We will show that at each iteration the set of chosen 
indices is disjoint from the current set $I$ of indices, 
and that either (i) or (ii) holds. Clearly if (ii) held in 
a previous iteration, it would hold in all future iterations. Thus we may assume that 
(ii) has not yet held. Since (i) has held at each previous iteration, we must have 
\begin{equation}\label{I2n}|I|\leq 2s.\end{equation}
 
Consider an iteration of ROMP, and let $r \ne 0$ be the residual at the start of that iteration. Let $J_0$ and $J$ be the sets found by ROMP in this iteration.
As in \cite{NV07:Uniform-Uncertainty}, we consider the subspace
$$
H := \range (\Phi_{\supp(v) \cup I})
$$
and its complementary subspaces
$$
F := \range (\Phi_I), \quad 
E_0 := \range (\Phi_{\supp(v) \setminus I}).
$$
Part~3 of Proposition~\ref{P:cons} states that the subspaces $F$ and $E_0$ are nearly orthogonal. For this reason we consider the subspace:
$$
E := F^{\perp}\cap H.
$$

First we write the residual $r$ in terms of projections onto these subspaces.

\begin{lemma}[Residual]     \label{residualSt}
  Here and onward, denote by $P_L$ the orthogonal projection in $\R^m$ 
  onto a linear subspace $L$. Then the residual $r$ has the following form:
  $$
  r = P_E \Phi x + P_{F^\perp}e.
  $$
\end{lemma}

\begin{proof}
By definition of the residual $r$ in the ROMP algorithm,
$r = P_{F^\perp} u = P_{F^\perp}(\Phi x + e)$. To complete the proof we need that
$P_{F^\perp}\Phi x = P_E \Phi x$. This follows from the orthogonal decomposition $H=F+E$ and 
the fact that $\Phi x \in H$. 
\end{proof}

Next we examine the missing portion of the signal as well as its measurements:
\begin{equation}            \label{v0 x0St}
  x_0 := x|_{\supp(x) \setminus I}, \quad 
  u_0 := \Phi x_0 \in E_0.
\end{equation}

In the next two lemmas we show that the subspaces $E$ and $E_0$ are indeed close. 

\begin{lemma}[Approximation of the residual]\label{C:projSt}
  Let $r$ be the residual vector and $u_0$ as in~\eqref{v0 x0St}. Then 
  $$
  \|u_0 - r\|_2 \leq 2.2 \e \|u_0\|_2 + \|e\|_2.
  $$
\end{lemma}

\begin{proof}
Since $x - x_0$ has support in $I$, we have 
$\Phi x - u_0 = \Phi(x - x_0) \in F$. 
Then by Lemma~\ref{residual},
$r = P_{E}\Phi x + P_{F^\perp}e = P_{E}u_0 + P_{F^\perp}e$. Therefore,
$$
\|x_0 - r\|_2 = \|x_0 - P_E x_0 - P_{F^\perp}e\|_2 \leq \|P_Fx_0\|_2 + \|e\|_2. 
$$
Note that by \eqref{I2n}, the union of the sets $I$ and $\supp(x) \setminus I$ has cardinality no greater than $3s$. Thus by Part 3 of Proposition~\ref{P:cons}, we have
$$
\|P_Fu_0\|_2 + \|e\|_2 = \|P_F P_{E_0}u_0\|_2 + \|e\|_2 \le 2.2 \e \|u_0\|_2 + \|e\|_2.
$$
\end{proof}

\begin{lemma}[Approximation of the observation]\label{L:ujSt}
  Let $y_0 = \Phi^*u_0$ and $y = \Phi^*r$. Then for any set $T \subset \{1, \ldots, d\}$ with $|T|\leq 3s$,
  $$
  \|(y_0 - y)|_T\|_2 \leq 2.4 \e \|x_0\|_2 + (1+\e)\|e\|_2.
  $$
\end{lemma}

\begin{proof}
By Lemma~\ref{C:projSt} 
and the Restricted Isometry Condition we have
$$
\|u_0 - r\|_2 
\le 2.2 \e \|\Phi x_0\|_2 + \|e\|_2 
\le 2.2 \e (1+\e) \|x_0\|_2 + \|e\|_2 
\le 2.3 \e \|x_0\|_2 + \|e\|_2.
$$
Then by Part 2 of Proposition~\ref{P:cons} we have the desired result,
$$\|(y_0 - y)|_T\|_2 \le (1 + \e)\|u_0 - r\|_2.$$
\end{proof}

The result of the theorem requires us to show that we correctly gain a portion of the support of the signal $x$. To this end, we first show that ROMP correctly chooses a portion of the energy. The regularization step will then imply that the support is also selected correctly. We thus next show that the energy of $y$ when restricted to the sets $J$ and $J_0$ is sufficiently large.


\begin{lemma}[Localizing the energy]\label{C:ujSt}
  Let $y$ be the observation vector and $x_0$ be as in~\eqref{v0 x0St}. Then $\|y|_J\|_2 \ge 0.8 \|x_0\|_2 - (1+\e)\|e\|_2$.
\end{lemma}

\begin{proof}
Let $S$ = $\supp(x) \setminus I$ be the missing support.
Since $|S| \leq s$, by definition of $J$ in the algorithm, we have 
$$
\|y|_J\|_2 \geq \|y|_S\|_2.
$$
By Lemma~\ref{L:ujSt}, 
$$
\|y|_S\|_2 \geq \|y_0|_S\|_2 - 2.4\e\|x_0\|_2 - (1+\e)\|e\|_2.
$$
Since $x_0|_S = x_0$, Part 1 of Proposition~\ref{P:cons} implies
$$
\|y_0|_S\|_2 \geq (1 - 2.03\e)\|x_0\|_2.
$$
These three inequalities yield
$$
\|y|_J\|_2 \ge (1 - 2.03\e)\|x_0\|_2 - 2.4\e\|x_0\|_2 - (1+\e)\|e\|_2 \ge 0.8 \|x_0\|_2  - (1+\e)\|e\|_2.
$$
This completes the proof.
\end{proof}


\begin{lemma}[Regularizing the energy]\label{C:uj0St} 
  Again let $y$ be the observation vector and $x_0$ be as in~\eqref{v0 x0St}. Then
  $$
  \|y|_{J_0}\|_2 \ge \frac{1}{4\sqrt{\log s}}\|x_0\|_2 - \frac{\|e\|_2}{2\sqrt{\log s}}.
  $$
\end{lemma}
\begin{proof}
By Lemma~\ref{L:reg} applied to the vector $y|_J$, we have
$$
\|y|_{J_0}\|_2 \geq \frac{1}{2.5\sqrt{\log s}}\|y|_J\|_2.
$$
Along with Lemma~\ref{C:ujSt} this implies the claim.
\end{proof}
\medskip

We now conclude the proof of Theorem~\ref{T:itSt}. The claim that $J_0 \cap I = \emptyset$  
follows by the same arguments as in \cite{NV07:Uniform-Uncertainty}. 


It remains to show its last claim, that either (i) or (ii) holds.
Suppose (i) in the theorem fails. That is, suppose 
$|J_0 \cap \supp(x)| < \frac{1}{2}|J_0|$, 
which means
$$
|J_0 \backslash \supp(x)| > \frac{1}{2}|J_0|.
$$
Set $\Lambda = J_0\backslash\supp(x)$. 
By the definition of $J_0$ in the algorithm 
and since $|\Lambda| > \frac{1}{2}|J_0|$, we have by Lemma~\ref{C:uj0St},
\begin{equation}\label{E:ubigSt} 
  \|y|_{\Lambda}\|_2 > \frac{1}{\sqrt{5}}\|y|_{J_0}\|_2 
  \ge \frac{1}{4\sqrt{5\log s}}\|x_0\|_2 - \frac{\|e\|_2}{2\sqrt{5\log s}}. 
\end{equation}

Next, we also have
\begin{equation}                \label{u u0St}
  \|y|_{\Lambda}\|_2 
  \le \|y|_{\Lambda} - y_0|_{\Lambda}\|_2 + \|y_0|_{\Lambda}\|_2.
\end{equation}
Since $\L \subset J$ and $|J| \leq s$, by Lemma~\ref{L:ujSt} we have
$$
\|y|_{\Lambda} - y_0|_{\Lambda}\|_2 \le 2.4\e\|x_0\|_2 + (1+\e)\|e\|_2.
$$
By the definition of $x_0$ in \eqref{v0 x0St}, it must be that $x_0|_\Lambda = 0$. 
Thus by Part 1 of Proposition~\ref{P:cons}, 
$$
\|y_0|_{\Lambda}\|_2 \le 2.03 \e \|x_0\|_2.
$$
Using the previous inequalities along with \eqref{u u0St}, we deduce that 
$$
\|y|_{\Lambda}\|_2 \le 4.43 \e \|x_0\|_2 + (1+\e)\|e\|_2.
$$
This is a contradiction to~(\ref{E:ubigSt}) 
whenever 
$$
\e \leq \frac{0.02}{\sqrt{\log s}} - \frac{\|e\|_2}{\|x_0\|_2}.
$$ 
If this is true, then indeed (i) in the theorem must hold. If
it is not true, then by the choice of $\e$, this implies that
$$
\|x_0\|_2 \leq 100\|e\|_2\sqrt{\log s}.
$$
This proves Theorem~\ref{T:itSt}.\qed

%% file: rompProofCor1.tex
\begin{proof}
We first partition $x$ so that $u = \Phi x_{2s} + \Phi (x-x_{2s}) + e$. Then since $\Phi$ satisfies the Restricted Isometry Condition with parameters $(8s, \e)$, by Theorem~\ref{T:stability} and the triangle inequality,
\begin{equation}\label{primebnd}
\|x_{2s}-\hat{x}\|_2 \leq 104\sqrt{\log 2s}(\|\Phi (x-x_{2s})\|_2+\|e\|_2),
\end{equation}
The following lemma as in \cite{GSTV07:HHS} relates the $2$-norm of a vector's tail to its $1$-norm. An application of this lemma combined with \eqref{primebnd} will prove Corollary~\ref{T:stabsig}.

\begin{lemma}[Comparing the norms]\label{L:ve}
Let $v\in\R^d$, and let $v_T$ be the vector of the $T$ largest coordinates in absolute value from $v$.  Then
$$
\|v-v_T\|_2 \leq \frac{\|v\|_1}{2\sqrt{T}}.
$$
\end{lemma}
\begin{proof}
By linearity, we may assume $\|v\|_1 = d$. Since $v_T$ consists of the largest $T$ coordinates of $v$ in absolute value, we must have that $\|v-v_T\|_2 \leq \sqrt{d-T}$. (This is because the term $\|v-v_T\|_2$ is greatest when the vector $v$ has constant entries.) Then by the AM-GM inequality, 
$$
\|v-v_T\|_2 \sqrt{T} \leq \sqrt{d-T}\sqrt{T} \leq (d - T + T )/ 2 = d/2 = \|v\|_1/2.$$
\nopagebreak
This completes the proof.
\end{proof}


By Lemma 29 of \cite{GSTV07:HHS}, we have 
$$
\|\Phi (x-x_{2s})\|_2 \leq (1+\e)\Big(\|x-x_{2s}\|_2 + \frac{\|x-x_{2s}\|_1}{\sqrt{s}}\Big).
$$
Applying Lemma~\ref{L:ve} to the vector $v=x-x_s$ we then have
$$
\|\Phi (x-x_{2s})\|_2 \leq 1.5(1+\e)\frac{\|x-x_s\|_1}{\sqrt{s}}.
$$
Combined with \eqref{primebnd}, this proves the corollary.

\end{proof}

%% file: rompProofCor2.tex
Often one wishes to find a \textit{sparse} approximation to a signal. We now show that by simply truncating the reconstructed vector, we obtain a $2s$-sparse vector very close
to the original signal. 

\begin{proof}
Let $x_S := x_{2s}$ and $\hat{x_T} := \hat{x}_{2s}$, and let $S$ and $T$ denote the 
supports of $x_S$ and $\hat{x_T}$ respectively.
By Corollary~\ref{T:stabsig}, it suffices to show that 
$\|x_S - \hat{x}_T\|_2 \leq 3\|x_S - \hat{x}\|_2$. 

Applying the triangle inequality, we have
$$
\|x_S - \hat{x}_T\|_2 \leq \|(x_S - \hat{x}_T)|_T\|_2 + \|x_S|_{S\backslash T}\|_2 =: a + b.
$$
We then have
$$
a = \|(x_S - \hat{x}_T)|_T\|_2 = \|(x_S - \hat{x})|_T\|_2 \leq \|x_S - \hat{x}\|_2
$$
and
$$
b \leq \|\hat{x}|_{S\backslash T}\|_2 + \|(x_S-\hat{x})|_{S\backslash T}\|_2.
$$
Since $|S| = |T|$, we have $|S\backslash T| = |T\backslash S|$. By the definition of $T$, every coordinate of $\hat{x}$ in $T$ is greater than or equal to every coordinate of $\hat{x}$ in $T^c$ in absolute value. Thus we have,
$$
\|\hat{x}|_{S\backslash T}\|_2 \leq \|\hat{x}|_{T\backslash S}\|_2 = \|(x_S - \hat{x})|_{T\backslash S}\|_2.
$$
Thus $b \leq 2\|x_S - \hat{x}\|_2$, and so
$$
a+b \leq 3\|x_S - \hat{x}\|_2.
$$
This completes the proof.
\end{proof}

\medskip

\noindent {\bf Remark.} 
Corollary~\ref{C:napprox} combined with Corollary~\ref{T:stabsig} and \eqref{vbound} implies that we can also estimate a bound on the whole signal $v$:
$$
\|x-\hat{x}_{2s}\|_2 \leq C\sqrt{\log 2s}\Big(\|e\|_2 + \frac{\|x-x_{s}\|_1}{\sqrt{s}}\Big).
$$

%% file: cosamp.tex
    		Regularized Orthogonal Matching Pursuit bridged a critical gap between the major approaches in compressed sensing.  It provided the speed of the greedy approach and the strong guarantees of the convex optimization approach.  Although its contributions were significant, it still left room for improvement.  The requirements imposed by ROMP on the restricted isometry condition were slightly stronger than those imposed by the convex optimization approach.  This then in turn weakened the error bounds provided by ROMP in the case of noisy signals and measurements.  These issues were resolved by our algorithm Compressive Sampling Matching Pursuit (CoSaMP). A similar algorithm, Subspace Matching Pursuit was also developed around this time, which provides similar benefits to those of CoSaMP. See~\cite{DM08:Subspace-Pursuit} for details.

\subsection[Description]{Description}
    		\label{sec:New:Compressive:Description}
    		
				One of the key differences between ROMP and OMP is that at each iteration ROMP selects more than one coordinate to be in the support set.  Because of this, ROMP is able to make mistakes in the support set, while still correctly reconstructing the original signal.  This is accomplished because we bound the number of incorrect choices the algorithm can make.  Once the algorithm chooses an incorrect coordinate, however, there is no way for it to be removed from the support set.  An alternative approach would be to allow the algorithm to choose incorrectly as well as fix its mistakes in later iterations.  In this case, at each iteration we select a slightly larger support set, reconstruct the signal using that support, and use that estimation to calculate the residual. 
				
				Tropp and Needell developed a new variant of OMP, Compressive Sampling Matching Pursuit (CoSaMP)~\cite{NT08:Cosamp,DT08:CoSaMP-TR}. This new algorithm has the same uniform guarantees as ROMP, but does not impose the logarithmic term for the Restricted Isometry Property or in the error bounds. Since the sampling operator $\Phi$ satisfies the Restricted Isometry Property, every $s$ entries of the signal proxy $\vct{y} = \Phi^* \Phi \vct{x}$ are close in the Euclidean norm to the $s$ corresponding entries of the signal $\vct{x}$. Thus as in ROMP, the algorithm first selects the largest coordinates of the signal proxy $\vct{y}$ and adds these coordinates to the running support set. Next however, the algorithm performs a least squares step to get an estimate $\vct{b}$ of the signal, and prunes the signal to make it sparse. The algorithm's major steps are described as follows:

\begin{enumerate} \setlength{\itemsep}{0.5pc}
\item	{\bf Identification.}  The algorithm forms a proxy of the residual from the current samples and locates the largest components of the proxy.


\item	{\bf Support Merger.}  The set of newly identified components is united with the set of components that appear in the current approximation.

\item	{\bf Estimation.}  The algorithm solves a least-squares problem to approximate the target signal on the merged set of components.


\item	{\bf Pruning.}  The algorithm produces a new approximation by retaining only the largest entries in this least-squares signal approximation.


\item	{\bf Sample Update.}  Finally, the samples are updated so that they reflect the residual, the part of the signal that has not been approximated.
\end{enumerate}

\noindent
These steps are repeated until the halting criterion is triggered.  Initially, we concentrate on methods that use a fixed number of iterations.  Section~\ref{sec:New:Compressive:Implementation} discusses some other simple stopping rules that may also be useful in practice.  Using these ideas, the pseudo-code for CoSaMP can thus be described as follows.

    		\bigskip
   \textsc{Compressive Matching Pursuit (CoSaMP)~\cite{NT08:Cosamp}}

\nopagebreak

\fbox{\parbox{\algorithmwidth}{
  \textsc{Input:} Measurement matrix $\Phi$, measurement vector $\vct{u}=\Phi \vct{x}$, sparsity level $s$
  
  \textsc{Output:} $s$-sparse reconstructed vector $\hat{\vct{x}} = \vct{a}$

  \textsc{Procedure:}

\begin{description}

\item[Initialize] Set $\vct{a}^0 = \vct{0}$, $\vct{v} = \vct{u}$, $k = 0$. Repeat the following steps and increment $k$ until the halting criterion is true.

\item[Signal Proxy] Set $\vct{y} = \Phi^\adj \vct{v}$, $\Omega = \supp{ \vct{y}_{2s} }$ and merge the supports: $T = \Omega \cup \supp{ \vct{a}^{k-1} }$.

\item[Signal Estimation] Using least-squares, set $\vct{b}\restrict{T} = \Phi_T^\psinv \vct{u}$ and $\vct{b}\restrict{T^c} = \vct{0}$.

\item[Prune] To obtain the next approximation, set $\vct{a}^{k} = \vct{b}_s$. 

\item [Sample Update] Update the current samples: $\vct{v} = \vct{u} - \Phi \vct{a}^{k}$.
	
  \end{description}
 }}
 
 \bigskip 
 
There are a few major concepts of which the algorithm CoSaMP takes advantage. Unlike some other greedy algorithms, CoSaMP selects many components at each iteration. This idea can be found in theoretical work on greedy algorithms by Temlyakov as well as some early work of Gilbert, Muthukrishnan, Strauss and Tropp \cite{GMS03:Approximation-Functions},\cite{TGMS03:Improved-Sparse}. It is also the key idea of recent work on the Fourier sampling algorithm \cite{GMS05:Improved}. The ROMP and StOMP algorithms also incorporate this notion \cite{NV07:Uniform-Uncertainty}, \cite{DTDS06:Sparse-Solution}.

The application of the Restricted Isometry Property to compare the norms of vectors under the action of the sampling operator and its adjoint is also key in this algorithm and its analysis. The Restricted Isometry Property is due to Cand\`es and Tao \cite{CT05:Decoding-Linear}. The application of the property to greedy algorithms is relatively new, and appears in \cite{GSTV07:HHS} and \cite{NV07:Uniform-Uncertainty}.

Another key idea present in the algorithm is the pruning step to maintain sparsity of the approximation. This also has significant ramifications in other parts of the analysis and the running time. Since the Restricted Isometry Property only holds for sparse vectors, it is vital in the analysis that the approximation remain sparse. This idea also appears in \cite{GSTV07:HHS}. 

Our analysis focuses on mixed-norm error bounds. This idea appears in the work of Cand\`es, Romberg, Tao \cite{CRT06:Stable} as well as \cite{GSTV07:HHS} and \cite{CDD06:Remarks}. In our analysis, we focus on the fact that if the error is large, the algorithm must make progress. This idea appears in work by Gilbert and Strauss, for example \cite{GSTV07:HHS}.
  
The $L_1$-minimization method and the ROMP algorithm provide the strongest known guarantees of sparse recovery. These guarantees are \textit{uniform} in that once the sampling operator satisfies the Restricted Isometry Property, both methods work correctly for all sparse signals. $L_1$-minimization is based on linear programming, which provides only a polynomial runtime.  Greedy algorithms such as OMP and ROMP on the other hand, are much faster both in theory and empirically. Our algorithm CoSaMP provides both uniform guarantees as well as fast runtime, while improving upon the error bounds and Restricted Isometry requirements of ROMP.  We describe these results next as we state the main theorems.  
    		
    		\subsection[Main Theorem]{Main Theorem}
    		\label{sec:New:Compressive:Main}
    		
    		Next we state the main theorem which guarantees exact reconstruction of sparse signals and approximate reconstruction of arbitrary signals.  The proof of the theorem is presented in Section~\ref{sec:New:Compressive:Proof}.
    		
    		\begin{theorem}[CoSaMP~\cite{NT08:Cosamp}] \label{thm:cosamp}
Suppose that $\Fee$ is an $m \times d$ sampling matrix with restricted isometry constant $\delta_{2s} \leq \cnst{c}$, as in~\eqref{eq:RIC2}.  Let $\vct{u} = \Fee \vct{x} + \vct{e}$ be a vector of samples of an arbitrary signal, contaminated with arbitrary noise.  For a given precision parameter $\eta$, the algorithm CoSaMP produces an $s$-sparse approximation $\hat{\vct{x}}$ that satisfies
$$
\enorm{ \vct{x} - \hat{\vct{x}} } \leq
	\cnst{C} \cdot \max\left\{ \eta, 
	\frac{1}{\sqrt{s}} \pnorm{1}{\vct{x} - \vct{x}_{s/2}} + \enorm{ \vct{e} }
	\right\}
$$
where $\vct{x}_{s/2}$ is a best $(s/2)$-sparse approximation to $\vct{x}$.
The running time is $\bigO( \coll{L} \cdot \log ( \enorm{\vct{x}} / \eta ) )$, where $\coll{L}$ bounds the cost of a matrix--vector multiply with $\Fee$ or $\Fee^\adj$.  Working storage is $\bigO(d)$.
\end{theorem}  

\begin{remarks} 
{\bf 1. } We note that as in the case of ROMP, CoSaMP requires knowledge of the sparsity level $s$.  As described in Section~\ref{sec:New:Regularized:Description}, there are several strategies to estimate this level. 

{\bf 2. } In the hypotheses, a bound on the restricted isometry constant $\delta_{2s}$ also suffices.  Indeed, Corollary~\ref{cor:dumb-rip-bd} of the sequel implies that $\delta_{4s} \leq 0.1$ holds whenever $\delta_{2s} \leq 0.025$. 

{\bf 3. } Theorem~\ref{thm:cosamp} is a result of running CoSaMP using an iterative algorithm to solve the least-squares step.  We analyze this step in detail below.  In the case of exact arithmetic, we again analyze CoSaMP and provide an iteration count for this case:

\begin{theorem}[Iteration Count] \label{thm:cosamp-count}
Suppose that CoSaMP is implemented with exact arithmetic.  After at most $6(s+1)$ iterations, CoSaMP produces an $s$-sparse approximation $\hat{\vct{x}}$ that satisfies
$$
\enorm{ \vct{x} - \vct{a} } \leq 20 \nu,
$$
where $\nu$ is the unrecoverable energy~\eqref{eqn:unrecoverable}. 
\end{theorem}
  See Theorem~\ref{thm:count-detail} in Section~\ref{sec:New:Compressive:Proof} below for more details. 

\end{remarks}

The algorithm produces an $s$-sparse approximation whose $\ell_2$ error is comparable with the scaled $\ell_1$ error of the best $(s/2)$-sparse approximation to the signal.  Of course, the algorithm cannot resolve the uncertainty due to the additive noise, so we also pay for the energy in the noise.  This type of error bound is structurally optimal, as discuss when describing the unrecoverable energy below.  Some disparity in the sparsity levels (here, $s$ versus $s/2$) seems to be necessary when the recovery algorithm is computationally efficient~\cite{RG08:Sampling-Bounds}.

To prove our theorem, we show that CoSaMP makes significant progress during each iteration where the approximation error is large relative to \term{unrecoverable energy} $\nu$ in the signal.  This quantity measures the baseline error in our approximation that occurs because of noise in the samples or because the signal is not sparse.  For our purposes, we define the unrecoverable energy by the following.

\begin{equation} \label{eqn:unrecoverable}
\nu = \enorm{ \vct{x} - \vct{x}_s }
	+ \frac{1}{\sqrt{s}} \pnorm{1}{ \vct{x} - \vct{x}_s }
	+ \enorm{ \vct{e} }.
\end{equation}

The expression \eqref{eqn:unrecoverable} for the unrecoverable energy can be simplified using Lemma~7 from~\cite{GSTV07:HHS}, which states that, for every signal $\vct{y} \in \Cspace{N}$ and every positive integer $t$, we have
\begin{equation*} \label{eqn:heads-tails}
\enorm{ \vct{y} - \vct{y}_t } \leq \frac{1}{2\sqrt{t}} \pnorm{1}{\vct{y}}.
\end{equation*}
Choosing $\vct{y} = \vct{x} - \vct{x}_{s/2}$ and $t = s/2$, we reach 
\begin{equation} \label{eqn:unrecov-l1}
\nu 
	\leq \frac{1.71}{\sqrt{s}} \pnorm{1}{ \vct{x} - \vct{x}_{s/2} }
	+ \enorm{\vct{e}}.
\end{equation}
In words, the unrecoverable energy is controlled by the scaled $\ell_1$ norm of the signal tail.  

The term ``unrecoverable energy'' is justified by several facts.  First, we must pay for the $\ell_2$ error contaminating the samples.  To check this point, define $S = \supp{\vct{x}_s}$.  The matrix $\Fee_S$ is nearly an isometry from $\ell_2^S$ to $\ell_2^m$, so an error in the large components of the signal induces an error of equivalent size in the samples.  Clearly, we can never resolve this uncertainty.

The term $s^{-1/2} \pnorm{1}{\vct{x} - \vct{x}_{s}}$ is also required on account of classical results about the Gel'fand widths of the $\ell_1^d$ ball in $\ell_2^d$, due to Kashin~\cite{Kas77:The-widths} and Garnaev--Gluskin~\cite{GG84:On-widths}.  In the language of compressive sampling, their work has the following interpretation.  Let $\Fee$ be a fixed $m \times d$ sampling matrix.  Suppose that, for every signal $\vct{x} \in \Cspace{d}$, there is an algorithm that uses the samples $\vct{u} = \Fee \vct{x}$ to construct an approximation $\vct{a}$ that achieves the error bound
$$
\enorm{ \vct{x} - \vct{a} }
	\leq \frac{\cnst{C}}{\sqrt{s}} \pnorm{1}{\vct{x}}.
$$
Then the number $m$ of measurements must satisfy
$m \geq \cnst{c} s \log(d/s)$.

			 			\subsection[Proofs of Theorems]{Proofs of Theorems}
    		\label{sec:New:Compressive:Proof}
    		\input{cosampProof}
    		
    		\subsection[Implementation and Runtime]{Implementation and Runtime}
    		\label{sec:New:Compressive:Implementation}
    		
    	CoSaMP was designed to be a practical method for signal recovery.  An efficient implementation of the algorithm requires 
some ideas from numerical linear algebra, as well as some basic techniques from the theory of algorithms.  This section discusses the key issues and develops an analysis of the running time for the two most common scenarios.

We focus on the least-squares problem in the estimation step because it is the major obstacle to a fast implementation of the algorithm.  The algorithm guarantees that the matrix $\Fee_T$ never has more than $3s$ columns, so our assumption $\delta_{4s} \leq 0.1$ implies that the matrix $\Fee_T$ is extremely well conditioned.  As a result, we can apply the pseudoinverse $\Fee_T^\psinv = (\Fee_T^\adj \Fee_T)^{-1} \Fee_T^\adj$ very quickly using an iterative method, such as Richardson's iteration~\cite[Sec.~7.2.3]{Bjo96:Numerical-Methods} or conjugate gradient~\cite[Sec.~7.4]{Bjo96:Numerical-Methods}.  These techniques have the additional advantage that they only interact with the matrix $\Fee_T$ through its action on vectors.  It follows that the algorithm performs better when the sampling matrix has a fast matrix--vector multiply.

Let us stress that direct methods for least squares are likely to be extremely inefficient in this setting.  The first reason is that each least-squares problem may contain substantially different sets of columns from $\Fee$.  As a result, it becomes necessary to perform a completely new {\sf QR} or {\sf SVD} factorization during each iteration at a cost of $\bigO(s^2 m)$.  The second problem is that computing these factorizations typically requires direct access to the columns of the matrix, which is problematic when the matrix is accessed through its action on vectors.  Third, direct methods have storage costs $\bigO(sm)$, which may be deadly for large-scale problems.

The remaining steps of the algorithm involve standard techniques.  Let us estimate the operation counts.

\begin{description} \setlength{\itemsep}{0.3pc}
\item	[Proxy]
	Forming the proxy is dominated by the cost of the matrix--vector multiply $\Fee^\adj \vct{v}$.

\item	[Identification]
	We can locate the largest $2s$ entries of a vector in time $\bigO(N)$ using
the approach in \cite[Ch.~9]{CLRS01:Intro-Algorithms}.  In practice, it may be faster to sort the entries of the signal in decreasing order of magnitude at cost $\bigO(N \log N)$ and then select the first $2s$ of them.  The latter procedure can be accomplished with quicksort, mergesort, or heapsort~\cite[Sec.~II]{CLRS01:Intro-Algorithms}.  To implement the algorithm to the letter, the sorting method needs to be stable because we stipulate that ties are broken lexicographically. This point is not important in practice.

\item	[Support Merger]
	We can merge two sets of size $\bigO(s)$ in expected time $\bigO(s)$ using randomized hashing methods~\cite[Ch.~11]{CLRS01:Intro-Algorithms}.  One can also sort both sets first and use the elementary merge procedure~\cite[p.~29]{CLRS01:Intro-Algorithms} for a total cost $\bigO(s \log s)$.

\item	[LS estimation]
	We use Richardson's iteration or conjugate gradient to compute $\Fee_T^\psinv \vct{u}$.  Initializing the least-squares algorithm requires a matrix--vector multiply with $\Fee_T^\adj$.  Each iteration of the least-squares method requires one matrix--vector multiply each with $\Fee_T$ and $\Fee_T^\adj$.  Since $\Fee_T$ is a submatrix of $\Fee$, the matrix--vector multiplies can also be obtained from multiplication with the full matrix.  We proved above that a constant number of least-squares iterations suffices for Theorem~\ref{thm:cosamp-invar} to hold.
	
	\item	[Pruning]
	This step is similar to identification.  Pruning can be implemented in time $\bigO(s)$, but it may be preferable to sort the components of the vector by magnitude and then select the first $s$ at a cost of $\bigO(s\log s)$.

\item	[Sample Update]
	This step is dominated by the cost of the multiplication of $\Fee$ with the $s$-sparse vector $\vct{a}^k$.
\end{description}
    		
    		Table~\ref{tab:runtime} summarizes this discussion in two particular cases.  The first column shows what happens when the sampling matrix $\Fee$ is applied to vectors in the standard way, but we have random access to submatrices.  The second column shows what happens when the sampling matrix $\Fee$ and its adjoint $\Fee^\adj$ both have a fast multiply with cost $\coll{L}$, where we assume that $\coll{L} \geq N$.  A typical value is $\coll{L} = \bigO(N \log N)$.  In particular, a partial Fourier matrix satisfies this bound.

\begin{table}[thb]
\centering
\renewcommand{\arraystretch}{1.25}
\caption{Operation count for CoSaMP.  Big-O notation is omitted for legibility.  The dimensions of the sampling matrix  $\Fee$ are $m \times N$; the sparsity level is $s$.  The number $\coll{L}$ bounds the cost of a matrix--vector multiply with $\Fee$ or $\Fee^\adj$.}  
	\label{tab:runtime}
\vspace{1pc}
\begin{tabular}{|l||r|r|}
\hline
Step				& Standard multiply		& Fast multiply \\
\hline\hline
Form proxy			& $mN$ 					& $\coll{L}$ \\
Identification		& $N$					& $N$ \\
Support merger		& $s$					& $s$ \\
LS estimation		& $sm$ 					& $\coll{L}$ \\
Pruning				& $s$ 					& $s$ \\
Sample update			& $sm$					& $\coll{L}$ \\
\hline
Total per iteration & $\bigO(mN)$ 			& $\bigO(\coll{L})$ \\
\hline
\end{tabular}
\vspace{1pc}
\end{table}

Finally, we note that the storage requirements of the algorithm are also favorable.  Aside from the storage required by the sampling matrix, the algorithm constructs only one vector of length $N$, the signal proxy.  The sample vectors $\vct{u}$ and $\vct{v}$ have length $m$, so they require $\bigO(m)$ storage.  The signal approximations can be stored using sparse data structures, so they require at most $\bigO(s \log N)$ storage.  Similarly, the index sets that appear require only $\bigO(s \log N)$ storage.  The total storage is at worst $\bigO(N)$.

The following result summarizes this discussion.

\begin{theorem}[Resource Requirements] \label{thm:resources}
Each iteration of CoSaMP requires $\bigO(\coll{L})$ time, where $\coll{L}$ bounds the cost of a multiplication with the matrix $\Fee$ or $\Fee^\adj$.  The algorithm uses storage $\bigO(N)$.
\end{theorem}

\subsubsection{Algorithmic Variations}

This section describes other possible halting criteria and their consequences.  It also proposes some other variations on the algorithm.

There are three natural approaches to halting the algorithm.  The first, which we have discussed in the body of the paper, is to stop after a fixed number of iterations.  Another possibility is to use the norm $\enorm{ \vct{v} }$ of the current samples as evidence about the norm $\enorm{\vct{r}}$ of the residual.  A third possibility is to use the magnitude $\infnorm{\vct{y}}$ of the entries of the proxy to bound the magnitude $\infnorm{\vct{r}}$ of the entries of the residual.

It suffices to discuss halting criteria for sparse signals because Lemma~\ref{lem:reduction} shows that the general case can be viewed in terms of sampling a sparse signal.  Let $\vct{x}$ be an $s$-sparse signal, and let $\vct{a}$ be an $s$-sparse approximation.  The residual $\vct{r} = \vct{x} - \vct{a}$.  We write $\vct{v} = \Fee \vct{r} + \vct{e}$ for the induced noisy samples of the residual and $\vct{y} = \Fee^\adj \vct{v}$ for the signal proxy. 

The discussion proceeds in two steps.  First, we argue that an \term{a priori} halting criterion will result in a guarantee about the quality of the final signal approximation.


\begin{theorem}[Halting I]
The halting criterion $\enorm{ \vct{v} } \leq \eps$ ensures that
$$
\enorm{ \vct{x} - \vct{a} } \leq 1.06 \cdot ( \eps + \enorm{ \vct{e}}).
$$
The halting criterion $\pnorm{\infty}{ \vct{y} } \leq \eta / \sqrt{2s}$ ensures that
$$
\pnorm{\infty}{ \vct{x} - \vct{a} } \leq 1.12 \eta + 1.17 \enorm{\vct{e}}.
$$
\end{theorem}

\begin{proof}
Since $\vct{r}$ is $2s$-sparse, Proposition~\ref{prop:rip-basic} ensures that
$$
\sqrt{1 - \delta_{2s}} \enorm{ \vct{r} } - \enorm{ \vct{e} }
	\leq \enorm{ \vct{v} }.
$$
If $\enorm{ \vct{v} } \leq \eps$, it is immediate that
$$
\enorm{ \vct{r} } \leq
\frac{ \eps + \enorm{ \vct{e} } }{\sqrt{1 - \delta_{2s}}}.
$$
The definition $\vct{r} = \vct{x} - \vct{a}$ and the numerical bound $\delta_{2s} \leq \delta_{4s} \leq 0.1$ dispatch the first claim.

Let $R = \supp{ \vct{r} }$, and note that $\abs{R} \leq 2s$.  Proposition~\ref{prop:rip-basic} results in
$$
(1 - \delta_{2s}) \enorm{ \vct{r} } - \sqrt{1 + \delta_{2s}} \enorm{ \vct{e} }
	\leq \enorm{ \vct{y}\restrict{R} }.
$$
Since
$$
\enorm{ \vct{y} \restrict{R} }
	\leq \sqrt{2s} \pnorm{ \infty}{ \vct{y}\restrict{R} }
	\leq \sqrt{2s} \pnorm{ \infty}{ \vct{y} },
$$
we find that the requirement $\pnorm{\infty}{\vct{y}} \leq \eta /\sqrt{2s}$ ensures that
$$
\pnorm{\infty}{ \vct{r} } \leq
	\frac{\eta + \sqrt{1 + \delta_{2s}}\enorm{ \vct{e} }}{1 - \delta_{2s}}.
$$
The numerical bound $\delta_{2s} \leq 0.1$ completes the proof.
\end{proof}

Second, we check that each halting criterion is triggered when the residual has the desired property.

\begin{theorem}[Halting II]
The halting criterion $\enorm{ \vct{v} } \leq \eps$ is triggered as soon as
$$
\enorm{ \vct{x} - \vct{a} } \leq 0.95 \cdot ( \eps - \enorm{ \vct{e} } ).
$$
The halting criterion $\pnorm{\infty}{\vct{y}} \leq \eta / \sqrt{2s}$ is triggered as soon as
$$
\pnorm{\infty}{ \vct{x} - \vct{a} } \leq \frac{0.45 \eta}{ s } - \frac{0.68 \enorm{\vct{e}}}{\sqrt{s}}.
$$\end{theorem}

\begin{proof}
Proposition \ref{prop:rip-basic} shows that
$$
\enorm{ \vct{v} } \leq \sqrt{1 + \delta_{2s}} \enorm{ \vct{r}} + \enorm{ \vct{e} }.
$$
Therefore, the condition
$$
\enorm{\vct{r}} \leq
	\frac{\eps - \enorm{\vct{e}}}{\sqrt{1 + \delta_{2s}}}
$$
ensures that $\enorm{ \vct{v} } \leq \eps$.  Note that $\delta_{2s} \leq 0.1$ to complete the first part of the argument.

Now let $R$ be the singleton containing the index of a largest-magnitude coefficient of $\vct{y}$.
Proposition~\ref{prop:rip-basic} implies that
$$
\pnorm{ \infty}{ \vct{y} } = \enorm{ \vct{y}\restrict{R} } \leq \sqrt{1 + \delta_{1}}\enorm{ \vct{v} }.
$$
By the first part of this theorem, the halting criterion $\pnorm{\infty}{\vct{y}} \leq \eta / \sqrt{2s}$ is triggered as soon as
$$
\enorm{ \vct{x} - \vct{a} } \leq 0.95 \cdot \left( \frac{\eta}{\sqrt{2s}\sqrt{1 + \delta_{1}}} - \enorm{ \vct{e} } \right).
$$
Since $\vct{x} - \vct{a}$ is $2s$-sparse, we have the bound $\enorm{ \vct{x} - \vct{a} } \leq \sqrt{2s}\pnorm{\infty}{ \vct{x-a} }$.  To wrap up, recall that $\delta_{1} \leq \delta_{2s} \leq 0.1$.
\end{proof}

Next we discuss other variations of the algorithm.

Here is a version of the algorithm that is, perhaps, simpler than that described above.  At each iteration, we approximate the current residual rather than the entire signal.  This approach is similar to HHS pursuit \cite{GSTV07:HHS}.  The inner loop changes in the following manner.
\begin{description} \setlength{\itemsep}{0.5pc}
\item	[Identification]
	As before, select $\Omega = \supp{\vct{y}_{2s}}$.

\item	[Estimation]
	Solve a least-squares problem with the \emph{current samples} instead of the original samples to obtain an approximation of the \emph{residual signal}.  Formally, $\vct{b} = \Fee_\Omega^\psinv \vct{v}$.  In this case, one initializes the iterative least-squares algorithm with the zero vector to take advantage of the fact that the residual is becoming small.

\item	[Approximation Merger]
	Add this approximation of the residual to the previous approximation of the signal to obtain a new approximation of the signal: $\vct{c} = \vct{a}^{k-1} + \vct{b}$.

\item	[Pruning]
	Construct the $s$-sparse signal approximation: $\vct{a}^k = \vct{c}_s$.

\item	[Sample Update]
	Update the samples as before: $\vct{v} = \vct{u} - \Fee \vct{a}^k$.
\end{description}

By adapting the argument in this paper, we have been able to show that this algorithm also satisfies Theorem~\ref{thm:cosamp}.  We believe this version is quite promising for applications.

An alternative variation is as follows.  After the inner loop of the algorithm is complete, we can solve another least-squares problem in an effort to improve the final result.  If $\vct{a}$ is the approximation at the end of the loop, we set $T = \supp{ \vct{a} }$.  Then solve $\vct{b} = \Fee_T^\psinv \vct{u}$ and output the $s$-sparse signal approximation $\vct{b}$.  Note that the output approximation is not guaranteed to be better than $\vct{a}$ because of the noise vector $\vct{e}$, but it should never be much worse.

Another variation is to prune the merged support $T$ down to $s$ entries \emph{before} solving the least-squares problem.  One may use the values of the proxy $\vct{y}$ as surrogates for the unknown values of the new approximation on the set $\Omega$.  Since the least-squares problem is solved at the end of the iteration, the columns of $\Fee$ that are used in the least-squares approximation are orthogonal to the current samples $\vct{v}$.  As a result, the identification step always selects new components in each iteration.  We have not attempted an analysis of this algorithm.

    		\subsection[Numerical Results]{Numerical Results}
    		\label{sec:New:Compressive:Numerical}
    		
    			\subsubsection[Noiseless Numerical Studies]{Noiseless Numerical Studies}
    		
    		This section describes our experiments that illustrate the signal recovery power of CoSaMP.  See Section~\ref{app:code:Cosampcode} for the Matlab code used in these studies.
We experimentally examine how many measurements $m$ are necessary to recover various kinds of $s$-sparse
signals in $\R^d$ using ROMP. 
We also demonstrate that the number of iterations CoSaMP needs to recover a sparse signal is
in practice at most linear the sparsity, and in fact this serves as a successful halting criterion. 

First we describe the setup of our experiments. For many values of the ambient dimension $d$, 
the number of measurements $m$, and the sparsity $s$, we reconstruct random signals using CoSaMP.
For each set of values, we generate an $m \times d$ Gaussian measurement matrix $\Phi$ and then perform $500$ independent trials. The results
we obtained using Bernoulli measurement matrices were very similar.
In a given trial, we generate an $s$-sparse signal $x$ in one of two ways. In either case, we first select the support of the signal by choosing $s$ components uniformly 
at random (independent from the measurement matrix $\Phi$). In the cases where we wish to generate flat signals, we then  set these components to one.  In the cases where we wish to generate sparse compressible signals, we set the $i^{th}$ component of the support to plus or minus $i^{-1/p}$ for a specified value of $0 < p < 1$. We then execute CoSaMP with the measurement
vector $u = \Phi x$.  

Figure~\ref{figcos:percent} depicts the percentage (from the $500$ trials) of sparse flat signals that were 
reconstructed exactly. This plot was generated with $d = 256$ for various levels of sparsity $s$. The horizontal axis represents the number of measurements $m$, and the vertical
axis represents the exact recovery percentage. 

Figure~\ref{figcos:99} depicts a plot of the values for $m$ and $s$ at which $99\%$ of sparse flat signals are recovered exactly. This plot was generated with $d=256$. The horizontal axis represents the number of measurements $m$, and the vertical axis the sparsity level $s$. 

Our results guarantee that CoSaMP reconstructs signals correctly with just $O(s)$ iterations. Figure~\ref{figcos:itsROMP} depicts the number of iterations needed by CoSaMP for $d=10,000$ and $m=200$ for perfect reconstruction. CoSaMP was 
executed under the same setting as described above for sparse flat signals as well as sparse compressible signals for various values of $p$, and the number of iterations
in each scenario was averaged over the $500$ trials. These averages were plotted against
the sparsity of the signal. The plot demonstrates that often far fewer iterations are actually needed in some cases.  This is not surprising, since as we discussed above alternative halting criteria may be better in practice.  The plot also demonstrates that the number of iterations needed for sparse compressible
is higher than the number needed for sparse flat signals, as one would expect. The plot suggests that for smaller
values of $p$ (meaning signals that decay more rapidly) CoSaMP needs more iterations.

\begin{figure}[ht] 
  \includegraphics[width=0.8\textwidth,height=3.2in]{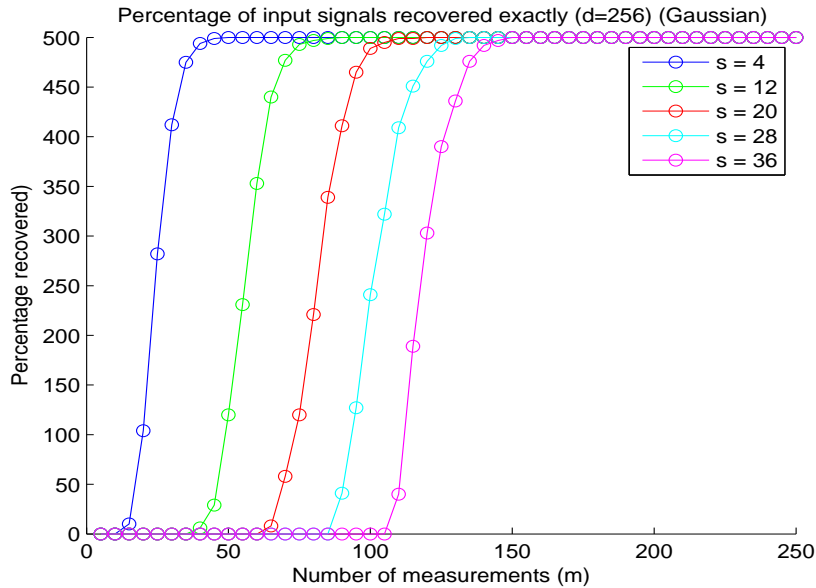}
  \caption{The percentage of sparse flat signals exactly recovered by CoSaMP as a function of the number of measurements in dimension $d=256$ for various levels of sparsity.}\label{figcos:percent}
\end{figure}

\begin{figure}[ht] 
  \includegraphics[width=0.8\textwidth,height=3.2in]{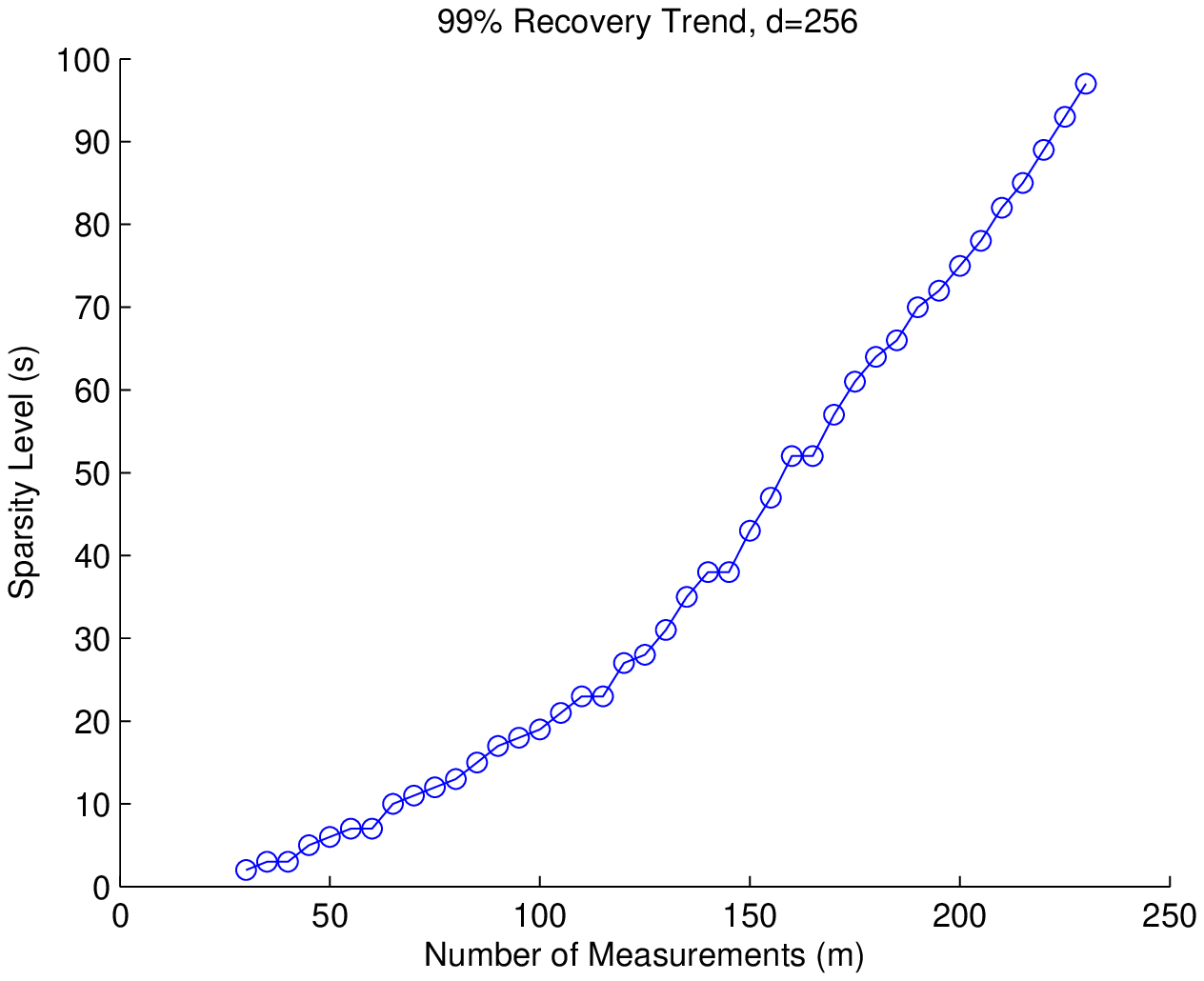}
  \caption{The $99\%$ recovery limit as a function of the sparsity and the number of measurements for sparse flat signals.}\label{figcos:99}
\end{figure}

\begin{figure}[ht] 
  \includegraphics[width=0.8\textwidth,height=3.2in]{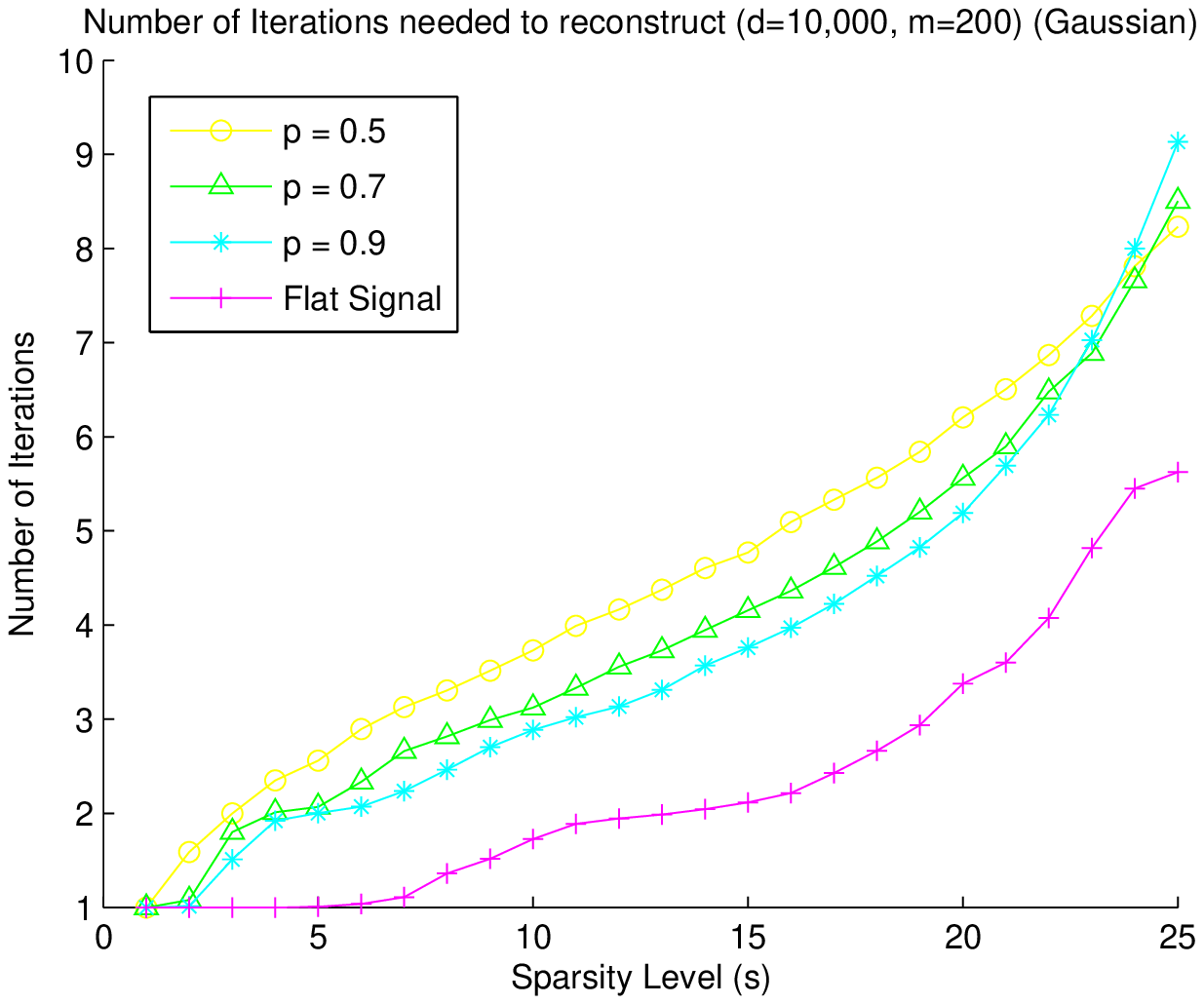}
  \caption{The number of iterations needed by CoSaMP as a function of the sparsity in dimension $d=10,000$ with $200$ measurements.}\label{figcos:itsROMP}
\end{figure}
    		
    		\subsubsection[Noisy Numerical Studies]{Noisy Numerical Studies}
    		
    		This section describes our numerical experiments that illustrate the stability of CoSaMP. 
We study the recovery error using CoSaMP for both perturbed
measurements and signals. 
The empirical recovery error confirms that given in the theorems. 

First we describe the setup to our experimental studies. We run CoSaMP on various values of the ambient dimension $d$, 
the number of measurements $m$, and the sparsity level $s$, and attempt to reconstruct random signals.
For each set of parameters, we perform $500$ trials. Initially, we generate an $m \times d$ Gaussian measurement matrix $\Phi$. For each trial, independent of the matrix, we generate an $s$-sparse signal $x$ by choosing $s$ components uniformly 
at random and setting them to one.
In the case of perturbed signals, we add to the signal a $d$-dimensional error vector with Gaussian entries. In the case of perturbed measurements, we add an $m$-dimensional error vector with Gaussian entries to the measurement vector $\Phi x$.  We then execute ROMP with the measurement
vector $u = \Phi x$ or $u + e$ in the perturbed measurement case. After CoSaMP terminates (using a fixed number of iterations of $10s$), we output the reconstructed vector $\hat{x}$ obtained from the least squares calculation and calculate its distance from the original signal. 

Figure~\ref{figcos:meas2} depicts the recovery error $\|x - \hat{x}\|_2$ when CoSaMP was run with perturbed measurements. This plot was generated with $d = 256$ for various levels of sparsity $s$. The horizontal axis represents the number of measurements $m$, and the vertical
axis represents the average normalized recovery error. 

Figure~\ref{figcos:sig4} depicts the normalized recovery error when the signal was perturbed by a Gaussian vector. Again these results are consistent with our proven theorems, but notice that we normalize here by $\|x-x_s\|_1/\sqrt{s}$ rather than $\|x-x_{s/2}\|_1/\sqrt{s}$ as our theorems suggest.  These plots show that perhaps the former normalization factor is actually more accurate, and the latter may be a consequence of our analysis only.


\begin{figure}[ht] 
  \includegraphics[width=0.8\textwidth,height=3.2in]{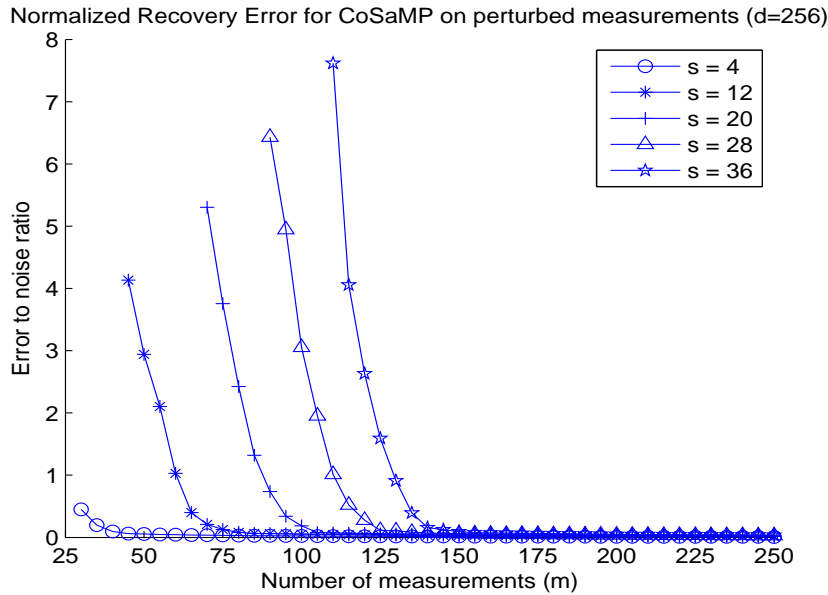}
  \caption{The error to noise ratio $\frac{\|\hat{x} - x\|_2}{\|e\|_2}$ as a function of the number of measurements $m$ in dimension $d=256$ for various levels of sparsity $s$.}\label{figcos:meas2}
\end{figure}

\begin{figure}[ht] 
  \includegraphics[width=0.8\textwidth,height=3.2in]{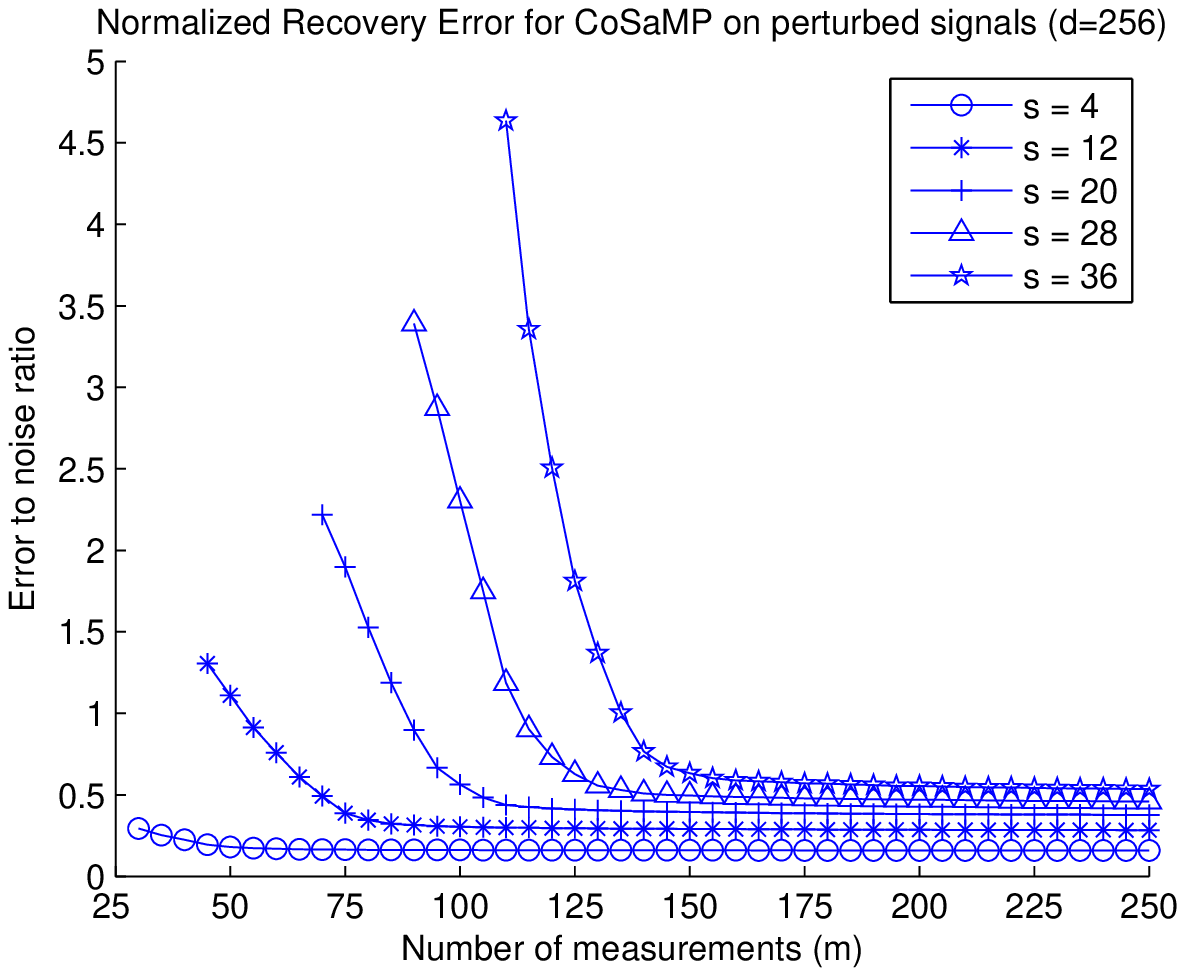}
  \caption{The error to noise ratio $\frac{\|\hat{x} - x_{s}\|_2}{\|x-x_s\|_1/\sqrt{s}}$ using a perturbed signal, as a function of the number of measurements $m$ in dimension $d=256$ for various levels of sparsity $s$.}\label{figcos:sig4}
\end{figure}
    		
    		\subsection[Summary]{Summary}
    		\label{sec:New:Compressive:Summary}

CoSaMP draws on both algorithmic ideas and analytic techniques that have appeared before.  Here we summarize the results in the context of other work.  This discussion can also be found in~\cite{NT08:Cosamp}.
The initial discovery works on compressive sampling proposed to perform signal recovery by solving a convex optimization problem~\cite{CRT06:Robust-Uncertainty,Don06:Compressed-Sensing} (see also Section~\ref{sec:Approaches:Basis} above).  Given a sampling matrix $\Fee$ and a noisy vector of samples $\vct{u} = \Fee \vct{x} + \vct{e}$ with $\enorm{\vct{e}}\leq \eps$, we consider the mathematical program~\eqref{eqn:bp}.
In words, we look for a signal reconstruction that is consistent with the samples but has minimal $\ell_1$ norm.  The intuition behind this approach is that minimizing the $\ell_1$ norm promotes sparsity, so allows the approximate recovery of compressible signals.  Cand{\`e}s, Romberg, and Tao established in~\cite{CRT06:Stable} that a minimizer $\vct{a}$ of \eqref{eqn:bp} satisfies
\begin{equation} \label{eqn:bp-err}
\enorm{ \vct{x} - \vct{a} } \leq \cnst{C} \left[ \frac{1}{\sqrt{s}} \pnorm{1}{\vct{x} - \vct{x}_s }
	+ \eps \right]
\end{equation}
provided that the sampling matrix $\Fee$ has restricted isometry constant $\delta_{4s} \leq 0.2$.  In \cite{Can08:Restricted-Isometry}, the hypothesis on the restricted isometry constant is sharpened to $\delta_{2s} \leq \sqrt{2} - 1$.  The error bound for CoSaMP is equivalent, modulo the exact value of the constants.

The literature describes a huge variety of algorithms for solving the optimization problem \eqref{eqn:bp}.  The most common approaches involve interior-point methods~\cite{CRT06:Robust-Uncertainty,KKL+06:Method-Large-Scale}, projected gradient methods~\cite{FNW07:Gradient-Projection}, or iterative thresholding~\cite{DDM04:Iterative-Thresholding} 
The interior-point methods are guaranteed to solve the problem to a fixed precision in time $\bigO( m^2 d^{1.5} )$, where $m$ is the number of measurements and $d$ is the signal length~\cite{NN94:Interior-Point}.  Note that the constant in the big-O notation depends on some of the problem data.  The other convex relaxation algorithms, while sometimes faster in practice, do not currently offer rigorous guarantees.  CoSaMP provides rigorous bounds on the runtime that are much better than the available results for interior-point methods.

Tropp and Gilbert proposed the use of a greedy iterative algorithm called \term{orthogonal matching pursuit} (OMP) for signal recovery \cite{TG07:Signal-Recovery} (see also Section~\ref{sec:Approaches:Greedy:Orthogonal} above).  
Tropp and Gilbert were able to prove a weak result for the performance of OMP \cite{TG07:Signal-Recovery}.  Suppose that $\vct{x}$ is a fixed, $s$-sparse signal, and let $m = \cnst{C} s \log s$.  Draw an $m \times s$ sampling matrix $\Fee$ whose entries are independent, zero-mean subgaussian random variables with equal variances.  Given noiseless measurements $\vct{u} = \Fee \vct{x}$, OMP reconstructs $\vct{x}$ after $s$ iterations, except with probability $s^{-1}$.  In this setting, OMP must fail for some sparse signals\cite{R08:Impossibility}, so it does not provide the same uniform guarantees as convex relaxation.  It is unknown whether OMP succeeds for compressible signals or whether it succeeds when the samples are contaminated with noise.

Donoho et al.~invented another greedy iterative method called \term{stagewise OMP}, or StOMP \cite{DTDS06:Sparse-Solution} (see also Section~\ref{sec:Approaches:Greedy:Stagewise} above).  This algorithm uses the signal proxy to select multiple components at each step, using a rule inspired by ideas from wireless communications.  The algorithm is faster than OMP because of the selection rule, and it sometimes provides good performance, although parameter tuning can be difficult.  There are no rigorous results available for StOMP.

Needell and Vershynin developed and analyzed another greedy approach, called \term{regularized OMP}, or ROMP \cite{NV07:Uniform-Uncertainty, NV07:ROMP-Stable} (see also Section~\ref{sec:New:Regularized} above).  The work on ROMP represents an advance because the authors establish under restricted isometry hypotheses that their algorithm can approximately recover any compressible signal from noisy samples.  More precisely, suppose that the sampling matrix $\Fee$ has restricted isometry constant $\delta_{8s} \leq 0.01 / \sqrt{\log s}$.   Given noisy samples $\vct{u} = \Fee\vct{x} + \vct{e}$, ROMP produces a $2s$-sparse signal approximation $\vct{a}$ that satisfies
$$
\enorm{ \vct{x} - \vct{a} } \leq \cnst{C} \sqrt{\log s} \left[ \frac{1}{\sqrt{s}}
	\pnorm{1}{ \vct{x} - \vct{x}_s } + \enorm{ \vct{e} } \right].
$$
This result is comparable with the result for convex relaxation, aside from the extra logarithmic factor in the restricted isometry hypothesis and the error bound.  The results for CoSaMP show that it does not suffer these parasitic factors, so its performance is essentially optimal.

%% file: cosampProof.tex
Theorem~\ref{thm:cosamp} will be shown by demonstrating that the following iteration invariant holds.  These results can be found in~\cite{NT08:Cosamp}.

\begin{theorem}[Iteration Invariant] \label{thm:cosamp-invar}
For each iteration $k \geq 0$, the signal approximation $\vct{a}^k$ is $s$-sparse and
$$
\smnorm{2}{ \vct{x} - \vct{a}^{k+1} }
	\leq 0.5 \smnorm{2}{ \vct{x} - \vct{a}^{k} } + 10 \nu.
$$
In particular,
$$
\smnorm{2}{ \vct{x} - \vct{a}^{k} } \leq 2^{-k} \enorm{ \vct{x} } +
	20 \nu.
$$
\end{theorem}

We will first show this holds for sparse input signals, and then derive the general case.

When the sampling matrix satisfies the restricted isometry inequalities \eqref{eq:RIC2}, it has several other properties that we require repeatedly in the proof that the CoSaMP algorithm is correct.
Our first observation is a simple translation of \eqref{eq:RIC2} into other terms, in the same light as Proposition~\ref{P:cons} used in the proof of ROMP.

\begin{proposition} \label{prop:rip-basic}
Suppose $\Fee$ has restricted isometry constant $\delta_r$.  Let $T$ be a set of $r$ indices or fewer.  Then 
\begin{align*}
\phantom{\frac{1}{\sqrt{\delta_r}}}
\smnorm{2}{ \Fee_T^\adj \vct{u} }
	&\leq \sqrt{1 + \delta_r} \enorm{ \vct{u} } \\
\phantom{\frac{1}{\sqrt{\delta_r}}}
\smnorm{2}{ \Fee_T^\psinv \vct{u} }
	&\leq \frac{1}{\sqrt{1 - \delta_r}} \enorm{ \vct{u} } \\
\phantom{\frac{1}{\sqrt{\delta_r}}}
\smnorm{2}{ \Fee_T^\adj \Fee_T \vct{x} }
	&\lesseqqgtr (1 \pm \delta_r) \enorm{ \vct{x}} \\
\phantom{\frac{1}{\sqrt{\delta_r}}}
\smnorm{2}{ (\Fee_T^\adj \Fee_T)^{-1} \vct{x} }
	&\lesseqqgtr \frac{1}{1 \pm \delta_r} \enorm{ \vct{x}}.
\end{align*}
where the last two statements contain an upper and lower bound, depending on the sign chosen.
\end{proposition}

\begin{proof}
The restricted isometry inequalities \eqref{eq:RIC2} imply that the singular values of $\Fee_T$ lie between $\sqrt{1 - \delta_r}$ and $\sqrt{1 + \delta_r}$. The bounds follow from standard relationships between the singular values of $\Fee_T$ and the singular values of basic functions of $\Fee_T$.
\end{proof}

A second consequence is that disjoint sets of columns from the sampling matrix span nearly orthogonal subspaces.  The following result quantifies this observation.

\begin{proposition}[Approximate Orthogonality] \label{prop:approx-orth}
Suppose $\Fee$ has restricted isometry constant $\delta_r$.  Let $S$ and $T$ be disjoint sets of indices whose combined cardinality does not exceed $r$.  Then
$$
\norm{ \Fee_S^\adj \Fee_T } \leq \delta_{r}.
$$
\end{proposition}

\begin{proof}
Abbreviate $R = S \cup T$, and observe that $\Fee_S^\adj \Fee_T$ is a submatrix of $\Fee_R^\adj \Fee_R - \Id$.  The spectral norm of a submatrix never exceeds the norm of the entire matrix.  We discern that
$$
\norm{ \Fee_S^\adj \Fee_T }
	\leq \norm{ \Fee_R^\adj \Fee_R - \Id }
	\leq \max\{ (1+\delta_r) - 1, 1 - (1 - \delta_r) \}
	= \delta_r
$$
because the eigenvalues of $\Fee_R^\adj \Fee_R$ lie between $1 - \delta_r$ and $1 + \delta_r$.
\end{proof}

This result will be applied through the following corollary.

\begin{corollary} \label{cor:cross-corr}
Suppose $\Fee$ has restricted isometry constant $\delta_r$.  Let $T$ be a set of indices, and let $\vct{x}$ be a vector.  Provided that $r \geq \abs{T \cup \supp{\vct{x}}}$,
$$
\enorm{ \Fee_T^\adj \Fee \cdot \vct{x}\restrict{T^c} }
	\leq \delta_{r} \enorm{ \vct{x}\restrict{T^c} }.
$$
\end{corollary}

\begin{proof}
Define $S = \supp{\vct{x}} \setminus T$, so we have $\vct{x}\restrict{S} = \vct{x}\restrict{T^c}$.  Thus,
$$
\enorm{ \Fee_T^\adj \Fee \cdot \vct{x}\restrict{T^c} }
	= \enorm{ \Fee_T^\adj \Fee \cdot \vct{x}\restrict{S} }
	\leq \norm{ \Fee_T^\adj \Fee_S } \enorm{ \vct{x}\restrict{S} }
	\leq \delta_{r} \enorm{ \vct{x}\restrict{T^c} },
$$
owing to Proposition~\ref{prop:approx-orth}.
\end{proof}

As a second corollary, we show that $\delta_{2r}$ gives weak control over the higher restricted isometry constants.

\begin{corollary} \label{cor:dumb-rip-bd}
Let $c$ and $r$ be positive integers.  Then $\delta_{cr} \leq c \cdot \delta_{2r}$.
\end{corollary}

\begin{proof}
The result is clearly true for $c = 1, 2,$ so we assume $c \geq 3$.  Let $S$ be an arbitrary index set of size $cr$, and let $\mtx{M} = \Fee_S^\adj \Fee_S - \Id$.  It suffices to check that $\norm{ \mtx{M} } \leq c \cdot \delta_{2r}$.  To that end, we break the matrix $\mtx{M}$ into $r \times r$ blocks, which we denote $\mtx{M}_{ij}$.  A block version of Gershgorin's theorem states that $\norm{\mtx{M}}$ satisfies at least one of the inequalities
$$
\abs{ \norm{ \mtx{M} } - \norm{\mtx{M}_{ii}} } \leq \sum\nolimits_{j\neq i} \norm{ \mtx{M}_{ij} }
\qquad\text{where $i = 1, 2, \dots, c$.}
$$
The derivation is entirely analogous with the usual proof of Gershgorin's theorem, so we omit the details.  For each diagonal block, we have $\norm{ \mtx{M}_{ii} } \leq \delta_r$ because of the restricted isometry inequalities \eqref{eq:RIC2}.  For each off-diagonal block, we have $\norm{ \mtx{M}_{ij} } \leq \delta_{2r}$ because of Proposition~\ref{prop:approx-orth}.  Substitute these bounds into the block Gershgorin theorem and rearrange to complete the proof.
\end{proof}

Finally, we present a result that measures how much the sampling matrix inflates nonsparse vectors.  This bound permits us to establish the major results for sparse signals and then transfer the conclusions to the general case.

\begin{proposition}[Energy Bound] \label{prop:k2-bd}
Suppose that $\Fee$ verifies the upper inequality of \eqref{eq:RIC2}, viz.
$$
\enorm{ \Fee \vct{x} } \leq \sqrt{1 + \delta_r} \enorm{ \vct{x} }
\qquad\text{when}\qquad
\pnorm{0}{\vct{x}} \leq r.
$$
Then, for every signal $\vct{x}$,
$$
\enorm{ \Fee \vct{x} } \leq \sqrt{1 + \delta_r}
	\left[ \enorm{ \vct{x} } + \frac{1}{\sqrt{r}}
		\pnorm{1}{\vct{x}} \right].
$$
\end{proposition}

\begin{proof}

First, observe that the hypothesis of the proposition can be regarded as a statement about the operator norm of $\Fee$ as a map between two Banach spaces.  For a set $I \subset \{1, 2, \dots, N\}$, write $B_2^I$ for the unit ball in $\ell_2(I)$.  Define the convex body
$$
S = \conv\left\{ \bigcup\nolimits_{\abs{I} \leq r} B_2^I \right\}
	\subset \Cspace{N},
$$
and notice that, by hypothesis, the operator norm
$$
\pnorm{S \to 2}{ \Fee } =
	\max_{\vct{x} \in S} \enorm{ \Fee \vct{x} }
	\leq \sqrt{1 + \delta_r}.
$$
Define a second convex body 
$$
K = \left\{ \vct{x} : \enorm{\vct{x}} + \frac{1}{\sqrt{r}} \pnorm{1}{\vct{x}}
	\leq 1 \right\} \subset \Cspace{N},
$$
and consider the operator norm
$$
\pnorm{K \to 2}{ \Fee } =
	\max_{\vct{x} \in K} \enorm{ \Fee \vct{x} }.
$$
The content of the proposition is the claim that
$$
\pnorm{K \to 2}{ \Fee } \leq \pnorm{S \to 2}{ \Fee }.
$$
To establish this point, it suffices to check that $K \subset S$.

Choose a vector $\vct{x} \in K$.  We partition the support of $\vct{x}$ into sets of size $r$.  Let $I_0$ index the $r$ largest-magnitude components of $\vct{x}$, breaking ties lexicographically.  Let $I_1$ index the next largest $r$ components, and so forth.  Note that the final block $I_J$ may have fewer than $r$ components.  We may assume that $\vct{x}\restrict{I_j}$ is nonzero for each $j$.

This partition induces a decomposition
$$
\vct{x} = \vct{x}\restrict{I_0} + \sum\nolimits_{j = 0}^J \vct{x}\restrict{I_j}
	= \lambda_0 \vct{y}_0 + \sum\nolimits_{j=0}^J \lambda_j \vct{y}_j
$$
where
$$
\lambda_j = \smnorm{2}{ \vct{x}\restrict{I_j} }
\quad\text{and}\quad
\vct{y}_j = \lambda_j^{-1} \vct{x}\restrict{I_j}.
$$
By construction, each vector $\vct{y}_j$ belongs to $S$ because it is $r$-sparse and has unit $\ell_2$ norm.  We will prove that $\sum_j\lambda_j \leq 1$, which implies that $\vct{x}$ can be written as a convex combination of vectors from the set $S$.  As a consequence, $\vct{x} \in S$.  It emerges that $K \subset S$.

Fix $j$ in the range $\{1, 2, \dots, J\}$.  It follows that $I_j$ contains at most $r$ elements and $I_{j-1}$ contains exactly $r$ elements.  Therefore,
$$
\lambda_j = \enorm{ \vct{x}\restrict{I_j} }
	\leq \sqrt{r} \infnorm{\vct{x}\restrict{I_j} }
	\leq \sqrt{r} \cdot \frac{1}{r} \pnorm{1}{ \vct{x} \restrict{I_{j-1}} }
$$
where the last inequality holds because the magnitude of $\vct{x}$ on the set $I_{j-1}$ dominates its largest entry in $I_j$.  Summing these relations, we obtain
$$
\sum\nolimits_{j=1}^J \lambda_j
	\leq \frac{1}{\sqrt{r}} \sum\nolimits_{j=1}^J \pnorm{1}{ \vct{x}\restrict{I_{j-1}}}
	= \frac{1}{\sqrt{r}} \pnorm{1}{\vct{x}}.
$$
It is clear that $\lambda_0 = \enorm{ \vct{x}\restrict{I_0} } \leq \enorm{ \vct{x} }$.  We may conclude that
$$
\sum\nolimits_{j=0}^J \lambda_j
	\leq \enorm{\vct{x}} + \frac{1}{\sqrt{r}} \pnorm{1}{\vct{x}}
	\leq 1
$$
because $\vct{x} \in K$.
\end{proof}

\subsubsection[Iteration Invariant: Sparse Case]{Iteration Invariant: Sparse Case}

We now commence the proof of Theorem~\ref{thm:cosamp-invar}.  For the moment, let us assume that the signal is actually sparse.  We will remove this assumption later.

The result states that each iteration of the algorithm reduces the approximation error by a constant factor, while adding a small multiple of the noise.  As a consequence, when the approximation error is large in comparison with the noise, the algorithm makes substantial progress in identifying the unknown signal.

\begin{theorem}[Iteration Invariant: Sparse Case]
	\label{thm:invar-sparse}
Assume that $\vct{x}$ is $s$-sparse.  For each $k \geq 0$, the signal approximation $\vct{a}^k$ is $s$-sparse, and
$$
\smnorm{2}{ \vct{x} - \vct{a}^{k+1} }
	\leq 0.5 \smnorm{2}{ \vct{x} - \vct{a}^{k} } + 7.5 \enorm{ \vct{e} }.
$$
In particular,
$$
\smnorm{2}{ \vct{x} - \vct{a}^{k} }
	\leq 2^{-k} \enorm{ \vct{x} } + 15 \enorm{ \vct{e} }.
$$
\end{theorem}

\noindent
The argument proceeds in a sequence of short lemmas, each corresponding to one step in the algorithm.  Throughout this section, we retain the assumption that $\vct{x}$ is $s$-sparse.

Fix an iteration $k \geq 1$.  We write $\vct{a} = \vct{a}^{k-1}$ for the signal approximation at the beginning of the iteration.  Define the residual $\vct{r} = \vct{x} - \vct{a}$, which we interpret as the part of the signal we have not yet recovered.  Since the approximation $\vct{a}$ is always $s$-sparse, the residual $\vct{r}$ must be $2s$-sparse.  Notice that the vector $\vct{v}$ of updated samples can be viewed as noisy samples of the residual:
$$
\vct{v} \defby \vct{u} - \Fee \vct{a} = \Fee (\vct{x} - \vct{a}) + \vct{e}
	= \Fee \vct{r} + \vct{e}.
$$

The identification phase produces a set of components where the residual signal still has a lot of energy.  

\begin{lemma}[Identification] \label{lem:ident}
The set $\Omega = \supp{ \vct{y}_{2s} }$, where $\vct{y} = \Fee^\adj \vct{v}$ is the signal proxy, contains at most $2s$ indices, and
$$
\enorm{ \vct{r}\restrict{\Omega^c} }
	\leq 0.2223 \enorm{ \vct{r} } + 2.34 \enorm{\vct{e}}.
$$
\end{lemma}

\begin{proof}
The identification phase forms a proxy $\vct{y} = \Fee^\adj \vct{v}$ for the residual signal.  The algorithm then selects a set $\Omega$ of $2s$ components from $\vct{y}$ that have the largest magnitudes.  The goal of the proof is to show that the energy in the residual on the set $\Omega^c$ is small in comparison with the total energy in the residual.

Define the set $R = \supp{\vct{r}}$.  Since $R$ contains at most $2s$ elements, our choice of $\Omega$ ensures that $\enorm{ \vct{y}\restrict{R} } \leq \enorm{ \vct{y}\restrict{\Omega} }$.  By squaring this inequality and canceling the terms in $R \cap \Omega$, we discover that
$$
\enorm{ \vct{y}\restrict{R \setminus \Omega} }
	\leq \enorm{ \vct{y}\restrict{\Omega \setminus R} }.
$$
Since the coordinate subsets here contain few elements, we can use the restricted isometry constants to provide bounds on both sides.

First, observe that the set $\Omega \setminus R$ contains at most $2s$ elements.  Therefore, we may apply Proposition~\ref{prop:rip-basic} and Corollary~\ref{cor:cross-corr} to obtain
\begin{align*}
\enorm{ \vct{y}\restrict{\Omega \setminus R} }
	&= \smnorm{2}{ \Fee_{\Omega \setminus R}^\adj (\Fee \vct{r} + \vct{e}) } \\
	&\leq \smnorm{2}{ \Fee_{\Omega \setminus R}^\adj \Fee \vct{r} }
		+ \smnorm{2}{ \Fee_{\Omega \setminus R}^\adj \vct{e} } \\
	&\leq \delta_{4s} \enorm{ \vct{r} }
		+ \sqrt{1 + \delta_{2s}} \enorm{ \vct{e} }.
\end{align*}
Likewise, the set $R \setminus \Omega$ contains $2s$ elements or fewer, so Proposition~\ref{prop:rip-basic} and Corollary~\ref{cor:cross-corr} yield
\begin{align*}
\enorm{ \vct{y}\restrict{R \setminus \Omega} }
	&= \smnorm{2}{ \Fee_{R \setminus \Omega}^\adj (\Fee \vct{r} + \vct{e}) } \\
	&\geq \smnorm{2}{ \Fee_{R \setminus \Omega}^\adj \Fee
			\cdot \vct{r}\restrict{R \setminus \Omega} }
		- \smnorm{2}{ \Fee_{R \setminus \Omega}^\adj \Fee
			\cdot \vct{r}\restrict{\Omega} }
		- \smnorm{2}{ \Fee_{R \setminus \Omega}^\adj \vct{e} } \\
	&\geq (1 - \delta_{2s}) \smnorm{2}{ \vct{r}\restrict{R \setminus \Omega} }
		- \delta_{2s} \enorm{ \vct{r} }
		- \sqrt{1 + \delta_{2s}} \enorm{ \vct{e} }. 
\end{align*}
Since the residual is supported on $R$, we can rewrite $\vct{r}\restrict{R \setminus \Omega} = \vct{r}\restrict{\Omega^c}$.  Finally, combine the last three inequalities and rearrange to obtain
$$
\enorm{ \vct{r}\restrict{\Omega^c} }
	\leq \frac{ (\delta_{2s} + \delta_{4s}) \enorm{ \vct{r} }
		+ 2 \sqrt{1 + \delta_{2s}} \enorm{ \vct{e} } }
		{ 1 - \delta_{2s} }.
$$
Invoke the numerical hypothesis that $\delta_{2s} \leq \delta_{4s} \leq 0.1$ to complete the argument.
\end{proof}

The next step of the algorithm merges the support of the current signal approximation $\vct{a}$ with the newly identified set of components.  The following result shows that components of the signal $\vct{x}$ outside this set have very little energy.

\begin{lemma}[Support Merger] \label{lem:merger}
Let $\Omega$ be a set of at most $2s$ indices.  The set $T = \Omega \cup \supp{\vct{a}}$ contains at most $3s$ indices, and
$$
\enorm{ \vct{x}\restrict{T^c} }
	\leq \enorm{ \vct{r}\restrict{\Omega^c} }.
$$
\end{lemma}

\begin{proof}
Since $\supp{\vct{a}} \subset T$, we find that
$$
\enorm{ \vct{x}\restrict{T^c} }
	= \enorm{ (\vct{x} - \vct{a})\restrict{T^c} }
	= \enorm{ \vct{r}\restrict{T^c} }
	\leq \enorm{ \vct{r}\restrict{\Omega^c} },
$$
where the inequality follows from the containment $T^c \subset \Omega^c$.
\end{proof}

The estimation step of the algorithm solves a least-squares problem to obtain values for the coefficients in the set $T$.  We need a bound on the error of this approximation.

\begin{lemma}[Estimation] \label{lem:estimation}
Let $T$ be a set of at most $3s$ indices, and define the least-squares signal estimate $\vct{b}$ by the formulae
$$
\vct{b}\restrict{T} = \Fee_T^\psinv \vct{u}
\qquad\text{and}\qquad
\vct{b}\restrict{T^c} = \vct{0},
$$
where $\vct{u} = \Fee\vct{x} + \vct{e}$. Then
$$
\enorm{ \vct{x} - \vct{b} }
	\leq 1.112 \enorm{ \vct{x}\restrict{T^c} } + 1.06 \enorm{\vct{e}}.
$$
\end{lemma}

This result assumes that we solve the least-squares problem in infinite precision.  In practice, the right-hand side of the bound contains an extra term owing to the error from the iterative least-squares solver.  Below, we study how many iterations of the least-squares solver are required to make the least-squares error negligible in the present argument.

\begin{proof}
Note first that
$$
\enorm{ \vct{x} - \vct{b} }
	\leq \enorm{ \vct{x}\restrict{T^c} }
		+ \enorm{ \vct{x}\restrict{T} - \vct{b}\restrict{T} }.
$$
Using the expression $\vct{u} = \Fee\vct{x} + \vct{e}$ and the fact $\Fee_T^\psinv \Fee_T = \Id_T$, we calculate that
\begin{align*}
\enorm{ \vct{x}\restrict{T} - \vct{b}\restrict{T} }
	&= \smnorm{2}{ \vct{x}\restrict{T} - \Fee_T^\psinv(\Fee \cdot \vct{x}\restrict{T} + \Fee \cdot \vct{x}\restrict{T^c} + \vct{e}) } \\
	&=  \smnorm{2}{ \Fee_T^\psinv (\Fee \cdot \vct{x}\restrict{T^c}
		+ \vct{e}) } \\
	&\leq \smnorm{2}{ (\Fee_T^\adj \Fee_T)^{-1} \Fee_T^\adj \Fee \cdot \vct{x}\restrict{T^c} } + \smnorm{2}{ \Fee_T^\psinv \vct{e} }.
\end{align*}
The cardinality of $T$ is at most $3s$, and $\vct{x}$ is $s$-sparse, so Proposition~\ref{prop:rip-basic} and Corollary~\ref{cor:cross-corr} imply that
\begin{align*}
\enorm{ \vct{x}\restrict{T} - \vct{b}\restrict{T} }
	&\leq \frac{1}{1 - \delta_{3s}} \enorm{ \Fee_T^\adj \Fee \cdot \vct{x}\restrict{T^c} } + \frac{1}{\sqrt{1 - \delta_{3s}}} \enorm{ \vct{e}} \\
	&\leq \frac{\delta_{4s}}{1 - \delta_{3s}}
		\enorm{ \vct{x}\restrict{T^c} }
		+ \frac{\enorm{\vct{e}}}{\sqrt{1 - \delta_{3s}}}.
\end{align*}
Combine the bounds to reach
$$
\enorm{ \vct{x} - \vct{b} } \leq
	\left[ 1 + \frac{\delta_{4s}}{1 - \delta_{3s}} \right]
		\enorm{ \vct{x}\restrict{T^c} }
	+ \frac{\enorm{\vct{e}}}{\sqrt{1 - \delta_{3s}}}.
$$
Finally, invoke the hypothesis that $\delta_{3s} \leq \delta_{4s} \leq 0.1$.
\end{proof}

The final step of each iteration is to prune the intermediate approximation to its largest $s$ terms.  The following lemma provides a bound on the error in the pruned approximation.

\begin{lemma}[Pruning] \label{lem:pruning}
The pruned approximation $\vct{b}_s$ satisfies
$$
\enorm{ \vct{x} - \vct{b}_s }
	\leq 2 \enorm{ \vct{x} - \vct{b} }.
$$
\end{lemma}

\begin{proof}
The intuition is that $\vct{b}_s$ is close to $\vct{b}$, which is close to $\vct{x}$.  Rigorously,
$$
\enorm{ \vct{x} - \vct{b}_s }
	\leq \enorm{ \vct{x} - \vct{b} } + \enorm{ \vct{b} - \vct{b}_s }
	\leq 2 \enorm{ \vct{x} - \vct{b} }.
$$
The second inequality holds because $\vct{b}_s$ is the best $s$-sparse approximation to $\vct{b}$.  In particular, the $s$-sparse vector $\vct{x}$ is a worse approximation.
\end{proof}

We now complete the proof of the iteration invariant for sparse signals, Theorem~\ref{thm:invar-sparse}.  At the end of an iteration, the algorithm forms a new approximation $\vct{a}^{k} = \vct{b}_s$, which is evidently $s$-sparse.  Applying the lemmas we have established, we easily bound the error:
\begin{align*}
\smnorm{2}{ \vct{x} - \vct{a}^k }
	&= \enorm{ \vct{x} - \vct{b}_s } \\
	&\leq 2 \enorm{ \vct{x} - \vct{b} }
		&& \text{Pruning (Lemma \ref{lem:pruning})} \\
	&\leq 2 \cdot ( 1.112 \enorm{ \vct{x}\restrict{T^c} }
		+ 1.06\enorm{\vct{e}} )
		&& \text{Estimation (Lemma \ref{lem:estimation})} \\
	&\leq 2.224 \enorm{ \vct{r}\restrict{\Omega^c} } + 2.12 \enorm{\vct{e}}
		&& \text{Support merger (Lemma \ref{lem:merger})} \\
	&\leq 2.224 \cdot ( 0.2223 \enorm{ \vct{r} } + 2.34 \enorm{\vct{e}} )
		+ 2.12 \enorm{\vct{e}}
		&& \text{Identification (Lemma \ref{lem:ident})} \\
	&< 0.5 \enorm{ \vct{r} } + 7.5 \enorm{\vct{e}} \\
	&= 0.5 \smnorm{2}{ \vct{x} - \vct{a}^{k-1} } + 7.5 \enorm{\vct{e}}.
\end{align*}
To obtain the second bound in Theorem~\ref{thm:invar-sparse}, simply solve the error recursion and note that
$$
(1 + 0.5 + 0.25 + \dots) \cdot 7.5 \enorm{\vct{e}} \leq 15 \enorm{\vct{e}}.
$$
This point completes the argument.  

Before extending the iteration invariant to the sparse case, we first analyze in detail the least-sqaures step.  This will allow us to completely prove our main result, Theorem~\ref{thm:cosamp}.

\subsubsection[Least Squares Analysis]{Least Squares Analysis}\label{sec:LSA}

To develop an efficient implementation of CoSaMP, it is critical to use an iterative method when we solve the least-squares problem in the estimation step.  Here we analyze this step for the noise-free case.  The two natural choices are Richardson's iteration and conjugate gradient.  The efficacy of these methods rests on the assumption that the sampling operator has small restricted isometry constants.  Indeed, since the set $T$ constructed in the support merger step contains at most $3s$ components, the hypothesis $\delta_{4s} \leq 0.1$ ensures that the condition number
$$
\kappa( \Fee_T^\adj \Fee_T )
	= \frac{\lambda_{\max}(\Fee_T^\adj \Fee_T)}{\lambda_{\min}(\Fee_T^\adj \Fee_T)}
	\leq \frac{1 + \delta_{3s}}{1 - \delta_{3s}}
	< 1.223.
$$
This condition number is closely connected with the performance of Richardson's iteration and conjugate gradient.  In this section, we show that Theorem~\ref{thm:invar-sparse} holds if we perform a constant number of iterations of either least-squares algorithm.

For completeness, let us explain how Richardson's iteration can be applied to solve the least-squares problems that arise in CoSaMP.  Suppose we wish to compute $\mtx{A}^\psinv \vct{u}$ where $\mtx{A}$ is a tall, full-rank matrix.  Recalling the definition of the pseudoinverse, we realize that this amounts to solving a linear system of the form
$$
(\mtx{A}^\adj \mtx{A}) \vct{b} = \mtx{A}^\adj \vct{u}.
$$
This problem can be approached by \term{splitting} the Gram matrix:
$$
\mtx{A}^\adj \mtx{A} = \Id + \mtx{M}
$$
where $\mtx{M} = \mtx{A}^\adj \mtx{A} - \Id$.  Given an initial iterate $\vct{z}^0$, Richardon's method produces  the subsequent iterates via the formula
$$
\vct{z}^{\ell+1} = \mtx{A}^\adj \vct{u} - \mtx{M} \vct{z}^{\ell}.
$$
Evidently, this iteration requires only matrix--vector multiplies with $\mtx{A}$ and $\mtx{A}^\adj$.  It is worth noting that Richardson's method can be accelerated \cite[Sec.~7.2.5]{Bjo96:Numerical-Methods}, but we omit the details.

It is quite easy to analyze Richardson's iteration~\cite[Sec.~7.2.1]{Bjo96:Numerical-Methods}.  Observe that
$$
\smnorm{2}{ \vct{z}^{\ell+1} - \mtx{A}^\psinv \vct{u} }
	= \smnorm{2}{ \mtx{M} ( \vct{z}^\ell - \mtx{A}^{\psinv} \vct{u}) }
	\leq \norm{ \mtx{M} } \smnorm{2}{ \vct{z}^\ell - \mtx{A}^{\psinv} \vct{u} }.
$$
This recursion delivers
$$
\smnorm{2}{ \vct{z}^{\ell} - \mtx{A}^\psinv \vct{u} }
	\leq \norm{ \mtx{M} }^\ell \smnorm{2}{ \vct{z}^0 - \mtx{A}^\psinv \vct{u} }
\qquad\text{for $\ell = 0, 1, 2, \dots$.}
$$
In words, the iteration converges linearly.

In our setting, $\mtx{A} = \Fee_T$ where $T$ is a set of at most $3s$ indices.  Therefore, the restricted isometry inequalities \eqref{eq:RIC2} imply that
$$
\norm{ \mtx{M} } = \norm{ \Fee_T^\adj \Fee_T - \Id } \leq \delta_{3s}.
$$
We have assumed that $\delta_{3s} \leq \delta_{4s} \leq 0.1$, which means that the iteration converges quite fast.  Once again, the restricted isometry behavior of the sampling matrix plays an essential role in the performance of the CoSaMP algorithm.

Conjugate gradient provides even better guarantees for solving the least-squares problem, but it is somewhat more complicated to describe and rather more difficult to analyze.  We refer the reader to \cite[Sec.~7.4]{Bjo96:Numerical-Methods} for more information.  The following lemma summarizes the behavior of both Richardson's iteration and conjugate gradient in our setting.

\begin{lemma}[Error Bound for LS]
	\label{lem:ls-error}
Richardson's iteration produces a sequence $\{ \vct{z}^\ell \}$ of iterates that satisfy
$$
\smnorm{2}{ \vct{z}^\ell - \Fee_T^\psinv \vct{u} }
	\leq 0.1^\ell \cdot \smnorm{2}{ \vct{z}^0 - \Fee_T^\psinv \vct{u} }
\qquad\text{for $\ell = 0, 1, 2, \dots$}. 
$$ 
Conjugate gradient produces a sequence of iterates that satisfy
$$
\smnorm{2}{ \vct{z}^\ell - \Fee_T^\psinv \vct{u} }
	\leq 2 \cdot \rho^\ell \cdot \smnorm{2}{ \vct{z}^0 - \Fee_T^\psinv \vct{u} }
\qquad\text{for $\ell = 0, 1, 2, \dots$}. 
$$
where
$$
\rho = \frac{ \sqrt{\kappa(\Fee_T^\adj \Fee_T)} - 1 }{\sqrt{\kappa(\Fee_T^\adj \Fee_T)} + 1} \leq 0.072.
$$
\end{lemma}

This can even be improved further if the eigenvalues of $\Phi_T^*\Phi_T$ are clustered~\cite{deanna}.  Iterative least-squares algorithms must be seeded with an initial iterate, and their performance depends heavily on a wise selection thereof.  CoSaMP offers a natural choice for the initializer: the current signal approximation.  As the algorithm progresses, the current signal approximation provides an increasingly good starting point for solving the least-squares problem.
The following shows that the error in the initial iterate is controlled by the current approximation error.

\begin{lemma}[Initial Iterate for LS]
	\label{lem:ls-init}
Let $\vct{x}$ be an $s$-sparse signal with noisy samples $\vct{u} = \Fee \vct{x} + \vct{e}$.  Let $\vct{a}^{k-1}$ be the signal approximation at the end of the $(k-1)$th iteration, and let $T$ be the set of components identified by the support merger. Then
$$
\smnorm{2}{ \vct{a}^{k-1} - \Fee_T^\psinv \vct{u}  }
	\leq 2.112 \smnorm{2}{ \vct{x} - \vct{a}^{k-1} } + 1.06 \enorm{ \vct{e} }
$$
\end{lemma}

\begin{proof}
By construction of $T$, the approximation $\vct{a}^{k-1}$ is supported inside $T$, so
$$
\smnorm{2}{ \vct{x}\restrict{T^c} }
	= \smnorm{2}{(\vct{x} - \vct{a}^{k-1})\restrict{T^c} }
	\leq \smnorm{2}{\vct{x} - \vct{a}^{k-1}}.
$$
Using Lemma~\ref{lem:estimation}, we may calculate how far $\vct{a}^{k-1}$ lies from the solution to the least-squares problem.
\begin{align*}
\smnorm{2}{ \vct{a}^{k-1} - \Fee_T^\psinv \vct{u} }
	&\leq \smnorm{2}{ \vct{x} - \vct{a}^{k-1} }
		+ \smnorm{2}{ \vct{x} - \Fee_T^\psinv \vct{u} } \\
	&\leq \smnorm{2}{ \vct{x} - \vct{a}^{k-1} }
		+ 1.112 \smnorm{2}{ \vct{x}\restrict{T^c} }
		+ 1.06 \enorm{ \vct{e} } \\
	&\leq 2.112 \smnorm{2}{ \vct{x} - \vct{a}^{k-1} } + 1.06 \enorm{ \vct{e} }.
\end{align*}

\end{proof}

We need to determine how many iterations of the least-squares algorithm are required to ensure that the approximation produced is sufficiently good to support the performance of CoSaMP.

\begin{corollary}[Estimation by Iterative LS]
	\label{cor:ls-est}
Suppose that we initialize the LS algorithm with $\vct{z}^0 = \vct{a}^{k-1}$.  After at most three iterations, both Richardson's iteration and conjugate gradient produce a signal estimate $\vct{b}$ that satisfies
$$
\smnorm{2}{ \vct{x} - \vct{b} } \leq
	1.112 \smnorm{2}{ \vct{x}\restrict{T^c} }
	+ 0.0022 \smnorm{2}{ \vct{x} - \vct{a}^{k-1} } + 1.062 \enorm{ \vct{e} }.
$$
\end{corollary}

\begin{proof}
Combine Lemma~\ref{lem:ls-error} and Lemma~\ref{lem:ls-init} to see that three iterations of Richardson's method yield
$$
\smnorm{2}{ \vct{z}^3 - \Fee_T^\psinv \vct{u}  }
	\leq 0.002112 \smnorm{2}{ \vct{x} - \vct{a}^{k-1} } + 0.00106 \enorm{ \vct{e} }.
$$
The bound for conjugate gradient is slightly better.  Let $\vct{b}\restrict{T} = \vct{z}^3$.  According to the estimation result, Lemma~\ref{lem:estimation}, we have
$$
\smnorm{2}{ \vct{x} - \Fee_T^\psinv \vct{u} } \leq
	1.112 \smnorm{2}{ \vct{x}\restrict{T^c} } + 1.06 \enorm{ \vct{e} }.
$$
An application of the triangle inequality completes the argument.
\end{proof}

Finally, we need to check that the sparse iteration invariant, Theorem~\ref{thm:invar-sparse} still holds when we use an iterative least-squares algorithm.

\begin{theorem}[Sparse Iteration Invariant with Iterative LS] \label{thm:invar-sparse-ls}
Suppose that we use Richardson's iteration or conjugate gradient for the estimation step, initializing the LS algorithm with the current approximation $\vct{a}^{k-1}$ and performing three LS iterations.  Then Theorem~\ref{thm:invar-sparse} still holds.
\end{theorem}

\begin{proof}
We repeat the calculation from the above case using Corollary~\ref{cor:ls-est} instead of the simple estimation lemma.  To that end, recall that the residual $\vct{r} = \vct{x} - \vct{a}^{k-1}$.  Then
\begin{align*}
\smnorm{2}{ \vct{x} - \vct{a}^k }
	&\leq 2 \enorm{ \vct{x} - \vct{b} } \\
	&\leq 2 \cdot ( 1.112 \enorm{ \vct{x}\restrict{T^c} }
		+ 0.0022 \enorm{ \vct{r} } + 1.062 \enorm{\vct{e}} ) \\
	&\leq 2.224 \enorm{ \vct{r}\restrict{\Omega^c} } + 0.0044 \enorm{\vct{r}} + 2.124 \enorm{\vct{e}} \\
	&\leq 2.224 \cdot ( 0.2223 \enorm{ \vct{r} } + 2.34 \enorm{\vct{e}} ) + 0.0044 \enorm{\vct{r}}
		+ 2.124 \enorm{\vct{e}} \\
	&< 0.5 \enorm{ \vct{r} } + 7.5 \enorm{\vct{e}} \\
	&= 0.5 \smnorm{2}{ \vct{x} - \vct{a}^{k-1} } + 7.5 \enorm{\vct{e}}.
\end{align*}
This bound is precisely what is required for the theorem to hold.
\end{proof}

We are now prepared to extend the iteration invariant to the general case, and prove Theorem~\ref{thm:cosamp}.

\subsubsection{Extension to General Signals}

In this section, we finally complete the proof of the main result for CoSaMP, Theorem~\ref{thm:cosamp-invar}.  The remaining challenge is to remove the hypothesis that the target signal is sparse, which we framed in Theorems~\ref{thm:invar-sparse} and~\ref{thm:invar-sparse-ls}.  Although this difficulty might seem large, the solution is simple and elegant.  It turns out that we can view the noisy samples of a general signal as samples of a sparse signal contaminated with a different noise vector that implicitly reflects the tail of the original signal.

\begin{lemma}[Reduction to Sparse Case] \label{lem:reduction}
Let $\vct{x}$ be an arbitrary vector in $\Cspace{N}$.  The sample vector $\vct{u} = \Fee \vct{x} + \vct{e}$ can also be expressed as $\vct{u} = \Fee \vct{x}_s + \widetilde{\vct{e}}$ where
$$
\enorm{ \widetilde{\vct{e}} } \leq
1.05 \left[ \enorm{ \vct{x} - \vct{x}_s } + \frac{1}{\sqrt{s}}
	\pnorm{1}{ \vct{x} - \vct{x}_s } \right] + \enorm{ \vct{e} }.
$$
\end{lemma}


\begin{proof}
Decompose $\vct{x} = \vct{x}_s + (\vct{x} - \vct{x}_s)$ to obtain $\vct{u} = \Fee \vct{x}_s + \widetilde{\vct{e}}$ where $\widetilde{\vct{e}} = \Fee (\vct{x} - \vct{x}_s) + \vct{e}$.
To compute the size of the error term, we simply apply the triangle inequality and Proposition~\ref{prop:k2-bd}:
$$
\enorm{ \widetilde{\vct{e}} }
	\leq \sqrt{1 + \delta_s} \left[ \enorm{ \vct{x} - \vct{x}_s } + \frac{1}{\sqrt{s}}
		\pnorm{1}{ \vct{x} - \vct{x}_s } \right]
	+ \enorm{ \vct{e} }.
$$
Finally, invoke the fact that $\delta_s \leq \delta_{4s} \leq 0.1$ to obtain $\sqrt{1 + \delta_s} \leq 1.05$.
\end{proof}

This lemma is just the tool we require to complete the proof of Theorem~\ref{thm:cosamp-invar}.

\begin{proof}[Proof of Theorem~\ref{thm:cosamp-invar}]
Let $\vct{x}$ be a general signal, and use Lemma~\ref{lem:reduction} to write the noisy vector of samples $\vct{u} = \Fee \vct{x}_s + \widetilde{\vct{e}}$.  Apply the sparse iteration invariant, Theorem~\ref{thm:invar-sparse}, or the analog for iterative least-squares, Theorem~\ref{thm:invar-sparse-ls}.  We obtain
$$
\smnorm{2}{ \vct{x}_s - \vct{a}^{k+1} }
	\leq 0.5 \smnorm{2}{ \vct{x}_s - \vct{a}^k }
		+ 7.5 \enorm{ \widetilde{\vct{e}}}.
$$
Invoke the lower and upper triangle inequalities to obtain
$$
\smnorm{2}{ \vct{x} - \vct{a}^{k+1} }
	\leq 0.5 \smnorm{2}{ \vct{x} - \vct{a}^k }
		+ 7.5 \enorm{ \widetilde{\vct{e}} }
		+ 1.5 \enorm{ \vct{x} - \vct{x}_s }.
$$
Finally, recall the estimate for $\enorm{\widetilde{\vct{e}}}$ from Lemma~\ref{lem:reduction}, and simplify to reach
\begin{align*}
\smnorm{2}{ \vct{x} - \vct{a}^{k+1} }
	&\leq 0.5 \smnorm{2}{ \vct{x} - \vct{a}^k }
		+ 9.375 \enorm{ \vct{x} - \vct{x}_s }
		+ \frac{7.875}{\sqrt{s}} \pnorm{1}{ \vct{x} - \vct{x}_s }
		+ 7.5 \enorm{ \vct{e} } \\
	&< 0.5 \smnorm{2}{ \vct{x} - \vct{a}^k } + 10 \nu.
\end{align*}
where $\nu$ is the unrecoverable energy \eqref{eqn:unrecoverable}.
\end{proof}

We have now collected all the material we need to establish the main result.
Fix a precision parameter $\eta$.  After at most $\bigO( \log(\enorm{\vct{x}} / \eta ) )$ iterations, CoSaMP produces an $s$-sparse approximation $\vct{a}$ that satisfies
$$
\enorm{ \vct{x} - \vct{a} } \leq \cnst{C} \cdot (\eta + \nu)
$$
in consequence of Theorem~\ref{thm:cosamp-invar}.  Apply inequality \eqref{eqn:unrecov-l1} to bound the unrecoverable energy $\nu$ in terms of the $\ell_1$ norm.  We see that the approximation error satisfies
$$
\enorm{ \vct{x} - \vct{a} } \leq \cnst{C} \cdot \max\left\{ \eta, \frac{1}{\sqrt{s}} \pnorm{1}{ \vct{x} - \vct{x}_{s/2}} + \enorm{\vct{e}} \right\}.
$$
According to Theorem~\ref{thm:resources}, each iteration of CoSaMP is completed in time $\bigO(\coll{L})$, where $\coll{L}$ bounds the cost of a matrix--vector multiplication with $\Fee$ or $\Fee^\adj$.  The total runtime, therefore, is $\bigO( \coll{L} \log(\enorm{\vct{x}} /\eta) )$.  The total storage is $\bigO(N)$.

Finally, in the statement of the theorem, we replace $\delta_{4s}$ with $\delta_{2s}$ by means of Corollary~\ref{cor:dumb-rip-bd}, which states that $\delta_{cr} \leq c \cdot \delta_{2r}$ for any positive integers $c$ and $r$.

\subsubsection{Iteration Count for Exact Arithmetic}

As promised, we now provide an iteration count for CoSaMP in the case of exact arithmetic.  These results can be found in~\cite{DT08:CoSaMP-TR}. 

We obtain an estimate on the number of iterations of the CoSaMP algorithm necessary to identify the recoverable energy in a sparse signal, assuming exact arithmetic.  Except where stated explicitly, we assume that $\vct{x}$ is $s$-sparse.  It turns out that the number of iterations depends heavily on the signal structure.  Let us explain the intuition behind this fact.  

When the entries in the signal decay rapidly, the algorithm must identify and remove the largest remaining entry from the residual before it can make further progress on the smaller entries.  Indeed, the large component in the residual contaminates each component of the signal proxy.  In this case, the algorithm may require an iteration or more to find each component in the signal.

On the other hand, when the $s$ entries of the signal are comparable, the algorithm can simultaneously locate many entries just by reducing the norm of the residual below the magnitude of the smallest entry.  Since the largest entry of the signal has magnitude at least $s^{-1/2}$ times the $\ell_2$ norm of the signal, the algorithm can find all $s$ components of the signal after about $\log s$ iterations. 

To quantify these intuitions, we want to collect the components of the signal into groups that are comparable with each other.  To that end, define the \term{component bands} of a signal $\vct{x}$ by the formulae
\begin{equation}\label{eq:bjs}
B_j \defby \left\{ i : 2^{-(j+1)} \enormsq{\vct{x}}
	< \abssq{ x_i } \leq 2^{-j} \enormsq{\vct{x}} \right\}
\qquad\text{for $j = 0, 1, 2, \dots$.}
\end{equation}
The \term{profile} of the signal is the number of bands that are nonempty:
$$
{\rm profile}(\vct{x}) \defby \# \{ j : B_j \neq \emptyset \}.
$$
In words, the profile counts how many orders of magnitude at which the signal has coefficients.  It is clear that the profile of an $s$-sparse signal is at most $s$.

See Figure~\ref{fig:profile} for images of stylized signals with different profiles.

 \begin{figure}[ht]
\begin{center}
$\begin{array}{c@{\hspace{.1in}}c}
\includegraphics[width=3in]{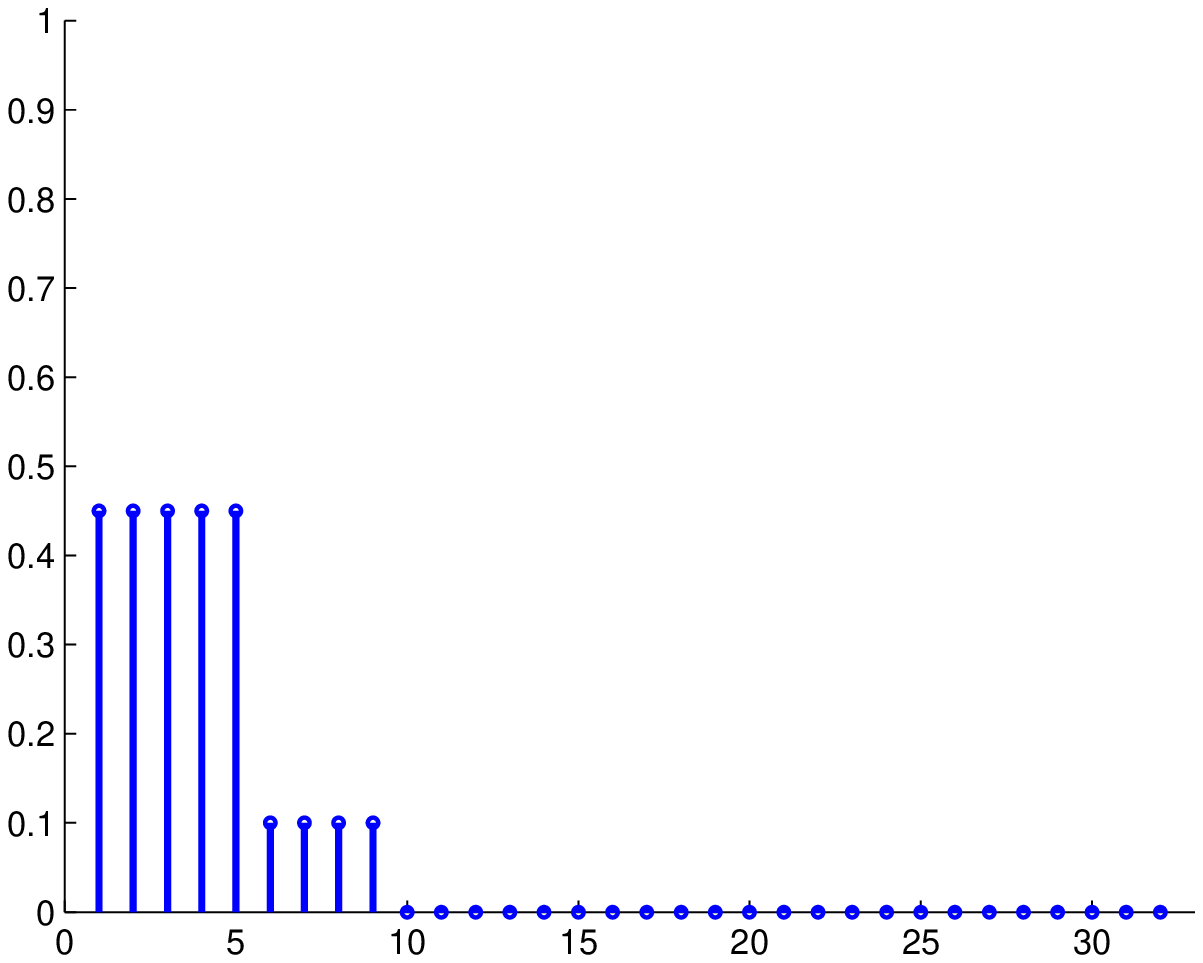}  &  
\includegraphics[width=3in]{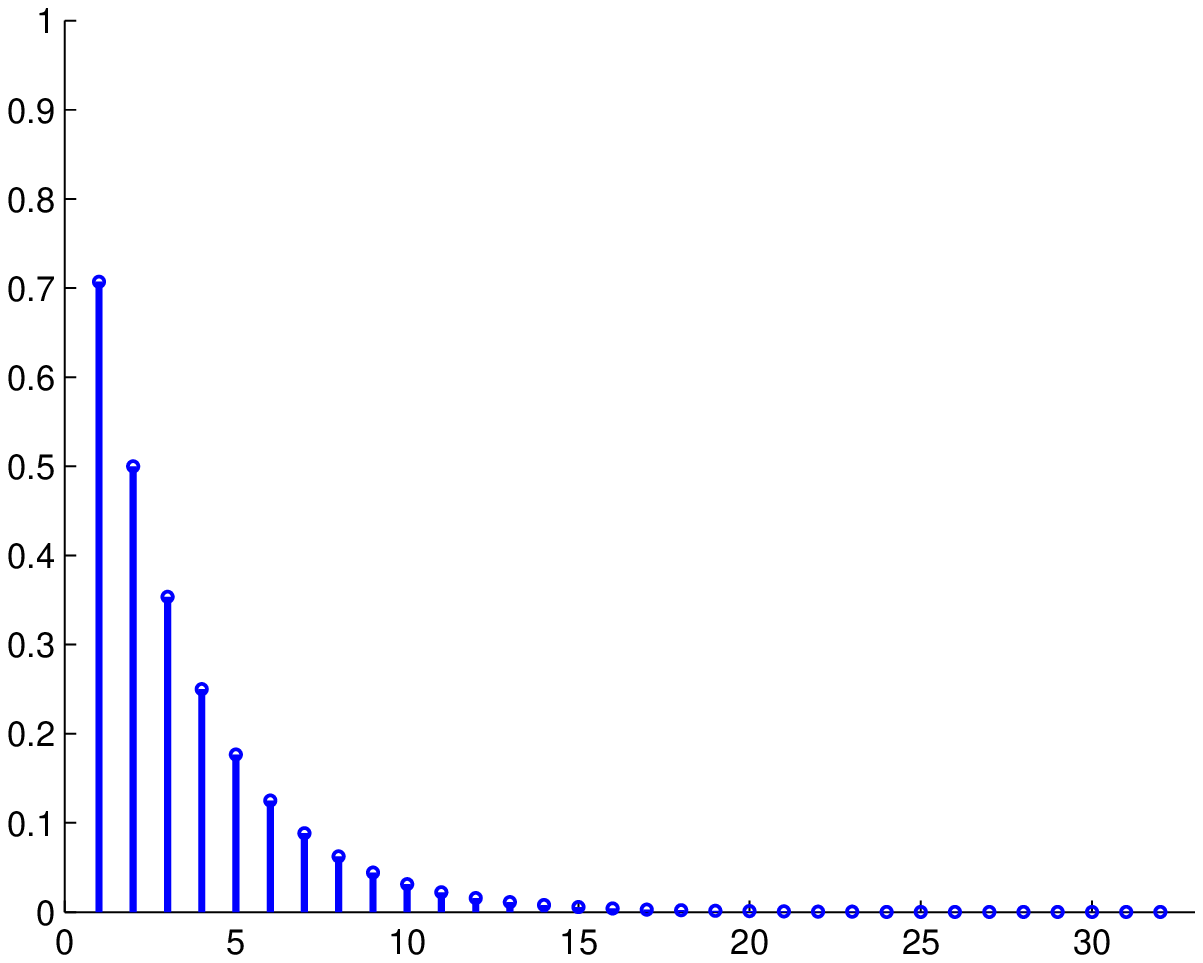}   \\
\end{array}$
\end{center}
\caption{Illustration of two unit-norm signals with sharply different profiles (left: low, right: high).}
\label{fig:profile}
\end{figure}


First, we prove a result on the number of iterations needed to acquire an $s$-sparse signal.  At the end of the section, we extend this result to general signals, which yields Theorem~\ref{thm:cosamp-count}.


\begin{theorem}[Iteration Count: Sparse Case] \label{thm:count-sparse}
Let $\vct{x}$ be an $s$-sparse signal, and define $p = {\rm profile}(\vct{x})$.  After at most
$$
p \log_{4/3}(1 + 4.6 \sqrt{s/p}) + 6
$$
iterations, CoSaMP produces an approximation $\vct{a}$ that satisfies
$$
\enorm{ \vct{x} - \vct{a} } \leq 17 \enorm{ \vct{e} }.
$$
\end{theorem}

For a fixed $s$, the bound on the number of iterations achieves its maximum value at $p = s$.  Since $\log_{4/3} 5.6 < 6$, the number of iterations never exceeds $6(s + 1)$.

Let us instate some notation that will be valuable in the proof of the theorem.  We write $p = {\rm profile}(\vct{x})$.  For each $k = 0, 1, 2, \dots$, the signal $\vct{a}^k$ is the approximation after the $k$th iteration.  We abbreviate $S_k = \supp{ \vct{a}^k }$, and we define the residual signal $\vct{r}^k = \vct{x} - \vct{a}^k$.  The norm of the residual can be viewed as the approximation error.

For a nonnegative integer $j$, we may define an auxiliary signal
$$
\vct{y}^j \defby \vct{x}\restrict{\bigcup_{i \geq j} B_i}
$$
In other words, $\vct{y}^j$ is the part of $\vct{x}$ contained in the bands $B_j$, $B_{j+1}$, $B_{j+2}$, \dots.  For each $j \in J$, we have the estimate
\begin{equation} \label{eqn:signal-tail-bd}
\enormsq{ \vct{y}^j } \leq \sum\nolimits_{i \geq j} 2^{-i} \enormsq{\vct{x}} \cdot \abs{B_i}
\end{equation}
by definition of the bands.  These auxiliary signals play a key role in the analysis.

The proof of the theorem involves a sequence of lemmas.  The first object is to establish an alternative that holds in each iteration.  One possibility is that the approximation error is small, which means that the algorithm is effectively finished.  Otherwise, the approximation error is dominated by the energy in the unidentified part of the signal, and the subsequent approximation error is a constant factor smaller.

\begin{lemma} \label{lem:iter-alter}
For each iteration $k = 0, 1, 2, \dots$, at least one of the following alternatives holds.  Either
\begin{equation} \label{eqn:resid-err-bd}
\smnorm{2}{ \vct{r}^k }
	\leq 70 \enorm{ \vct{e} }
\end{equation}
or else
\begin{align}
\smnorm{2}{ \vct{r}^k }
	&\leq 2.3 \smnorm{2}{ \vct{x}\restrict{S_k^c} }
	\qquad\text{and}
	 \label{eqn:missing-link} \\
\smnorm{2}{ \vct{r}^{k+1} }
	&\leq 0.75 \smnorm{2}{ \vct{r}^{k} }
	\label{eqn:error-reduct}.
\end{align}
\end{lemma}


\begin{proof}
Define $T_k$ as the merged support that occurs during iteration $k$.
The pruning step ensures that the support $S_{k}$ of the approximation at the end of the iteration is a subset of the merged support, so
$$
\smnorm{2}{ \vct{x}\restrict{T_{k}^c} } \leq
	\smnorm{2}{ \vct{x}\restrict{S_k^c} }
\qquad\text{for $k = 1, 2, 3, \dots$}.	
$$
At the end of the $k$th iteration, the pruned vector $\vct{b}_s$ becomes the next approximation $\vct{a}^{k}$, so the estimation and pruning results, Lemmas~\ref{lem:estimation} and~\ref{lem:pruning}, together imply that
\begin{align} 
\smnorm{2}{ \vct{r}^{k} }
	&\leq 2 \cdot ( 1.112 \smnorm{2}{ \vct{x}\restrict{T_{k}^c} }
		+ 1.06 \enorm{\vct{e}} ) \notag \\
	&\leq 2.224 \smnorm{2}{ \vct{x}\restrict{S_{k}^c} }
		+ 2.12 \enorm{\vct{e}} \label{eqn:est-bd}
\qquad\text{for $k = 1, 2, 3, \dots$}.
\end{align}
Note that the same relation holds trivially for iteration $k = 0$ because $\vct{r}^0 = \vct{x}$ and $S_0 = \emptyset$.

Suppose that there is an iteration $k \geq 0$ where
\begin{equation*} 
\smnorm{2}{ \vct{x}\restrict{S_{k}^c} } < 30 \enorm{\vct{e}}.
\end{equation*}
We can introduce this bound directly into the inequality \eqref{eqn:est-bd} to obtain the first conclusion \eqref{eqn:resid-err-bd}.

Suppose on the contrary that in iteration $k$ we have
\begin{equation*} 
\smnorm{2}{ \vct{x}\restrict{S_{k}^c} } \geq 30 \enorm{\vct{e}}.
\end{equation*}
Introducing this relation into the inequality \eqref{eqn:est-bd} leads quickly to the conclusion \eqref{eqn:missing-link}.  We also have the chain of relations
$$
\smnorm{2}{ \vct{r}^k }
	\geq \smnorm{2}{ \vct{r}^k\restrict{S_k^c} }
	= \smnorm{2}{ (\vct{x} - \vct{a}^k)\restrict{S_k^c} }
	= \enorm{ \vct{x}\restrict{S_k^c} }
	\geq 30 \enorm{ \vct{e}}.
$$
Therefore, the sparse iteration invariant, Theorem~\ref{thm:invar-sparse} ensures that \eqref{eqn:error-reduct} holds.
\end{proof}

The next lemma contains the critical part of the argument.  Under the second alternative in the previous lemma, we show that the algorithm completely identifies the support of the signal, and we bound the number of iterations required to do so.

\begin{lemma} \label{lem:iter-count}
Fix $K = \lfloor p \log_{4/3} (1 + 4.6 \sqrt{s/p}) \rfloor$.  Assume that \eqref{eqn:missing-link} and \eqref{eqn:error-reduct} are in force for each iteration $k = 0, 1, 2, \dots, K$.  Then $\supp{ \vct{a}^{K} } = \supp{\vct{x}}$.
\end{lemma}


\begin{proof}
First, we check that, once the norm of the residual is smaller than each element of a band, the components in that band persist in the support of each subsequent approximation.  Define $J$ to be the set of nonempty bands, and fix a band $j \in J$.  Suppose that, for some iteration $k$, the norm of the residual satisfies
\begin{equation} \label{eqn:resid-band-j}
\smnorm{2}{ \vct{r}^k } \leq 2^{-(j+1)/2} \enorm{ \vct{x} }.
\end{equation}
Then it must be the case that
$B_j \subset \supp{ \vct{a}^k }$.
If not, then some component $i \in B_j$ appears in the residual: $r^k_i = x_i$.  This supposition implies that
$$
\smnorm{2}{ \vct{r}^k } \geq \abs{x_i} > 2^{-(j+1)/2} \enorm{\vct{x}},
$$
an evident contradiction.  Since \eqref{eqn:error-reduct} guarantees that the norm of the residual declines in each iteration, \eqref{eqn:resid-band-j} ensures that the support of each subsequent approximation contains $B_j$.


Next, we bound the number of iterations required to find the next nonempty band $B_j$, given that we have already identified the bands $B_i$ where $i < j$.  Formally, assume that the support $S_k$ of the current approximation contains $B_i$ for each $i < j$.  In particular, the set of missing components $S_k^c \subset \supp{ \vct{y}^j }$.  It follows from relation \eqref{eqn:missing-link} that
$$
\smnorm{2}{\vct{r}^k} \leq 2.3 \enorm{ \vct{y}^j }.
$$
We can conclude that we have identified the band $B_j$ in iteration $k + \ell$ if
$$
\smnorm{2}{\vct{r}^{k+\ell}} \leq 2^{-(j+1)/2} \enorm{ \vct{x} }.
$$
According to \eqref{eqn:error-reduct}, we reduce the error by a factor of $\beta^{-1} = 0.75$ during each iteration.  Therefore, the number $\ell$ of iterations required to identify $B_j$ is at most
$$
\log_\beta \left\lceil \frac{ 2.3 \enorm{\vct{y}^j} }{ 2^{-(j+1)/2} \enorm{ \vct{x} } } \right\rceil
$$
We discover that the total number of iterations required to identify all the (nonempty) bands is at most
$$
k_{\star} \defby \sum\nolimits_{j\in J} \log_\beta \left\lceil 2.3 \cdot \frac{2^{(j+1)/2} \enorm{ \vct{y}^j }}{\enorm{ \vct{x} }} \right \rceil.
$$
For each iteration $k \geq \lfloor k_{\star} \rfloor$, it follows that $\supp{\vct{a}^k} = \supp{\vct{x}}$.

It remains to bound $k_{\star}$ in terms of the profile $p$ of the signal.  For convenience, we focus on a slightly different quantity.  First, observe that $p = \abs{J}$.  Using the geometric mean--arithmetic mean inequality, we discover that
\begin{align*}
\exp\left\{ \frac{1}{p} \sum\nolimits_{j \in J} \log \left\lceil
	2.3 \cdot \frac{2^{(j+1)/2} \enorm{ \vct{y}^j }}{\enorm{ \vct{x} }} \right\rceil
	\right\}
&\leq \exp\left\{ \frac{1}{p} \sum\nolimits_{j \in J} \log \left( 1 +
	2.3 \cdot \frac{2^{(j+1)/2} \enorm{ \vct{y}^j }}{\enorm{ \vct{x} }} \right)
	\right\} \\
&\leq \frac{1}{p} \sum\nolimits_{j \in J} \left( 1 +
	2.3 \cdot \frac{2^{(j+1)/2} \enorm{ \vct{y}^j }}{\enorm{ \vct{x} }} \right) \\
&= 1 + \frac{2.3}{p} \sum\nolimits_{j \in J}
	\left( \frac{ 2^{j+1} \enormsq{ \vct{y}^j }}{\enormsq{ \vct{x} }} \right)^{1/2}.
\end{align*}
To bound the remaining sum, we recall the relation \eqref{eqn:signal-tail-bd}.  Then we invoke Jensen's inequality and simplify the result.
\begin{align*}
\frac{1}{p} \sum_{j \in J}
	\left( \frac{ 2^{j+1} \enormsq{ \vct{y}^j }}{\enormsq{ \vct{x} }} \right)^{1/2}
&\leq \frac{1}{p} \sum\nolimits_{j \in J}
	\left( 2^{j+1} \sum\nolimits_{i \geq j} 2^{-i} \abs{B_i} \right)^{1/2}
\leq \left( \frac{1}{p} \sum\nolimits_{j\in J}
	2^{j+1} \sum_{i \geq j} 2^{-i} \abs{B_i} \right)^{1/2} \\
&\leq \left( \frac{1}{p} \sum\nolimits_{i \geq 0} \abs{B_i}
	\sum\nolimits_{j \leq i} 2^{j - i + 1} \right)^{1/2}
\leq \left( \frac{4}{p} \sum\nolimits_{i \geq 0} \abs{B_i} \right)^{1/2} \\
&= 2 \sqrt{s/p}.
\end{align*}
The final equality holds because the total number of elements in all the bands equals the signal sparsity $s$.  Combining these bounds, we reach
$$
\exp\left\{ \frac{1}{p} \sum\nolimits_{j \in J} \log \left\lceil
	2.3 \cdot \frac{2^{(j+1)/2} \enorm{ \vct{y}^j }}
		{\enorm{\vct{x}}} \right\rceil \right\}
\leq 1 + 4.6 \sqrt{s/p}.
$$
Take logarithms, multiply by $p$, and divide through by $\log \beta$.  We conclude that the required number of iterations $k_{\star}$ is bounded as
$$
k_{\star} \leq p \log_\beta( 1 + 4.6 \sqrt{s/p} ).
$$
This is the advertised conclusion.
\end{proof}

Finally, we check that the algorithm produces a small approximation error within a reasonable number of iterations.

\begin{proof}[Proof of Theorem~\ref{thm:count-sparse}]
Abbreviate $K = \lfloor p \log( 1 + 4.6 \sqrt{s/p} ) \rfloor$.
Suppose that \eqref{eqn:resid-err-bd} never holds during the first $K$ iterations of the algorithm.  Under this circumstance, Lemma~\ref{lem:iter-alter} implies that both \eqref{eqn:missing-link} and \eqref{eqn:error-reduct} hold during each of these $K$ iterations.  It follows from Lemma \ref{lem:iter-count} that the support $S_K$ of the $K$th approximation equals the support of $\vct{x}$.  Since $S_K$ is contained in the merged support $T_K$, we see that the vector $\vct{x}\restrict{T_K^c} = \vct{0}$.  Therefore, the estimation and pruning results, Lemmas~\ref{lem:estimation} and~\ref{lem:pruning}, show that
$$
\smnorm{2}{ \vct{r}^K }
	\leq 2 \cdot \left( 1.112 \enorm{ \vct{x}\restrict{T_K^c} }
		+ 1.06 \enorm{ \vct{e}} \right)
	= 2.12 \enorm{\vct{e}}.
$$
This estimate contradicts \eqref{eqn:resid-err-bd}. 

It follows that there is an iteration $k \leq K$ where  \eqref{eqn:resid-err-bd} is in force.  Repeated applications of the iteration invariant, Theorem~\ref{thm:invar-sparse}, allow us to conclude that
$$
\smnorm{2}{ \vct{r}^{K + 6} } < 17 \enorm{\vct{e}}.
$$
This point completes the argument.
\end{proof}

Finally, we extend the sparse iteration count result to the general case.

\begin{theorem}[Iteration Count] \label{thm:count-detail}
Let $\vct{x}$ be an arbitrary signal, and define $p = {\rm profile}(\vct{x}_s)$.  After at most
$$
p \log_{4/3}(1 + 4.6 \sqrt{s/p}) + 6
$$
iterations, CoSaMP produces an approximation $\vct{a}$ that satisfies
$$
\enorm{ \vct{x} - \vct{a} } \leq 20 \enorm{ \vct{e} }.
$$
\end{theorem}

\begin{proof}
Let $\vct{x}$ be a general signal, and let $p = {\rm profile}(\vct{x}_s)$.  Lemma~\ref{lem:reduction} allows us to write the noisy vector of samples $\vct{u} = \Fee \vct{x}_s + \widetilde{\vct{e}}$.  The sparse iteration count result, Theorem~\ref{thm:count-sparse}, states that after at most
$$
p \log_{4/3}(1 + 4.6 \sqrt{s/p}) + 6
$$
iterations, the algorithm produces an approximation $\vct{a}$ that satisfies
$$
\enorm{ \vct{x}_s - \vct{a} } \leq 17 \enorm{ \widetilde{\vct{e}} }.
$$
Apply the lower triangle inequality to the left-hand side.  Then recall the estimate for the noise in Lemma~\ref{lem:reduction}, and simplify to reach
\begin{align*}
\enorm{ \vct{x} - \vct{a} }
	&\leq 17 \enorm{ \widetilde{\vct{e}} } + \enorm{ \vct{x} - \vct{x}_s } \\
	&\leq 18.9 \enorm{ \vct{x} - \vct{x}_s }
		+ \frac{17.9}{\sqrt{s}} \pnorm{1}{ \vct{x} - \vct{x}_s }
		+ 17 \enorm{ \vct{e} } \\
	&< 20 \nu,
\end{align*}
where $\nu$ is the unrecoverable energy.  
\end{proof}

\begin{proof}[Proof of Theorem~\ref{thm:cosamp-count}]
Invoke Theorem~\ref{thm:count-detail}.  Recall that the estimate for the number of iterations is maximized with $p = s$, which gives an upper bound of $6(s+1)$ iterations, independent of the signal.
\end{proof}

%% file: rwl1.tex
As discussed in Section~\ref{sec:Approaches:Basis}, the $\ell_1$-minimization problem~\eqref{eq:ell1} is equivalent to the nonconvex problem~\eqref{eq:ell0} when the measurement matrix $\Phi$ satisfies a certain condition.  Let us first state the best known theorem for recovery using $\ell_1$, provided by Cand\`es, in more detail.  We will rely on this result to prove the new theorems.

\begin{theorem}[$\ell_1$-minimization from~\cite{Can08:Restricted-Isometry}]\label{candes} Assume $\Phi$ has $\delta_{2s} < \sqrt{2}-1$. Let $x$ be an arbitrary signal with noisy measurements $\Phi x + e$, where $\|e\|_2 \leq \varepsilon$. Then the approximation $\hat{x}$ to $x$ from $\ell_1$-minimization satisfies
$$
\|x-\hat{x}\|_2 \leq C\varepsilon + C'\frac{\|x-x_s\|_1}{\sqrt{s}},
$$
where $C = \frac{2\alpha}{1-\rho}$, $C' = \frac{2(1+\rho)}{1-\rho}$, $\rho = \frac{\sqrt{2}\delta_{2s}}{1-\delta_{2s}}$, and $\alpha = \frac{2\sqrt{1+\delta_{2s}}}{\sqrt{1-\delta_{2s}}}$ .

\end{theorem}

The key difference between the two problems of course, is that the $\ell_1$ formulation depends on the magnitudes of the coefficients of a signal, whereas the $\ell_0$ does not.  To reconcile this imbalance, a new weighted $\ell_1$-minimization algorithm was proposed by Cand\`es, Wakin, and Boyd~\cite{CWB08:Reweighted}.  This algorithm solves the following weighted version of $(L_1)$ at each iteration:

\begin{equation*}\tag{$WL_1$}\label{WL1}
\min_{\hat{\vct{x}}\in\R^d} \sum_{i=1}^d \delta_i \hat{\vct{x}}_i \text{ subject to } \Phi \vct{x} = \Phi \hat{\vct{x}}.
\end{equation*}

It is clear that in this formulation, large weights $\delta_i$ will encourage small coordinates of the solution vector, and small weights will encourage larger coordinates.  Indeed, suppose the $s$-sparse signal $x$ was known exactly, and that the weights were set as
\begin{equation}\label{ws}
\delta_i = \frac{1}{|x_i|}.
\end{equation}
Notice that in this case, the weights are infinite at all locations outside of the support of $x$.  This will force the coordinates of the solution vector $\hat{\vct{x}}$ at these locations to be zero.  Thus if the signal $\vct{x}$ is $s$-sparse with $s\leq m$, these weights would guarantee that $\hat{\vct{x}} = \vct{x}$.  Of course, these weights could not be chosen without knowing the actual signal $x$ itself, but this demonstrates the positive effect that the weights can have on the performance of $\ell_1$-minimization.  

The helpful effect of the weights can also be viewed geometrically.  Recall that the solution to the problem $(L_1)$ is the contact point where the smallest $\ell_1$-ball meets the subspace $\vct{x} + \ker \Phi$.  When the solution to $(L_1)$ does not coincide with the original signal $\vct{x}$, it is because there is an $\ell_1$-ball smaller than the one containing $\vct{x}$, which meets the subspace $x + \ker \Phi$.  The solution to problem $(WL_1)$, however, is the place where the \textit{weighted} $\ell_1$-ball meets the subspace.  When the weights are appropriate, this is an $\ell_1$-ball that has been pinched toward the signal $x$ (see Figure~\ref{fig2rw}).  This new geometry reduces the likelihood of the incorrect solution.

\begin{center}
 	\begin{figure}[ht] 
 	\centering
  \includegraphics[width=0.4\textwidth,height=1.6in]{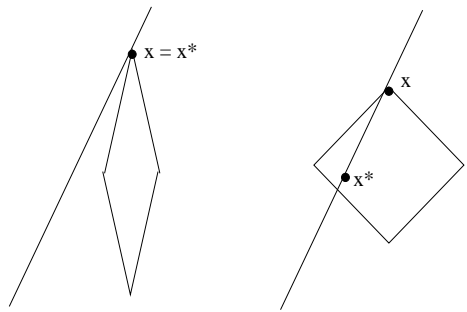}
  \caption{Weighted $\ell_1$-ball geometry (right) versus standard (left).}\label{fig2rw}
\end{figure} 
\end{center}

Although the weights might not initially induce this geometry, one hopes that by solving the problem $(WL_1)$ at each iteration, the weights will get closer to the optimal values~\eqref{ws}, thereby improving the reconstruction of $x$.  Of course, one cannot actually have an infinite weight as in~\eqref{ws}, so a stability parameter must also be used in the selection of the weight values.  The reweighted $\ell_1$-minimization algorithm can thus be described precisely as follows.

\bigskip

\textsc{Reweighted $\ell_1$-minimization}

\nopagebreak

\fbox{\parbox{\algorithmwidth}{
  \textsc{Input:} Measurement vector $u \in \R^m$, stability parameter $a$  
  
  \textsc{Output:} Reconstructed vector $\hat{x}$

  \begin{description}
    \item[Initialize] Set the weights $\delta_i = 1$ for $i=1\ldots d$.\\
      Repeat the following until convergence or a fixed number of times:
    \item[Approximate] Solve the reweighted $\ell_1$-minimization problem:
    $$ \hat{x} = \argmin_{\hat{x}\in\R^d} \sum_{i=1}^d \delta_i \hat{x}_i \text{ subject to } u = \Phi \hat{x} \text{ (or } \|\Phi \hat{x} - u\|_2 \leq \e \text{)}. $$
    \item[Update] Reset the weights:
    $$ \delta_i = \frac{1}{|\hat{x}_i| + a}. $$
  \end{description}
 }}

    \bigskip

\begin{remark}Note that the optional set of constraints given in the algorithm is only for the case in which the signal or measurements may be corrupted with noise.  It may also be advantageous to decrease the stability parameter $a$ so that $a \rightarrow 0$ as the iterations tend to infinity.  See the proof of Theorem~\ref{theone} below for details.
\end{remark} 

In~\cite{CWB08:Reweighted}, the reweighted $\ell_1$-minimization algorithm is discussed thoroughly, and experimental results are provided to show that it often outperforms the standard method.  However, no provable guarantees have yet been made for the algorithm's success.  Here we analyze the algorithm when the measurements and signals are corrupted with noise.  Since the reweighted method needs a weight vector that is somewhat close to the actual signal $x$, it is natural to consider the noisy case since the standard $\ell_1$-minimization method itself produces such a vector.  We are able to prove an error bound in this noisy case that improves upon the best known bound for the standard method.  We also provide numerical studies that show the bounds are improved in practice as well.



\subsection{Main Results}\label{sec:prfs}

The main theorem of this paper guarantees an error bound for the reconstruction using reweighted $\ell_1$-minimization that improves upon the best known bound of Theorem~\ref{candes} for the standard method.  For initial simplicity, we consider the case where the signal $x$ is exactly sparse, but the measurements $u$ are corrupted with noise.  Our main theorem, Theorem~\ref{theone} will imply results for the case where the signal $x$ is arbitrary.

\begin{theorem}[Reweighted $\ell_1$, Sparse Case]\label{theone}
Assume $\Phi$ satisfies the restricted isometry condition with parameters $(2s, \delta)$ where $\delta < \sqrt{2}-1$. Let $x$ be an $s$-sparse vector with noisy measurements $u = \Phi x + e$ where $\|e\|_2 \leq \varepsilon$. Assume the smallest nonzero coordinate $\mu$ of $x$ satisfies $\mu \geq \frac{4\alpha\varepsilon}{1-\rho}$.  Then the limiting approximation from reweighted $\ell_1$-minimization satisfies
$$
\|x-\hat{x}\|_2 \leq C''\e,
$$
where $C''= \frac{2\alpha}{1+\rho}$, $\rho = \frac{\sqrt{2}\delta}{1-\delta}$ and $\alpha = \frac{2\sqrt{1+\delta}}{1-\delta}$. 
\end{theorem}

\begin{remarks} 

{\bf 1.} We actually show that the reconstruction error satisfies
\begin{equation}\label{actualbnd}
\|x-\hat{x}\|_2 \leq \frac{2\alpha\varepsilon}{1 + \sqrt{1-\frac{4\alpha\varepsilon}{\mu}-\frac{4\alpha\varepsilon\rho}{\mu}}}.
\end{equation}
This bound is stronger than that given in Theorem~\ref{theone}, which is only equal to this bound when $\mu$ nears the value $\frac{4\alpha\varepsilon}{1-\rho}$.  However, the form in Theorem~\ref{theone} is much simpler and clearly shows the role of the parameter $\delta$ by the use of $\rho$.

{\bf 2.} For signals whose smallest non-zero coefficient $\mu$ does not satisfy the condition of the theorem, we may apply the theorem to those coefficients that do satisfy this requirement, and treat the others as noise.  See Theorem~\ref{ddnoise} below.

{\bf 3.} Although the bound in the theorem is the {\em limiting} bound, we provide a recursive relation~\eqref{Ek} in the proof which provides an exact error bound per iteration.  In Section~\ref{sec:num} we use dynamic programming to show that in many cases only a very small number of iterations are actually required to obtain the above error bound.
\end{remarks}


We now discuss the differences between Theorem~\ref{candes} and our new result Theorem~\ref{theone}. In the case where $\delta$ nears its limit of $\sqrt{2}-1$, the constant $\rho$ increases to $1$, and so the constant $C$ in Theorem~\ref{candes} is unbounded. However, the constant $C''$ in Theorem~\ref{theone} remains bounded even in this case.  In fact, as $\delta$ approaches $\sqrt{2}-1$, the constant $C''$ approaches just $4.66$.  The tradeoff of course, is in the requirement on $\mu$.  As $\delta$ gets closer to $\sqrt{2}-1$, the bound needed on $\mu$ requires the signal to have unbounded non-zero coordinates relative to the noise level $\e$.  However, to use this theorem efficiently, one would select the largest $\delta < \sqrt{2}-1$ that allows the requirement on $\mu$ to be satisfied, and then apply the theorem for this value of $\delta$.  Using this strategy, when the ratio $\frac{\mu}{\e} = 10$, for example, the error bound is just $3.85\e$.  

Theorem~\ref{theone} and a short calculation will imply the following result for \textit{arbitrary} signals $x$.

\begin{theorem}[Reweighted $\ell_1$]\label{ddnoise}
Assume $\Phi$ satisfies the restricted isometry condition with parameters $(2s, \sqrt{2}-1)$. Let $x$ be an arbitrary vector with noisy measurements $u = \Phi x + e$ where $\|e\|_2 \leq \varepsilon$. Assume the smallest nonzero coordinate $\mu$ of $x_s$ satisfies $\mu \geq \frac{4\alpha\varepsilon_0}{1-\rho},$ where $\varepsilon_0 = 
1.2(\|x-x_s\|_2 + \frac{1}{\sqrt{s}}\|x-x_s\|_1) + \varepsilon$.
Then the limiting approximation from reweighted $\ell_1$-minimization satisfies
$$
\|x-\hat{x}\|_2 \leq \frac{4.1\alpha}{1+\rho}\Big(\frac{\|x-x_{s/2}\|_1}{\sqrt{s}} + \varepsilon \Big),
$$
and
$$
\|x-\hat{x}\|_2 \leq \frac{2.4\alpha}{1+\rho}\Big(\|x-x_s\|_2 + \frac{\|x-x_s\|_1}{\sqrt{s}} + \varepsilon \Big),
$$
where $\rho$ and $\alpha$ are as in Theorem~\ref{theone}. 
\end{theorem}

  Again in the case where $\delta$ nears its bound of $\sqrt{2}-1$, both constants $C$ and $C'$ in Theorem~\ref{candes} approach infinity. However, in Theorem~\ref{ddnoise}, the constant remains bounded even in this case.  The same strategy discussed above for Theorem~\ref{theone} should also be used for Theorem~\ref{ddnoise}.  Next we begin proving Theorem~\ref{theone} and Theorem~\ref{ddnoise}.
  
  

\subsection{Proofs}

We will first utilize a lemma that bounds the $\ell_2$ norm of a small portion of the difference vector $x - \hat{x}$ by the $\ell_1$-norm of its remainder.  This lemma is proved in~\cite{Can08:Restricted-Isometry} and essentially in~\cite{CRT06:Stable} as part of the proofs of the main theorems of those papers.  We include a proof here as well since we require the final result as well as intermediate steps for the proof of our main theorem.

\begin{lemma}\label{tinyLem}
Set $h = \hat{x} - x$, and let $\alpha$, $\e$, and $\rho$ be as in Theorem~\ref{theone}. Let $T_0$ be the set of $s$ largest coefficients in magnitude of $x$ and $T_1$ be the $s$ largest coefficients of $h_{T_0^c}$. Then
$$
\|h_{T_0\cup T_1}\|_2 \leq \alpha\e + \frac{\rho}{\sqrt{s}}\|h_{T_0^c}\|_1.
$$
\end{lemma}
\begin{proof}
Continue defining sets $T_j$ by setting $T_1$ to be the $s$ largest coefficients of $h_{T_0^c}$, $T_2$ the next $s$ largest coefficients of $h_{T_0^c}$, and so on.
We begin by applying the triangle inequality and using the fact that $x$ itself is feasible in~\eqref{WL1}. This yields
\begin{equation}\label{first}
\|\Phi h\|_2 \leq \|\Phi\hat{x} - u\|_2 + \|\Phi x - u\|_2 \leq 2\e.
\end{equation}

By the decreasing property of these sets and since each set $T_j$ has cardinality at most $s$, we have for each $j\geq 2$,
$$
\|h_{T_j}\|_2 \leq \sqrt{s}\|h_{T_j}\|_\infty \leq \frac{1}{\sqrt{s}}\|h_{T_{j-1}}\|_1.
$$
Summing the terms, this gives
\begin{equation}\label{second}
\sum_{j\geq 2}\|h_{T_j}\|_2 \leq \frac{1}{\sqrt{s}}\|h_{T_0^c}\|_1.
\end{equation}
By the triangle inequality, we then also have
\begin{equation}\label{ref1}
\|h_{(T_0\cup T_1)^c}\|_2 \leq \frac{1}{\sqrt{s}}\|h_{T_0^c}\|_1.
\end{equation}
Now by linearity we have
$$
\|\Phi h_{T_0\cup T_1} \|_2^2 = \langle\Phi h_{T_0\cup T_1}, \Phi h\rangle - \langle\Phi h_{T_0\cup T_1}, \sum_{j\geq 2}\Phi h_{T_j}\rangle.
$$
By~\eqref{first} and the restricted isometry condition, we have
$$
|\langle\Phi h_{T_0\cup T_1}, \Phi h\rangle| \leq \|\Phi h_{T_0\cup T_1}\|_2\|\Phi h\|_2 \leq 2\e\sqrt{1+\delta}\|h_{T_0\cup T_1}\|_2.
$$
As shown in Lemma 2.1 of~\cite{Can08:Restricted-Isometry}, the restricted isometry condition and parallelogram law imply that for $j \geq 2$,
$$
|\langle \Phi h_{T_0}, \Phi h_{T_j} \rangle| \leq \delta\|\Phi h_{T_0}\|_2\|\Phi h_{T_j}\|_2 \text{ and } |\langle \Phi h_{T_1}, \Phi h_{T_j} \rangle| \leq \delta\|\Phi h_{T_1}\|_2\|\Phi h_{T_j}\|_2.
$$
Since all sets $T_j$ are disjoint, the above three inequalities yield
$$
(1-\delta)\|h_{T_0\cup T_1}\|_2^2 \leq \|\Phi h_{T_0\cup T_1}\|_2^2 \leq \|h_{T_0\cup T_1}\|_2(2\e\sqrt{1+\delta} + \sqrt{2}\delta\sum_{j\geq 2}\|h_{T_j}\|_2).
$$
Therefore by~\eqref{second}, we have
$$
\|h_{T_0\cup T_1}\|_2 \leq \alpha\e + \frac{\rho}{\sqrt{s}}\|h_{T_0^c}\|_1.
$$

\end{proof}

We will next require two lemmas that give results about a single iteration of reweighted $\ell_1$-minimization.

\begin{lemma}[Single reweighted $\ell_1$-minimization]\label{doublenoise}
Assume $\Phi$ satisfies the restricted isometry condition with parameters $(2s, \sqrt{2}-1)$. Let $x$ be an arbitrary vector with noisy measurements $u = \Phi x + e$ where $\|e\|_2 \leq \varepsilon$. Let $w$ be a vector such that $\|w-x\|_\infty \leq A$ for some constant $A$. Denote by $x_s$ the vector consisting of the $s$ (where $s \leq |\supp (x)|$) largest coefficients of $x$ in absolute value. Let $\mu$ be the smallest coordinate of $x_s$ in absolute value, and set $b = \|x-x_s\|_\infty$. Then when $\mu \geq A$ and $\rho C_1 < 1$, the approximation from reweighted $\ell_1$-minimization using weights $\delta_i = 1 / (w_i + a)$ satisfies
$$
\|x-\hat{x}\|_2 \leq
D_1\varepsilon + D_2\frac{\|x-x_s\|_1}{a},
$$
where $D_1=\frac{(1+C_1)\alpha}{1-\rho C_1}$, $D_2 = C_2+\frac{(1+C_1)\rho C_2}{1-\rho C_1}$, $C_1 = \frac{A+a+b}{\mu-A+a}$, $C_2 = \frac{2(A+a+b)}{\sqrt{s}}$, and $\rho$ and $\alpha$ are as in Theorem~\ref{theone}.
\end{lemma}

\begin{proof}[Proof of Lemma~\ref{doublenoise}]

Set $h$ and $T_j$ for $j\geq 0$ as in Lemma~\ref{tinyLem}.  For simplicity, denote by $\|\cdot\|_w$ the weighted $\ell_1$-norm: 
$$
\|z\|_w \defby \sum_{i=1}^d \frac{1}{|w_i| + a}z_i.
$$
Since $\hat{x} = x+h$ is the minimizer of~\eqref{WL1}, we have 
$$
\|x\|_w \geq \|x+h\|_w = \|(x+h)_{T_0}\|_w + \|(x+h)_{T_0^c}\|_w \geq \|x_{T_0}\|_w - \|h_{T_0}\|_w + \|h_{T_0^c}\|_w - \|x_{T_0^c}\|_w.
$$
This yields 
$$
\|h_{T_0^c}\|_w \leq \|h_{T_0}\|_w + 2\|x_{T_0^c}\|_w.
$$
Next we relate the reweighted norm to the usual $\ell_1$-norm. We first have
$$
\|h_{T_0^c}\|_w \geq \frac{\|h_{T_0^c}\|_1}{A+a+b},
$$
by definition of the reweighted norm as well as the values of $A$, $a$, and $b$.
Similarly we have
$$
\|h_{T_0}\|_w  \leq \frac{\|h_{T_0}\|_1}{\mu - A + a}.
$$
Combining the above three inequalities, we have
\begin{align}\label{ref2}
\|h_{T_0^c}\|_1 \leq (A+a+b)\|h_{T_0^c}\|_w &\leq (A+a+b)(\|h_{T_0}\|_w + 2\|x_{T_0^c}\|_w)\notag\\
&\leq \frac{A+a+b}{\mu-A+a}\|h_{T_0}\|_1 + 2(A+a+b)\|x_{T_0^c}\|_w.
\end{align}
Using~\eqref{ref1} and~\eqref{ref2} along with the fact $\|h_{T_0}\|_1 \leq \sqrt{s}\|h_{T_0}\|_2$, we have
\begin{equation}\label{third}
\|h_{(T_0\cup T_1)^c}\|_2 \leq C_1\|h_{T_0}\|_2 + C_2\|x_{T_0^c}\|_w, 
\end{equation}
where $C_1 = \frac{A+a+b}{\mu-A+a}$ and $C_2 = \frac{2(A+a+b)}{\sqrt{s}}$.
By Lemma~\ref{tinyLem}, we have
$$
\|h_{T_0\cup T_1}\|_2 \leq \alpha\varepsilon + \frac{\rho}{\sqrt{s}}\|h_{T_0^c}\|_1,
$$
where $\rho = \frac{\sqrt{2}\delta_{2s}}{1-\delta_{2s}}$ and $\alpha = \frac{2\sqrt{1+\delta_{2s}}}{\sqrt{1-\delta_{2s}}}$.
Thus by~\eqref{ref2}, we have 
$$
\|h_{T_0\cup T_1}\|_2 \leq \alpha\varepsilon + \frac{\rho}{\sqrt{s}}(C_1\|h_{T_0}\|_1 + 2(A+a+b)\|x_{T_0^c}\|_w) = \alpha\varepsilon + \rho C_1\|h_{T_0\cup T_1}\|_2 + \rho C_2\|x_{T_0^c}\|_w.
$$
Therefore, 
\begin{equation}\label{need}
\|h_{T_0\cup T_1}\|_2 \leq (1 - \rho C_1)^{-1}(\alpha\varepsilon + \rho C_2\|x_{T_0^c}\|_w).
\end{equation}
Finally by~\eqref{third} and~\eqref{need},
\begin{align*}
\|h\|_2 &\leq \|h_{T_0\cup T_1}\|_2 + \|h_{(T_0\cup T_1)^c}\|_2 \\
&\leq(1+C_1)\|h_{T_0\cup T_1}\|_2 + C_2\|x_{T_0^c}\|_w \\
&\leq (1+C_1)((1 - \rho C_1)^{-1}(\alpha\varepsilon + \rho C_2\|x_{T_0^c}\|_w)) + C_2\|x_{T_0^c}\|_w.
\end{align*}
Applying the inequality $\|x_{T_0^c}\|_w \leq (1/a)\|x_{T_0^c}\|_1$ and simplifying completes the claim.

\end{proof}

\begin{lemma}[Single reweighted $\ell_1$-minimization, Sparse Case]\label{noisymeas}
Assume $\Phi$ satisfies the restricted isometry condition with parameters $(2s, \sqrt{2}-1)$. Let $x$ be an $s$-sparse vector with noisy measurements $u = \Phi x + e$ where $\|e\|_2 \leq \varepsilon$. Let $w$ be a vector such that $\|w-x\|_\infty \leq A$ for some constant $A$. Let $\mu$ be the smallest non-zero coordinate of $x$ in absolute value. Then when $\mu \geq A$, the approximation from reweighted $\ell_1$-minimization using weights $\delta_i = 1 / (w_i + a)$ satisfies
$$
\|x-\hat{x}\|_2 \leq D_1\varepsilon. 
$$
Here $D_1 = \frac{(1+C_1)\alpha}{1-\rho C_1}$, $C_1 = \frac{A+a}{\mu-A+a}$, and $\alpha$ and $\rho$ are as in Theorem~\ref{theone}.

\end{lemma}
\begin{proof}
This is the case of Lemma~\ref{doublenoise} where $x-x_s = 0$ and $b=0$. 
\end{proof}

\begin{proof}[Proof of Theorem~\ref{theone}]
The proof proceeds as follows. First, we use the error bound in Theorem~\ref{candes} as the initial error, and then apply Lemma~\ref{noisymeas} repeatedly. We show that the error decreases at each iteration, and then deduce its limiting bound using the recursive relation. To this end, let $E(k)$ for $k=1,\ldots$, be the error bound on $\|x-\hat{x}_k\|_2$ where $\hat{x}_k$ is the reconstructed signal at the $k^{th}$ iteration. Then by Theorem~\ref{candes} and Lemma~\ref{noisymeas}, we have the recursive definition
\begin{equation}\label{Ek}
E(1) = \frac{2\alpha}{1-\rho}\varepsilon, \quad E(k+1) = \frac{(1+\frac{E(k)}{\mu-E(k)})\alpha}{1-\rho \frac{E(k)}{\mu-E(k)}}\varepsilon. 
\end{equation}
Here we have taken $a\rightarrow 0$ iteratively (or if $a$ remains fixed, a small constant $O(a)$ will be added to the error). First, we show that the base case holds, $E(1) \leq E(2)$. Since $\mu \geq \frac{4\alpha\varepsilon}{1-\rho}$, we have that
$$
\frac{E(1)}{\mu-E(1)} = \frac{\frac{2\alpha\varepsilon}{1-\rho}}{\mu-\frac{2\alpha\varepsilon}{1-\rho}} \leq 1.
$$
Therefore we have
$$
E(2) = \frac{(1+\frac{E(1)}{\mu-E(1)})\alpha}{1-\rho \frac{E(1)}{\mu-E(1)}}\varepsilon \leq \frac{2\alpha}{1-\rho}\varepsilon = E(1).
$$
Next we show the inductive step, that $E(k+1)\leq E(k)$ assuming the inequality holds for all previous $k$. Indeed, if $E(k) \leq E(k-1)$, then we have
$$
E(k+1) = \frac{(1+\frac{E(k)}{\mu-E(k)})\alpha}{1-\rho \frac{E(k)}{\mu-E(k)}}\varepsilon \leq \frac{(1+\frac{E(k-1)}{\mu-E(k-1)})\alpha}{1-\rho \frac{E(k-1)}{\mu-E(k-1)}}\varepsilon = E(k).
$$  
Since $\mu \geq \frac{4\alpha\varepsilon}{1-\rho}$ and $\rho \leq 1$ we have that $\mu - E(k)\geq 0$ and $\rho\frac{E(k)}{\mu-E(k)}\leq 1$, so $E(k)$ is also bounded below by zero. Thus E(k) is a bounded decreasing sequence, so it must converge. Call its limit $L$. By the recursive definition of $E(k)$, we must have
$$
L = \frac{(1+\frac{L}{\mu-L})\alpha}{1-\rho \frac{L}{\mu-L}}\varepsilon.
$$
Solving this equation yields
$$
L=\frac{\mu-\sqrt{\mu^2-4\mu\alpha\varepsilon-4\mu\alpha\varepsilon\rho}}{2(1+\rho)},
$$
where we choose the solution with the minus since $E(k)$ is decreasing and $E(1) < \mu/2$ (note also that $L=0$ when $\e=0$). Multiplying by the conjugate and simplifying yields
$$
L = \frac{4\mu\alpha\varepsilon+4\mu\alpha\varepsilon\rho}{2(1+\rho)(\mu+\sqrt{\mu^2-4\mu\alpha\varepsilon-4\mu\alpha\varepsilon\rho})} = \frac{2\alpha\varepsilon}{1 + \sqrt{1-\frac{4\alpha\varepsilon}{\mu}-\frac{4\alpha\varepsilon\rho}{\mu}}}.
$$
Then again since $\mu \geq \frac{4\alpha\varepsilon}{1-\rho}$, we have
$$
L \leq \frac{2\alpha\varepsilon}{1+\rho}.
$$ 
\end{proof}

\begin{proof}[Proof of Theorem~\ref{ddnoise}]
By Lemma 6.1 of~\cite{NT08:Cosamp} and Lemma 7 of~\cite{GSTV07:HHS}, we can rewrite $\Phi x + e$ as $\Phi x_s + \widetilde{e}$ where 
$$
\|\widetilde{e}\|_2 \leq 1.2(\|x-x_s\|_2 + \frac{1}{\sqrt{s}}\|x-x_s\|_1) + \|e\|_2 \leq 2.04\Big(\frac{\|x-x_{s/2}\|_1}{\sqrt{s}}\Big) + \|e\|_2.
$$
This combined with Theorem~\ref{theone} completes the claim.
\end{proof}

\subsection{Numerical Results and Convergence}\label{sec:num}

Our main theorems prove bounds on the reconstruction error limit.  However, as is the case with many recursive relations, convergence to this threshold is often quite fast.  To show this, we use dynamic programming to compute the theoretical error bound $E(k)$ given by~\eqref{Ek} and test its convergence rate to the threshold given by~eqref{actualbnd}. Since the ratio between $\mu$ and $\e$ is important, we fix $\mu = 10$ and test the convergence for various values of $\e$ and $\delta$.  We conclude that the error bound has achieved the threshold when it is within 0.001 of it.  The results are displayed in Figure~\ref{fig:its}.

\begin{figure}[ht] 
  \includegraphics[scale=0.6]{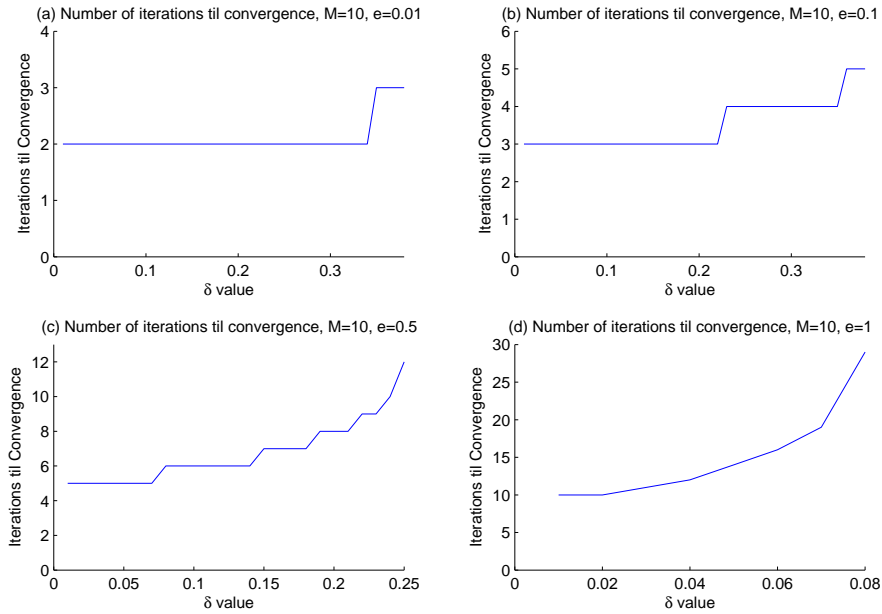}
  \caption{The number of iterations required for the theoretical error bounds~eqref{Ek} to reach the theoretical error threshold~\eqref{actualbnd} when (a) $\mu = 10$, $\e=0.01$, (b) $\mu = 10$, $\e=0.1$, (c) $\mu = 10$, $\e=0.5$, (d) $\mu = 10$, $\e=1.0$. }\label{fig:its}
\end{figure}

We observe that in each case, as $\delta$ increases we require slightly more iterations.  This is not surprising since higher $\delta$ means a lower bound.  We also confirm that less iterations are required when the ratio $\mu/\e$ is smaller. 

Next we examine some numerical experiments conducted to test the actual error with reweighted $\ell_1$-minimization versus the standard $\ell_1$ method.  In these experiments we consider signals of dimension $d=256$ with $s=30$ non-zero entries.  We use a $128 \times 256$ measurement matrix $\Phi$ consisting of Gaussian entries.  We note that we found similar results when the measurement matrix $\Phi$ consisted of symmetric Bernoulli entries.  Sparsity, measurement, and dimension values in this range demonstrate a good example of the advantages of the reweighted method.  Our results above suggest that improvements are made using the reweighted method when $\delta$ cannot be forced very small.  This means that in situations with sparsity levels $s$ much smaller or measurement levels $m$ much larger, the reweighted method may not offer much improvements.  These levels are not the only ones that show improvements, however, see for example those experiments in~\cite{CWB08:Reweighted}.  

For each trial in our experiments we construct an $s$-sparse signal $x$ with uniformly chosen support and entries from either the Gaussian distribution or the symmetric Bernoulli distribution.  All entries are chosen independently and independent of the measurement matrix $\Phi$.  We then construct the noise vector $e$ to have independent Gaussian entries, and normalize the vector to have norm a fraction (1/5) of the noiseless measurement vector $\Phi x$ norm.  We then run the reweighted $\ell_1$-algorithm using $\e$ such that $\e^2 = \sigma^2(m + 2\sqrt{2m})$ where $\sigma^2$ is the variance of the normalized error vectors. This value is likely to provide a good upper bound on the noise norm (see e.g. \cite{CRT06:Stable}, \cite{CWB08:Reweighted}). The stability parameter $a$ tends to zero with increased iterations. We run 500 trials for each parameter selection and signal type.  We found similar results for non-sparse signals such as noisy sparse signals and compressible signals.  This is not surprising since we can treat the signal error as measurement error after applying the the measurement matrix (see the proof of Theorem~\ref{ddnoise}). 

Figure~\ref{fig:dece} depicts the experiment with Gaussian signals and decreasing stability parameter $a$.  In particular, we set $a = 1/(1000k)$ in the $k^{th}$ iteration.  The plot (left) depicts the error after a single $\ell_1$-minimization and after $9$ iterations using the reweighted method.  The histogram (right) depicts the improvements $\|x - \hat{x}\|_2 / \|x - x^*\|_2$ where $\hat{x}$ and $x^*$ are the reconstructed vectors after $9$ iterations of reweighted and a single $\ell_1$-minimization, respectively.  
  
  \begin{figure}[h!]
\begin{center}
$\begin{array}{c@{\hspace{.1in}}c}
\includegraphics[width=3in]{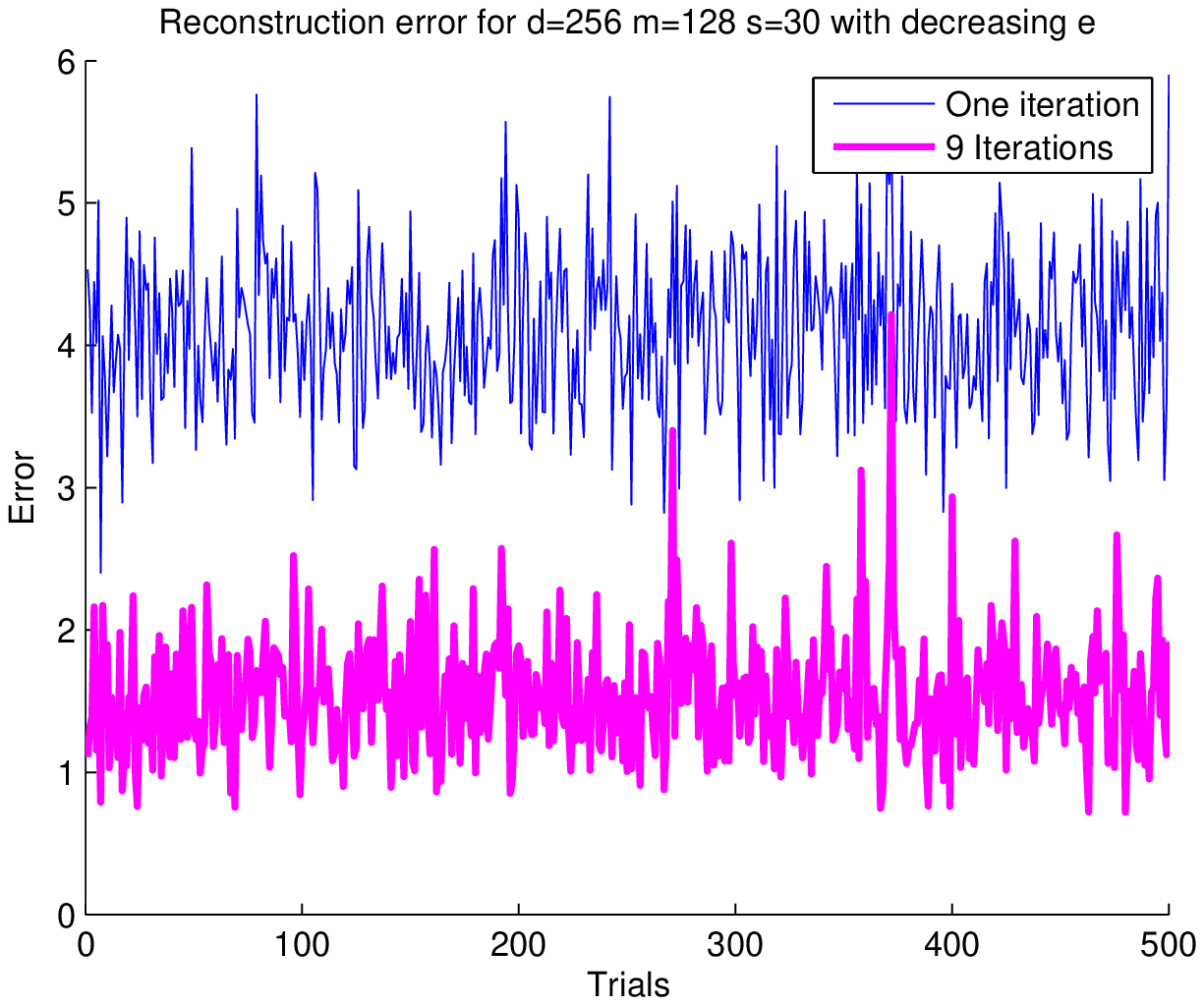}  &  
\includegraphics[width=3in]{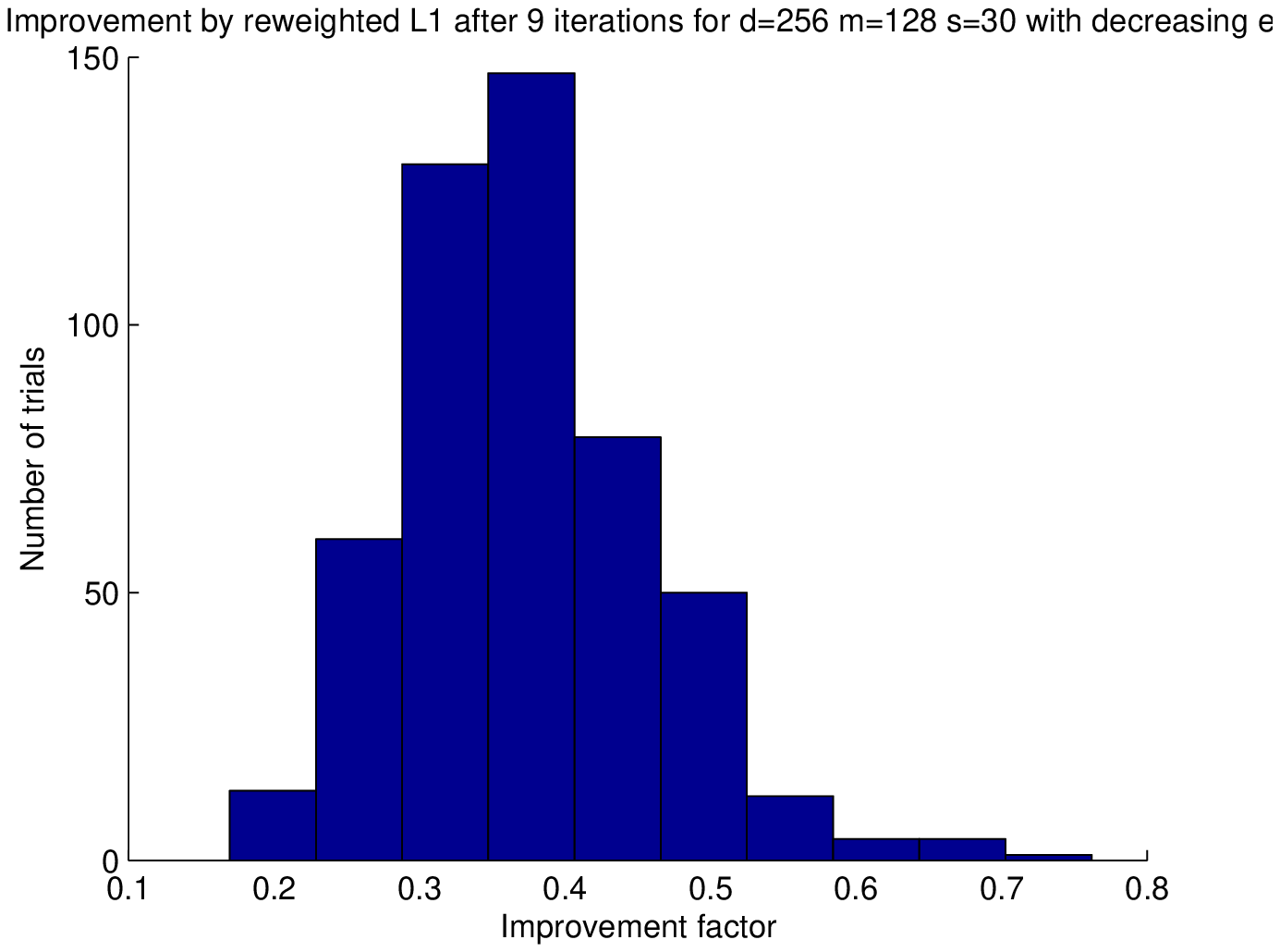}   \\
\end{array}$
\end{center}
\caption{Improvements in the $\ell_2$ reconstruction error using reweighted $\ell_1$-minimization versus standard $\ell_1$-minimization for Gaussian signals.}
\label{fig:dece}
\end{figure}

Figure~\ref{fig:deceBern} depicts the same results as Figure~\ref{fig:dece} above, but for the experiments with Bernoulli signals.  We see that in this case again reweighted $\ell_1$ offers improvements.  In this setting, the improvements in the Bernoulli signals seem to be slightly less than in the Gaussian case.  It is clear that in the case of flat signals like Bernoulli, the requirement on $\mu$ in Theorem~\ref{theone} may be easier to satisfy, since the signal will have no small components unless they are all small.  However, in the case of flat signals, if this requirement is not met, then the Theorem guarantees no improvements whatsoever.  In the Gaussian case, even if the requirement on $\mu$ is not met, the Theorem still shows that improvements will be made on the larger coefficients.  

  
  \begin{figure}[h!]
\begin{center}
$\begin{array}{c@{\hspace{.1in}}c}
\includegraphics[width=3in]{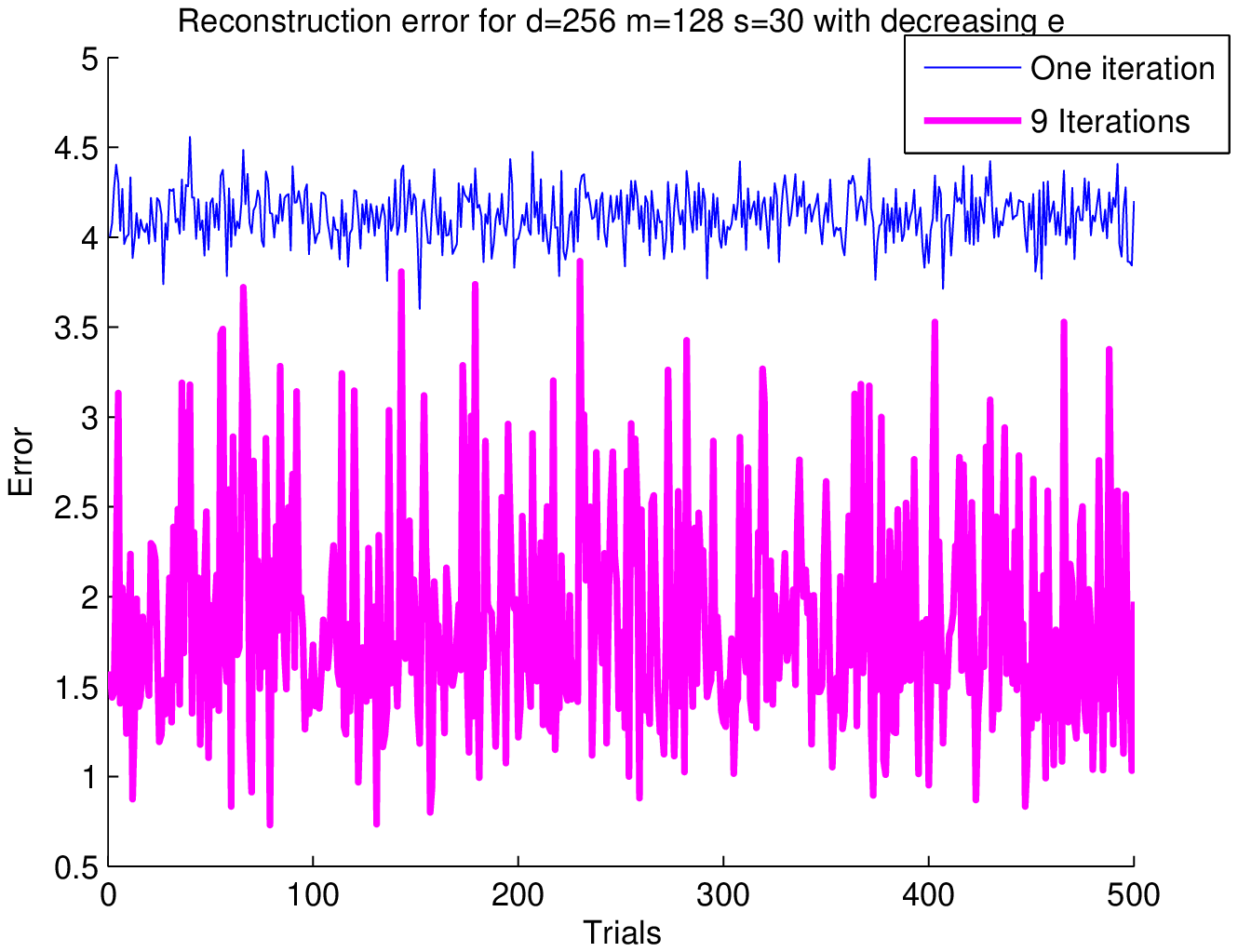}  &  
\includegraphics[width=3in]{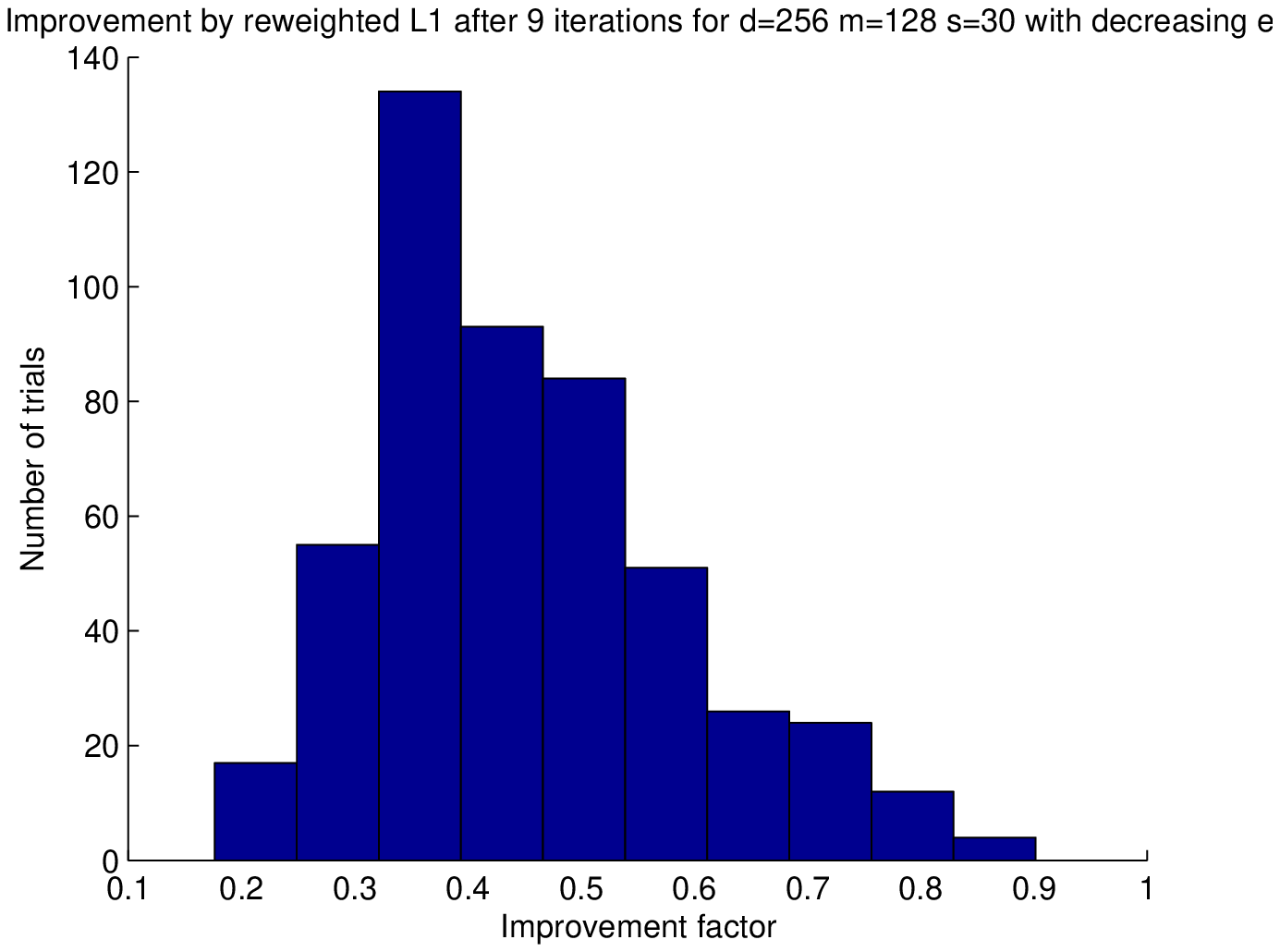}   \\
\end{array}$
\end{center}
\caption{Improvements in the $\ell_2$ reconstruction error using reweighted $\ell_1$-minimization versus standard $\ell_1$-minimization for Bernoulli signals.}
\label{fig:deceBern}
\end{figure}

%% file: rk.tex
Although the algorithms in compressed sensing themselves are not random, the measurement matrices used in the compression step are.  The suggestion that randomness often makes things easier is the key to the idea of randomized algorithms.  It is often the case that a deterministic algorithm's flaw can be fixed by introducing randomness.  This notion is at the heart of a new randomized version of the well known Kaczmarz algorithm.  Although the work on Kaczmarz is outside the realm on compressed sensing, it does bare some striking similarities.  

The Kaczmarz method~\cite{K37:Angena} is one of the most popular solvers of overdetermined linear systems and has numerous applications from computer tomography to image processing.  The algorithm consists of a series of alternating projections, and is often considered a type of \textit{Projection on Convex Sets} (POCS) method.  Given a consistent system of linear equations of the form 
$$
Ax = b,
$$
the Kaczmarz method iteratively projects onto the solution spaces of each equation in the system.  That is, if $a_1, \ldots, a_m \in \R^n$ denote the rows of $A$, the method cyclically projects the current estimate orthogonally onto the hyperplanes consisting of solutions to $\pr{a_i}{x} = b_i$.  Each iteration consists of a single orthogonal projection.  The algorithm can thus be described using the recursive relation,
$$
x_{k+1} = x_k + \frac{b_i - \pr{a_i}{x_k}}{\|a_i\|_2^2}a_i,
$$
where $x_k$ is the $k^{th}$ iterate and $i = (k$ mod $m) + 1$.  

Theoretical results on the convergence rate of the Kaczmarz method have been difficult to obtain.  Most known estimates depend on properties of the matrix $A$ which may be time consuming to compute, and are not easily comparable to those of other iterative methods (see e.g. ~\cite{DH97:Therate},~\cite{G05:Onthe},~\cite{HN90:Onthe}).  Since the Kaczmarz method cycles through the rows of $A$ sequentially, its convergence rate depends on the order of the rows.  Intuition tells us that the order of the rows of $A$ does not change the difficulty level of the system as a whole, so one would hope for results independent of the ordering.  One natural way to overcome this is to use the rows of $A$ in a random order, rather than sequentially.  Several observations were made on the improvements of this randomized version~\cite{N86:TheMath,HM93:Algebraic}, but only recently have theoretical results been obtained~\cite{SV06:Arandom}. 

In designing a random version of the Kaczmarz method, it is necessary to set the probability of each row being selected. Strohmer and Vershynin propose in~\cite{SV06:Arandom} to set the probability proportional to the Euclidean norm of the row. Their revised algorithm can then be described by the following:
$$
x_{k+1} = x_k + \frac{b_{p(i)} - \pr{a_{p(i)}}{x_k}}{\|a_{p(i)}\|_2^2}a_{p(i)},
$$
where $p(i)$ takes values in $\{1, \ldots, m\}$ with probabilities $\frac{\|a_{p(i)}\|_2^2}{\|A\|_F^2}$.  Here and throughout, $\|A\|_F$ denotes the Frobenius norm of $A$ and $\|\cdot\|_2$ denotes the usual Euclidean norm or spectral norm for vectors or matrices, respectively.  We note here that of course, one needs some knowledge of the norm of the rows of $A$ in this version of the algorithm. In general, this computation takes $\bigO(mn)$ time.  However, in many cases such as the case in which $A$ contains Gaussian entries, this may be approximately or exactly known.  

In~\cite{SV06:Arandom}, Strohmer and Vershynin prove the following exponential bound on the expected rate of convergence for the randomized Kaczmarz method,
\begin{equation}\label{SVbound}
\mathbb{E}\|x_k - x\|_2^2 \leq \Big(1 - \frac{1}{R}\Big)^k\|x_0 - x\|_2^2,
\end{equation}
where $R = \|A^{-1}\|^2\|A\|_F^2$ and $x_0$ is an arbitrary initial estimate. Here and throughout, $\|A^{-1}\| \defby \inf\{M : M\|Ax\|_2 \geq \|x\|_2$ for all $x\}$.  

The first remarkable note about this result is that it is essentially independent of the number $m$ of equations in the system.  Indeed, by the definition of $R$, $R$ is proportional to $n$ within a square factor of the condition number of $A$.  Also, since the algorithm needs only access to the randomly chosen rows of $A$, the method need not know the entire matrix $A$.  Indeed, the bound~\eqref{SVbound} and the relationship of $R$ to $n$ shows that the estimate $x_k$ converges exponentially fast to the solution in just $\bigO(n)$ iterations.  Since each iteration requires $\bigO(n)$ time, the method overall has a $\bigO(n^2)$ runtime.  Thus this randomized version of the algorithm provides advantages over all previous methods for extremely overdetermined linear systems.  There are of course situations where other methods, such as the conjugate gradient method, outperform the randomized Kaczmarz method.  However, numerical studies in~\cite{SV06:Arandom} show that in many scenarios (for example when $A$ is Gaussian), the randomized Kaczmarz method provides faster convergence than even the conjugate gradient method.

The remarkable benefits provided by the randomized Kaczmarz algorithm lead one to question whether the method works in the more realistic case where the system is corrupted by noise.  In this paper we provide theoretical and empirical results to suggest that in this noisy case the method converges exponentially to the solution within a specified error bound.  The error bound is proportional to $\sqrt{R}$, and we also provide a simple example showing this bound is sharp in the general setting.  

\subsection{Main Results}

In this section we discuss the results by Needell in~\cite{N9:RK}.

Theoretical and empirical studies have shown the randomized Kaczmarz algorithm to provide very promising results.  Here we show that it also performs well in the case where the system is corrupted with noise.  In this section we consider the system $Ax=b$ after an error vector $r$ is added to the right side:
$$
Ax \approx b + r.
$$
First we present a simple example to gain intuition about how drastically the noise can affect the system.  To that end, consider the $n \times n$ identity matrix $A$. Set $b=0$, $x=0$, and suppose the error is the vector whose entries are all one, $r = (1, 1, \ldots, 1)$.  Then the solution to the noisy system is clearly $r$ itself, and so by~\eqref{SVbound}, the iterates $x_k$ of randomized Kaczmarz will converge exponentially to $r$.  Since $A$ is the identity matrix, we have $R=n$.  Then by~\eqref{SVbound} and Jensen's inequality, we have  
$$
\mathbb{E}\|x_k - r\|_2 \leq \Big(1 - \frac{1}{R}\Big)^{k/2}\|x_0-r\|_2.
$$
Then by the triangle inequality, we have
$$
\mathbb{E}\|x_k - x\|_2 \geq \|r - x\|_2 - \Big(1 - \frac{1}{R}\Big)^{k/2}\|x_0-r\|_2.
$$
Finally by the definition of $r$ and $R$, this implies
$$
\mathbb{E}\|x_k - x\|_2 \geq \sqrt{R} - \Big(1 - \frac{1}{R}\Big)^{k/2}\|x_0-r\|_2.
$$
This means that the limiting error between the iterates $x_k$ and the original solution $x$ is $\sqrt{R}$. In~\cite{SV06:Arandom} it is shown that the bound provided in~\eqref{SVbound} is optimal, so if we wish to maintain a general setting, the best error bound for the noisy case we can hope for is proportional to $\sqrt{R}$.  Our main result proves this exact theoretical bound.

\begin{theorem}[Noisy randomized Kaczmarz]\label{thm}
Let $x_k^*$ be the $k^{th}$ iterate of the noisy Randomized Kaczmarz method run with $Ax \approx b + r$, and let $a_1, \ldots a_m$ denote the rows of $A$. Then we have
$$
\mathbb{E}\|x_k^* - x\|_2 \leq \Big(1 - \frac{1}{R}\Big)^{k/2}\|x_{0}\|_2 + \sqrt{R}\gamma,
$$
where $R = \|A^{-1}\|^2\|A\|_F^2$ and $\gamma = \max_i \frac{|r_i|}{\|a_i\|_2}$.
\end{theorem}

\begin{remark}
In the case discussed above, note that we have $\gamma = 1$, so the example indeed shows the bound is sharp.
\end{remark}

Before proving the theorem, it is important to first analyze what happens to the solution spaces of the original equations $Ax=b$ when the error vector is added.  Letting $a_1, \ldots a_m$ denote the rows of $A$, we have that each solution space $\pr{a_i}{x} = b_i$ of the original system is a hyperplane whose normal is $\frac{a_i}{\|a_i\|_2}$.  When noise is added, each hyperplane is translated in the direction of $a_i$.  Thus the new geometry consists of hyperplanes parallel to those in the noiseless case.  A simple computation provides the following lemma which specifies exactly how far each hyperplane is shifted.

\begin{lemma}\label{easylem}
Let $H_i$ be the affine subspace of $\R^n$ consisting of the solutions to $\left\langle a_i, x \right\rangle = b_i$. Let $H_i^*$ be the solution space of the noisy system $\left\langle a_i, x \right\rangle = b_i + r_i$. Then $H_i^* = \{w + \alpha_i a_i : w\in H_i\}$ where $\alpha_i = \frac{r_i}{\|a_i\|_2^2}$.
\end{lemma}
\begin{proof}
First, if $w\in H_i$ then $\pr{a_i}{w + \alpha a_i} = \pr{a_i}{w} + \alpha\|a_i\|_2^2 = b_i + r_i$, so $w + \alpha a_i \in H_i^*$. Next let $u \in H_i^*$. Set $ w = u - \alpha a_i$. Then $\pr{a_i}{w} = \pr{a_i}{u} - r_i = b_i + r_i - r_i = b_i$, so $w \in H_i^*$. This completes the proof.
\end{proof}

We will also utilize the following lemma which is proved in the proof of Theorem 2 in~\cite{SV06:Arandom}.

\begin{lemma}\label{SVlem}
Let $x_{k-1}^*$ be any vector in $\mathbb{R}^n$ and let $x_k$ be its orthogonal projection onto a random solution space as in the noiseless randomized Kaczmarz method run with $Ax=b$. Then we have
$$
\mathbb{E}\|x_k - x\|_2^2 \leq \Big(1 - \frac{1}{R}\Big)\|x_{k-1}^* - x\|_2^2,
$$
where $R = \|A^{-1}\|^2\|A\|_F^2$.
\end{lemma}

We are now prepared to prove Theorem~\ref{thm}.

\begin{proof}[Proof of Theorem~\ref{thm}]
Let $x_{k-1}^*$ denote the $(k-1)^{th}$ iterate of noisy Randomized Kaczmarz.  Using notation as in Lemma~\ref{easylem}, let $H_i^*$ be the solution space chosen in the $k^{th}$ iteration.  Then $x_k^*$ is the orthogonal projection of $x_{k-1}^*$ onto $H_i^*$. Let $x_k$ denote the orthogonal projection of $x_{k-1}^*$ onto $H_i$ (see Figure~\ref{fig0}). 

\begin{figure}[ht] 
  \includegraphics[scale=0.5]{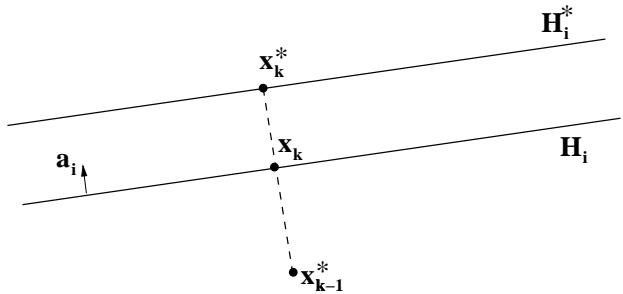}
  \caption{The parallel hyperplanes $H_i$ and $H_i^*$ along with the two projected vectors $x_k$ and $x_k^*$.}\label{fig0}
\end{figure}

By Lemma~\ref{easylem}  and the fact that $a_i$ is orthogonal to $H_i$ and $H_i^*$, we have that $x_k^* - x = x_k - x + \alpha_i a_i$.  Again by orthogonality, we have $\|x_k^* - x\|_2^2 = \|x_k - x\|_2^2 + \|\alpha_i a_i\|_2^2$.  Then by Lemma~\ref{SVlem} and the definition of $\gamma$, we have 
$$
\mathbb{E}\|x_k^* - x\|_2^2 \leq \Big(1 - \frac{1}{R}\Big)\|x_{k-1}^* - x\|_2^2 + \gamma^2,
$$
where the expectation is conditioned upon the choice of the random selections in the first $k-1$ iterations.
Then applying this recursive relation iteratively and taking full expectation, we have
\begin{align*}
\mathbb{E}\|x_k^* - x\|_2^2 &\leq \Big(1 - \frac{1}{R}\Big)^k\|x_{0} - x\|_2^2 + \sum_{j=0}^{k-1}\Big(1 - \frac{1}{R}\Big)^j\gamma^2\\
&\leq \Big(1 - \frac{1}{R}\Big)^k\|x_{0} - x\|_2^2 + R\gamma^2.
\end{align*}
By Jensen's inequality we then have
$$
\mathbb{E}\|x_k^* - x\|_2 \leq \left(\Big(1 - \frac{1}{R}\Big)^k\|x_{0} - x\|_2^2 + R\gamma^2\right)^{1/2} \leq \Big(1 - \frac{1}{R}\Big)^{k/2}\|x_{0} - x\|_2 + \sqrt{R}\gamma.
$$
This completes the proof.
\end{proof}

\subsection{Numerical Examples}

In this section we describe some of our numerical results for the randomized Kaczmarz method in the case of noisy systems.  The results displayed in this section use matrices with independent Gaussian entries.  Figure~\ref{fig1} depicts the error between the estimate by randomized Kaczmarz and the actual signal, in comparison with the predicted threshold value.  This study was conducted for 100 trials using $100 \times 50$ Gaussian matrices and independent Gaussian noise.  The systems were homogeneous, meaning $x=0$ and $b=0$.  The thick line is a plot of the threshold value, $\gamma\sqrt{R}$ for each trial.  The thin line is a plot of the error in the estimate after 1000 iterations for the corresponding trial.  As is evident by the plots, the error is quite close to the threshold.  Of course in the Gaussian case depicted, it is not surprising that often the error is below the threshold.  As discussed above, the threshold is sharp for certain kinds of matrices and noise vectors.

\begin{figure}[ht] 
  \includegraphics[scale=0.7]{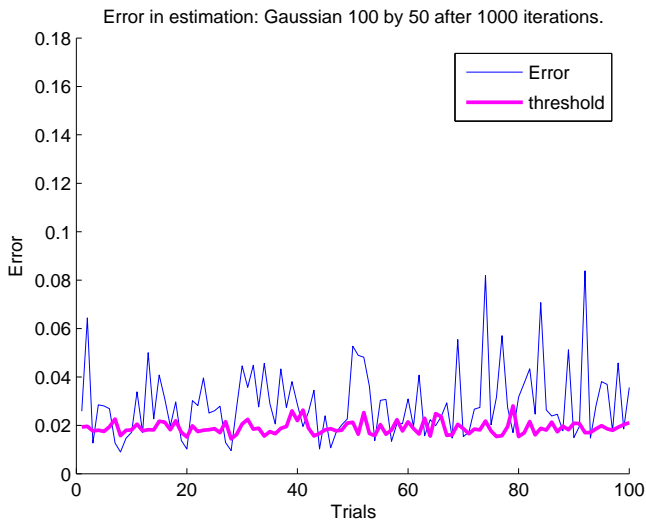}
  \caption{The comparison between the actual error in the randomized Kaczmarz estimate (thin line) and the predicted threshold (thick line).}\label{fig1}
\end{figure}

Our next study investigated the convergence rate for the randomized Kaczmarz method with noise for homogeneous systems.  Again we let our matrices $A$ be $100 \times 50$ Gaussian matrices, and our error vector be independent Gaussian noise.  Figure~\ref{fig2} displays a scatter plot of the results of this study over various trials.  It is not surprising that the convergence appears exponential as predicted by the theorems. 

\begin{figure}[ht] 
  \includegraphics[scale=0.7]{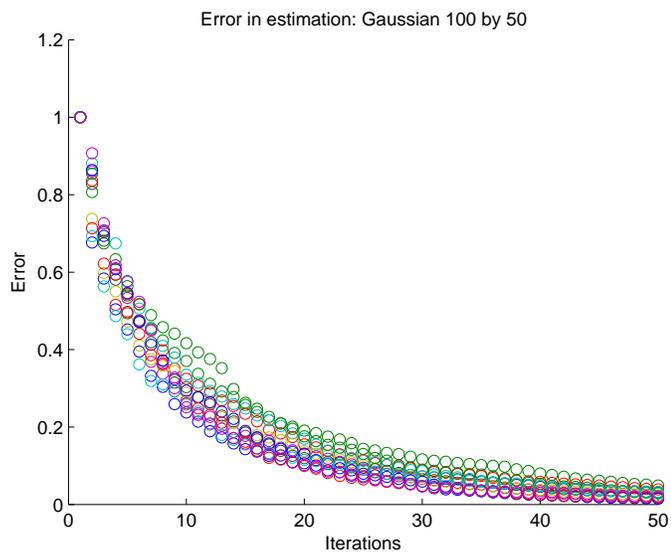}
  \caption{Convergence rate for randomized Kaczmarz over various trials.}\label{fig2}
\end{figure}

%% file: summary.tex
Compressed sensing is a new and fast growing field of applied mathematics that addresses the shortcomings of conventional signal compression. Given a signal with few nonzero coordinates relative to its dimension, compressed sensing seeks to reconstruct the signal from few nonadaptive linear measurements. As work in this area developed, two major approaches to the problem emerged, each with its own set of advantages and disadvantages. The first approach, L1-Minimization~\cite{CT05:Decoding,CRT06:Stable}, provided strong results, but lacked the speed of the second, the greedy approach. The greedy approach, while providing a fast runtime, lacked stability and uniform guarantees.  This gap between the approaches has led researchers to seek an algorithm that could provide the benefits of both.  In collaboration, my adviser Roman Vershynin and I have bridged this gap and provided a breakthrough algorithm, called Regularized Orthogonal Matching Pursuit (ROMP)~\cite{NV07:Uniform-Uncertainty,NV07:ROMP-Stable}. ROMP is the first algorithm to provide the stability and uniform guarantees similar to those of L1-Minimization, while providing speed as a greedy approach. After analyzing these results, my colleague Joel Tropp and I developed the algorithm Compressive Sampling Matching Pursuit (CoSaMP), which improved upon the guarantees of ROMP~\cite{NT08:Cosamp,DT08:CoSaMP-TR}. CoSaMP is the first algorithm to have provably optimal guarantees in every important aspect.

It was the negative opinions about conventional signal compression that spurred the development of compressed sensing. The traditional methodology is a costly process, which acquires the entire signal, compresses it, and then throws most of the information away. However, new ideas in the field combine signal acquisition and compression, significantly improving the overall cost.  The problem is formulated as the realization of such a signal $x$ from few linear measurements, of the form $\Phi x$ where $\Phi$ is a (usually random) measurement matrix. Since Linear Algebra clearly shows that recovery in this fashion is not possible, the domain from which the signal is reconstructed must be restricted. Thus the domain that is considered is the domain of all \textit{sparse} vectors. In particular, we call a signal \textit{$s$-sparse} when it has $s$ or less nonzero coordinates.
It is now well known that many signals in practice are sparse in this sense or with respect to a different basis.

As discussed, there are several critical properties that an ideal algorithm in compressed sensing should possess.  The algorithm clearly needs to be efficient in practice.  This means it should have a fast runtime and low memory requirements.  It should also provide uniform guarantees so that one measurement matrix works for all signals with high probability.  Lastly, the algorithm needs to provide stability under noise in order to be of any use in practice.  The second approach uses greedy algorithms and is thus quite fast both in theory and in practice, but lacks both stability and uniform guarantees. The first approach relies on a condition called the Restricted Isometry Property (RIP), which had never been usefully applied to greedy algorithms. For a measurement matrix $\Phi$, we say that $\Phi$ satisfies the RIP with parameters $(s,\e)$ if 
$$
(1-\e)\|v\|_2 \leq \|\Phi v\|_2 \leq (1+\e)\|v\|_2
$$
holds for all $s$-sparse vectors $v$.
We analyzed this property in a unique way and found consequences that could be used in a greedy fashion. Our breakthrough algorithm ROMP is the first to provide all these benefits (stability, uniform guarantees, and speed), while CoSaMP improves upon these significant results and provides completely optimal runtime bounds.

One of the basic greedy algorithms which inspired our work is Orthogonal Matching Pursuit (OMP), which was analyzed by Gilbert and Tropp~\cite{TG07:Signal-Recovery}. OMP uses the observation vector $u = \Phi^* \Phi x$ to iteratively calculate the support of the signal $x$ (which can then be used to reconstruct $x$). At each iteration, OMP selects the largest component of the observation vector $u$ to be in the support, and then subtracts off its contribution. Although OMP is very fast, it does not provide uniform guarantees. Indeed, the algorithm correctly reconstructs a \textit{fixed} signal $x$ with high probability, rather than \textit{all} signals. It is also unknown whether OMP is stable.

Our new algorithm ROMP is similar in spirit to OMP, in that it uses the observation vector $u$ to calculate the support of the signal $x$. One of the key differences in the algorithm is that ROMP selects many coordinates of $u$ at each iteration. The regularization step of ROMP guarantees that each of the selected coordinates have close to an equal share of the information about the signal $x$. This allows us to translate captured energy of the signal into captured support of the signal. We show that when the measurement matrix $\Phi$ satisfies the Restricted Isometry Property with parameters $(2s, c/\log s)$, ROMP exactly reconstructs any $s$-sparse signal in just $s$ iterations. Our stability results show that for an \textit{arbitrary} signal $x$ with noisy measurements $\Phi x +e$, ROMP provides an approximation $\hat{x}$ to $x$ that satisfies
\begin{equation}\label{vboundsum}
      \|\hat{x} - x\|_2 
      \leq C \sqrt{\log s} \Big( \|e\|_2 + \frac{\|x-x_s\|_1}{\sqrt{s}} \Big),
    \end{equation}
    where $x_s$ is the vector of the $s$ largest coordinates of $x$ in magnitude. 
 ROMP is thus the first greedy algorithm providing the strong guarantees of the L1-Minimization approach, and bridges the gap between the two approaches.
 
After the breakthrough of ROMP, Needell and Tropp developed a new algorithm, Compressive Sampling Matching Pursuit (CoSaMP)~\cite{NT08:Cosamp,DT08:CoSaMP-TR}. CoSaMP maintains an estimation of the support as well as an estimation of the signal throughout each iteration. It again is a greedy algorithm that provides the uniform and stability guarantees of the L1 approach. CoSaMP improves upon the stability bounds of ROMP as well as the number of measurements required by the algorithm. Indeed we show that for any measurement matrix satisfying the Restricted Isometry Property with parameters (2s, c),
CoSaMP approximately reconstructs any arbitrary signal $x$ from its noisy measurements $u = \Phi x + e$ in at most $6s$ iterations: 
  $$
    \|\hat{x} - x\|_2 
    \leq C\Big( \|e\|_2 + \frac{\|x-x_s\|_1}{\sqrt{s}} \Big).
  $$ 
  We also provide a rigorous analysis of the runtime, which describes how exactly the algorithm should be implemented in practice. CoSaMP thus provides optimal guarantees at every important aspect.

%% file: BPcode.tex
\singlespacing
\begin{verbatim}
%NAME: Basis Pursuit Tester
%PURPOSE: Tests sparse, noisy, compressible signals on Basis Pursuit
%AUTHOR: Deanna Needell
%OUTSIDE FUNCTIONS: l1eq_pd (L1-Magic, J. Romberg)

clear all
warning off all

%Variables
N=[10:5:250];
d=[256];
n=[2:2:90]; 
p=[0.2:0.2:1]; %Used for compressible signals

%Number of Trials for each parameter set
numTrials = 500;

%Counters
numN = size(N, 2);
numd = size(d, 2);
numn = size(n, 2);
nump = size(p, 2);

%Data Collection
numCorr = zeros(numN, numd, numn);
minMeas = zeros(numN);
AvgerrorNormd = zeros(numN, numd, numn, nump);
Avgerror = zeros(numN, numd, numn, nump);
%for ip:=1:nump %USED FOR COMPRESSIBLE SIGNALS
  for id = 1:numd
    for iN=1:numN 
     in=1;
     done=0;
     while in <= numn && ~done
       for trial = 1:numTrials
         tN = N(1, iN);
         td = d(1, id);
         tn = n(1, in);
         tp = p(1, ip);

         %Set Matrix
         Phi = randn(tN, td);
         Phi = sign(Phi);
         I = zeros(1,1);

         %Set signal
         z = randperm(td);
         v = zeros(td, 1);
         for t = 1:tn
             v(z(t))=1;
         end

         %USED IN THE CASE OF COMPRESSIBLE SIGNALS
         %y = sign(randn(td, 1));
         %noiErr=0;
         %for i=1:tn

             %v(z(i)) = i^(-1/tp)*y(i);
             %if i > tn
             %    noiErr = noiErr + abs(v(z(i)));
             %end
         %end

         %Set measurement and residual
         x = Phi * v;

         %USED IN THE CASE OF NOISE
         %e = randn(tN, 1);
         %x = x + e / norm(e, 2) / 2;

         %Set initial estimate
         x0 = Phi'*x;

         %Run Basis Pursuit (via L1-Magic)
         xp = l1eq_pd(x0, Phi, [], x, 1e-6);

         %Collect Data
         if norm(xp-v, 2) < 0.01
              numCorr(iN, id, in) = numCorr(iN, id, in) +1;
         end  
         AvgerrorNormd(iN, id, in, ip) = (AvgerrorNormd(iN, id, in, ip) * 
                   (counted-1) + (norm(xp-v,2)/noiErr*(tn)^0.5))/counted;
         Avgerror(iN, id, in, ip) = (Avgerror(iN, id, in, ip) * 
                   (counted-1) + norm(xp-v, 2))/counted;
         end % end Trial

       if numCorr(iN, id, in) / numTrials > 0.98
          minMeas(iN) = tn;
       else
          done=1;
       end
       in = in +1;
      end % n 
    end % N
  end % d
%end % p

\end{verbatim}

%% file: OMPcode.tex
\singlespacing
\begin{verbatim}
%NAME: Orthogonal Matching Pursuit Tester
%PURPOSE: Tests sparse signals on Orthogonal Matching Pursuit
%AUTHOR: Deanna Needell
%OUTSIDE FUNCTIONS: None

clear all
warning off all
%Variables
N=[10:5:250];
d=[256];
n=[2:2:90];  

numTrials = 500;
numN = size(N, 2);
numd = size(d, 2);
numn = size(n, 2);

%Data Collection
numCorr = zeros(numN, numd, numn);
mostSpars = zeros(numN);

for id = 1:numd
 for iN=1:numN 
   in=1;
   done=1;
   keepGo=1;
   while in <= numn && keepGo
     for trial = 1:numTrials
       tN = N(1, iN);
       td = d(1, id);
       tn = n(1, in);

       %Set Matrix
       Phi = randn(tN, td);
       Phi = sign(Phi);
       I = zeros(1,1);

       %Set signal
       z = randperm(td);
       supp=z(1:tn);
       v = zeros(td, 1);
       for t = 1:tn
           v(z(t))=1;
       end
       x = Phi * v;
       r = x;

       %Run OMP

       while length(I)-1 < tn
         u = Phi' * r;
         absu = abs(u);
         [b, ix] = sort(absu, 'descend');
         bestInd = ix(1);
         bestVal = b(1);

         %Update I
         I(length(I)+1) = bestInd;

         %Update the residual
         PhiSubI = Phi(:, I(2));
         for c=3:length(I)
             if ~isMember(I(2:c-1),I(c))
               PhiSubI(:,c-1) = Phi(:, I(c));
             end
         end
         y = lscov(PhiSubI, x);
         r = x - PhiSubI * y;
       end

	     if isMember(supp, I);
	          numCorr(iN, id, in) = numCorr(iN, id, in) +1;
	     end               
     end % end Trial
     
     if numCorr(iN, id, in) / numTrials > 0.98 && done
         mostSpars(iN) = tn;
     else
         done=0;
     end
     if numCorr(iN, id, in) <= 0.01
         keepGo=0;
     end
     in = in +1;
   end % n
 end % N
end % d

\end{verbatim}

%% file: ROMPcode.tex
\singlespacing
\begin{verbatim}
function [vOut, numIts] = romp(n, Phi, x)
% [vOut] = romp(n, Phi, x)
%%% PARAMETERS %%%%%%%%%%%%%%%%%%%%%%%%%%%
% d = ambient dimension of the signal v
% N = number of measurements
% n = sparsity level of n
% Phi = N by d measurement matrix
% x = measurement vector (Phi * v)
% vOut = reconstructed signal
%%%%%%%%%%%%%%%%%%%%%%%%%%%%%%%%%%%%%%%%%%

%%% FUNCTION DESCRIPTION %%%%%%%%%%%%%%%%%
% romp takes parameters as described
% above. Given the sparsity level n and
% the N by d measurement matrix Phi, and
% the measurement vector x = Phi * v, romp
% reconstructs the original signal v.
% This reconstruction is the output.
%%%%%%%%%%%%%%%%%%%%%%%%%%%%%%%%%%%%%%%%%%

clear r I J J0 u b ix numIts Jvals 
warning off all

N = size(Phi, 1);
d = size(Phi, 2);

% Set residual
r = x;

%Set index set to "empty"
I = zeros(1,1);

%Counter (to be used optionally)
numIts = 0;

%Run ROMP
while length(I)-1 < 2*n && norm(r) > 10^(-6)

   numIts = numIts + 1;

   %Find J, the biggest n coordinates of u
   u = Phi' * r;
   absu = abs(u);
   [b, ix] = sort(absu, 'descend');
   J = ix(1:n);
   Jvals = b(1:n);

   %Find J0, the set of comparable coordinates with maximal energy
   w=1;
   best = -1;
   J0 = zeros(1);
   while w <= n
       first = Jvals(w);
       firstw = w;
       energy = 0;
       while ( w <= n ) && ( Jvals(w) >= 1/2 * first )
           energy = energy + Jvals(w)^2;
           w = w+1;
       end
       if energy > best
           best = energy;
           J0 = J(firstw:w-1);
       end
   end

   %Add J0 to the index set I
   I(length(I)+1: length(I)+length(J0)) = J0;

   %Update the residual
   PhiSubI = Phi(:, I(2));
   for c=3:length(I)
       if ~isMember(I(2:c-1),I(c))
         PhiSubI(:,c-1) = Phi(:, I(c));
       end
   end
   y = lscov(PhiSubI, x);
   r = x - PhiSubI * y;

end % end Run IRA  

vSmall = PhiSubI \ x;
vOut = zeros(d, 1);
for c=2:length(I)
    vOut(I(c)) = vSmall(c-1);
end
   
\end{verbatim}            

%% file: Cosampcode.tex
\singlespacing
\begin{verbatim}

%NAME: CoSaMP Tester
%PURPOSE: Tests sparse signals on CoSaMP
%AUTHOR: Deanna Needell
%OUTSIDE FUNCTIONS: None

%Testing Parameters
sVals=[1:1:55]; % Sparsity levels
mVals=[5:5:250]; %Measurement levels
dVals=[256]; %dimension

numTrials=500; %Number of trials per parameter set

%Set Variable lengths and Data Collection
nums=length(sVals);
numm=length(mVals);
numd=length(dVals);
numCorrect = zeros(nums, numm, numd);
trend99 = zeros(numm, 1);

for is=1:nums
    for im=1:numm
        for id=1:numd
            %Set Parameters
            s = sVals(is);
            m = mVals(im);
            d = dVals(id);
            
            %Start a trial
            for trial=1:numTrials
            
                %Generate Measurement matrix
                Phi = randn(m,d);
                
                %Generate sparse signal
                z = randperm(d);
                x = zeros(d, 1);
                x(z(1:s)) = sign(randn(s,1));
                
                %Generate measurements
                u = Phi*x;
                
                %Begin CoSaMP
                
                %Initialize
                a = zeros(d,1);
                v = u;
                it=0;
                stop = 0;
                while ~stop
                    
                    %Signal Proxy
                    y = Phi'*v;
                    [tmp, ix] = sort(abs(y), 'descend');
                    Omega = ix(1:2*s);
                    [tmp, ix] = sort(abs(a), 'descend');
                    T = union(Omega, ix(1:s));
                    
                    %Signal Estimation
                    b = zeros(d, 1);
                    b(T) = Phi(:, T) \ u;
                    
                    %Prune
                    [tmp, ix] = sort(abs(b), 'descend');
                    a = zeros(d, 1);
                    a(ix(1:s), 1) = b(ix(1:s), 1);
                    
                    %Sample Update
                    v = u - Phi*a;
                    
                    %Iteration counter
                    it = it + 1;
                    
                    %Check Halting Condition
                    if norm(a-x) <= 10^(-4) || it > max(8*s, 60)
                        stop = 1;
                    end
                    
                end %End CoSaMP iteration
                
                %Collect Data
                if norm(a-x) <= 10^(-4)
                    numCorrect(is, im, id) = numCorrect(is, im, id) + 1;
                end
                
            end % End trial
        end %d
        
        if trend99(im) == 0 && numCorrect(is, im, id) >= 0.99*numTrials
            trend99(im) = s;
        end
    end %m
end %s
\end{verbatim} 

%% file: RWL1code.tex
\singlespacing
\begin{verbatim}
%NAME: Reweighted L1-Minimization Tester
%PURPOSE: Tests sparse and noisy signals on Reweighted L1
%AUTHOR: Deanna Needell
%OUTSIDE FUNCTIONS: CVX package (Michael Grant and Stephen Boyd)

N = 256; %dimension
M = 128; %measurements
kVals = [30]; %sparsity
eepsVals = [1];
numTrials = 500;
maxIter = 9;

errorVecDecoding = zeros(length(kVals),length(eepsVals),maxIter,numTrials);
errors = zeros(numTrials,1);
for trial = 1:numTrials
   for kIndex = 1:length(kVals)
       K = kVals(kIndex);
       for eIndex = 1:length(eepsVals)
           eeps = eepsVals(eIndex);
           
           % Gaussian spikes in random locations
           x = zeros(N,1); q = randperm(N);
           x(q(1:K)) = sign(randn(K,1));
           
           % measurement matrix
           Phi = sign(randn(M,N))/sqrt(M);
           
           % observations
           err = randn(M, 1);
           sigma = 0.2*norm(Phi*x,2)/norm(err,2);
           err = sigma*err;
           y = Phi*x + err;
           errors(trial) = norm(err,2);
           
           for iter = 1:maxIter
               %Set the weights
               if iter > 1
                   
                   weights = 1./(abs(xDecoding)+eeps/(1000*iter));
               else
                   weights = 1*ones(N,1);
               end
      
               %Set noise tolerance parameter
               delta=sqrt(sigma^2*(M+2*sqrt(2*M)));
               
               %Use CVX to perform minimization
               cvx_begin
               cvx_quiet(true)
                   variable xa(N);
                   minimize( norm(diag(weights)*xa,2) );
                   subject to
                        norm(Phi*xa - y, 2) <= delta;
               cvx_end
               
               %Collect results
               xDecoding = xa;
               errorVecDecoding(kIndex,eIndex,iter,trial) = norm((x-xDecoding),2);
               
           end
       end
   end
   
end

\end{verbatim}

%% file: RKcode.tex
\singlespacing
\begin{verbatim}
%NAME: Randomized Kaczmarz Tester
%PURPOSE: Tests RK on noisy systems
%AUTHOR: Deanna Needell
%OUTSIDE FUNCTIONS: none

clear all
warning off all

m = 100;  %rows
n=100;     %columns

numIts = 1000;
numTrials = 100;

A = zeros(numTrials, m, n);
e = zeros(numTrials, m);
x = zeros(numTrials, n);
b = zeros(numTrials, m);
est = zeros(numTrials, numIts, n);  %estimations
initErr = zeros(numTrials);
R = zeros(numTrials, 1); %Value of R (as in paper)
mu = zeros(numTrials); %Coherence
gamma = zeros(numTrials, 1); %worst error to row norm ratio
beta = zeros(numTrials); %beta is in theorem
errorsRK = zeros(numTrials, numIts); 
errorsCG = zeros(numTrials, numIts);

for trial=1:numTrials
    if mod(trial, 1) == 0
        display(['Trial ', num2str(trial)]);
    end
    %Set matrix equation
    A(trial, :, :) = sign(randn(m,n));
    x = zeros(n, 1)';
    b(trial,:) = reshape(A(trial, :, :), m, n)*(x');
    
    %Set initial guess, ||x - x0|| = 1
    est(trial, 1, :) = randn(1, n);
    est(trial, 1, :) = est(trial, 1, :) / norm(reshape(est(trial, 1, :), 1, n),2) ;
    origest = reshape(est(trial, 1, :), 1, n)';
    
    %Add error to RHS of Ax=0
    e(trial, :) = randn(1, m)*2;
    e(trial, :) = e(trial, :) / norm(e(trial, :), 2) / 10;
    b(trial, :) = b(trial, :)+e(trial, :);  %%%%NOISY!!
    
    %Calculate stats
    initErr = norm(reshape(est(trial, 1, :), 1, n) - x,2);
    fronorm = norm(reshape(A(trial, :, :), m, n), 'fro');
    R(trial) = norm(pinv(reshape(A(trial, :, :), m, n)),2)*fronorm;
    temp = zeros(1, m);
    for i=1:m
        temp(i) = norm(reshape(A(trial, i, :), 1, n), 2);
    end
    gamma(trial) = max(abs( e(trial, :)./temp )); 
    errorsRK(trial, 1) = norm(reshape(est(trial, 1, :), 1, n)-x,2);
    errorsCG(trial, 1) = norm(reshape(est(trial, 1, :), 1, n)-x,2);

    %Run RK
    for it=2:numIts
        %Select random hyperplane
        pick = rand * fronorm^2;
        counter = 0;
        index = 1;
        while counter + norm(reshape(A(trial, index, :),1, n), 2)^2 < pick
            counter = counter + norm(reshape(A(trial, index, :),1, n), 2)^2;
            index = index + 1;
        end
        
        %Modify estimation
        est(trial, it, :) = est(trial, it-1, :) + (b(index) - dot((A(trial, index, :)),
        (est(trial, it-1, :))) )/ (norm(reshape(A(trial, index, :), 1, n),2)^2) * A(trial, index, :); 
        errorsRK(trial, it) = norm(reshape(est(trial, it, :), 1, n)-x,2);
    end
    
    %Run CG
    for it=1:numIts
        [estcg,flag] = cgs(reshape(A(trial, :, :),m, n),b(trial,:)',10^(-10), it);
        errorsCG(trial, it) = norm(estcg-x',2);
    end

end

\end{verbatim}